\documentclass[11pt,a4paper, reqno]{amsart}
\usepackage[english]{babel}
\usepackage[margin=2.5cm]{geometry}
\usepackage{xcolor}
\usepackage{amsmath}
\usepackage{amsthm}
\usepackage{amssymb}
\usepackage{graphicx}
\usepackage{hyperref}
\usepackage{hyperref}
\usepackage{orcidlink}
\usepackage[title]{appendix}
\usepackage{mathrsfs}
\usepackage{amsfonts}
\usepackage{booktabs} 
\usepackage{caption}  
\usepackage{threeparttable} 
\usepackage{algorithm}
\usepackage{algorithmicx}
\usepackage{algpseudocode}
\usepackage{listings}
\usepackage{enumitem}
\usepackage{chngcntr}
\usepackage{booktabs}
\usepackage{lipsum}
\usepackage{subcaption}
\usepackage[T1]{fontenc}    
\usepackage{csquotes}      
\usepackage{diagbox}
\usepackage{breqn}
\usepackage{bm}
\usepackage{tikz}
\usepackage{float}
\usepackage{bbm}

\usepackage{array}

\newcolumntype{C}[1]{>{\centering\let\newline\\\arraybackslash\hspace{0pt}}m{#1}}

\usepackage{enumitem}

\newcommand{\QQ}{\mathbb{Q}}
\newcommand{\ZZ}{\mathbb{Z}}
\newcommand{\NN}{\mathbb{N}}
\newcommand{\FF}{\mathbb{F}}

\newcommand{\RR}{\mathbb{R}}

\DeclareMathOperator{\inv}{inv}
\DeclareMathOperator{\Gal}{Gal}

\usepackage{cleveref}
\newcommand{\legendre}[2]{\genfrac{(}{)}{}{}{#1}{#2}}

\usepackage{setspace}
\newcommand\numberthis{\addtocounter{equation}{1}\tag{\theequation}}

\usepackage[style=alphabetic, backend=biber, maxnames=10]{biblatex}
\renewbibmacro{in:}{}

\usepackage{float}   

\usepackage{caption} 



\usepackage{aliascnt}
\usepackage{cleveref}

\newtheorem{lemma}{Lemma}
\numberwithin{lemma}{section}

\newaliascnt{conjecture}{lemma}
\newtheorem{conjecture}[conjecture]{Conjecture}
\aliascntresetthe{conjecture}

\newaliascnt{theorem}{lemma}
\newtheorem{theorem}[theorem]{Theorem}
\aliascntresetthe{theorem}

\newaliascnt{heuristic}{lemma}
\newtheorem{heuristic}[heuristic]{Heuristic}
\aliascntresetthe{heuristic}

\newaliascnt{proposition}{lemma}
\newtheorem{proposition}[proposition]{Proposition}
\aliascntresetthe{proposition}

\newaliascnt{corollary}{lemma}
\newtheorem{corollary}[corollary]{Corollary}
\aliascntresetthe{corollary}

\theoremstyle{definition}

\newaliascnt{example}{lemma}

\aliascntresetthe{example}

\newaliascnt{definition}{lemma}
\newtheorem{definition}[definition]{Definition}
\aliascntresetthe{definition}

\newaliascnt{remark}{lemma}
\newtheorem{remark}[remark]{Remark}
\aliascntresetthe{remark}

\numberwithin{equation}{section}

\begin{document}
\title{Counting quaternion extensions by multi-heights}

\author{Willem Hansen}
 \address{Graz University of Technology, Institute of Analysis and Number Theory, Steyrergasse 30, 8010 Graz, Austria}
 \email{hansen@math.tugraz.at} 

\date{}

\begin{abstract}
    We prove a case of a conjecture of Gundlach on the number of Galois extensions over $\QQ$ with bounded multi-invariants for the group $G=Q_8$. We also obtain results when imposing certain local conditions at the wildly ramified place. As an application, we prove asymptotics for the number of $Q_8$-extensions with bounded product of ramified primes.
\end{abstract}
\maketitle

\section{Introduction}

In \cite{malle2002}, \cite{malle2004} Malle put forward a conjecture concerning the distribution of field extensions of a number field with a given Galois group and bounded discriminant. Let $G$ be a finite transitive permutation group on $n$ points and let $K$ be a number field. Fix an algebraic closure $\overline{K}$ of $K$. By abuse of notation, write $\Gal(M/K)=G$ if $M/K$ is a field extension inside $\overline{K}$ such that the Galois group of the Galois closure $\tilde{M}/K$, viewed as a permutation group on the set of embeddings of $M$ into $\tilde{M}$, is permutation isomorphic to $G$.

\begin{conjecture}[\texorpdfstring{\cite[Conjecture 1.1]{malle2004}}{}]
    Let $G\subset S_n$ be a non-trivial transitive permutation group and let $K$ be a number field. Then there exists a
constant $c_{K,G} > 0$ such that
\[
\#\{M/K: \Gal(M/K) = G , [M:K] = n,  |N_{K/\QQ} (\operatorname{Disc}({M/K}))| \leq X\} \sim c_{K,G} X^{a_G} (\log X)^{b_{K,G}-1}
\]
where $a_G>0$ and $b_{K,G}\geq 1$ are explicit constants.
\end{conjecture}
This conjecture is known to hold in various cases. For instance, the asymptotic for all abelian groups is known to hold from work of Wright \cite{wright_abelian} using class field theory. Many cases for other small non-abelian groups are known. In \cite{Koymans-Pagano-nilpotent} Koymans and Pagano show Malle's conjecture for a wide range of nilpotent groups, strengthening a previous result of Klüners and Malle \cite{klueners_malle}. The work of Koymans-Pagano \cite{Koymans-Pagano-nilpotent} includes an asymptotic for counting $Q_8$-extensions by discriminant. Koymans and Pagano also provide a useful parametrization of nilpotent extensions that we will make use of. 

Malle's conjecture has also been studied when counting field extensions by invariants other than $N_{K/\QQ} (\operatorname{Disc}{M/K})$, for example Artin conductors or the product of the norms of ramified primes.

Recently, Gundlach \cite{gundlach2022mallesconjecturemultipleinvariants} revisited Malle's conjecture by introducing a family of multi-invariants $\inv_i$ by which to order \'etale $G$-extensions. Here, a $G$-extension $M/K$ is to be understood as a continuous homomorphism from the absolute Galois group $\operatorname{Gal}(\overline{K}/ K)$ to $G$.  Gundlach introduces the following heuristic.
\begin{heuristic}[ \cite{gundlach2022mallesconjecturemultipleinvariants} Heuristic 1.2] \label{heuristic_all_inv}
    There is a constant $C>0$ such that if the $X_i \to \infty$ for all $i=1,\dots, m$, then 
    \[
        \sum_{\substack{\text{\'etale } G\text{-extensions }  M/K \\ \inv_i( M)\leq X_i \ \forall i}} \frac{1}{\# \operatorname{Aut}(M)} \sim C \cdot\prod_i X_i.
    \]
\end{heuristic}
Over the base field $\QQ$ this question also appears earlier in \cite[Question 4.5]{ellenberg_venkatesh}. This heuristic justifies the following conjecture.
\begin{conjecture}[\cite{gundlach2022mallesconjecturemultipleinvariants} Conjecture 1.5] \label{conj_gund}
     Fix some number $0<\delta<1$. There is a constant $C>0$ such that if the $X_i \to \infty$ for all $i=1, \dots m$, then $$\#\{ M/K: \Gal(M/K)\simeq G, \delta X_i< \inv_i(M)\leq X_i \ \forall i  \}\sim C\cdot \prod_i X_i.$$
\end{conjecture}
Here the lower bounds $\delta X_i$ on the invariants are introduced in order to avoid some known counterexamples that disprove Malle's conjecture and which would disprove the conjecture in this setting as well (see \cite{klueners2004counterexamplemallesconjecture}). The known counterexamples all have some bounded invariant, so that the proposed conjecture with lower bounds may still hold. In the above one may restrict from \'etale algebras to field extensions, as non-field extensions have some invariant bounded.

Note that for any finite group $G$, \Cref{heuristic_all_inv} implies \Cref{conj_gund}. Gundlach proves \Cref{heuristic_all_inv} for all abelian groups. In \cite{shankar2024asymptoticscubicfieldsordered}, Shankar and Thorne prove a version of \Cref{heuristic_all_inv} for $G=S_3$. Here, there are two invariants, and Shankar and Thorne prove asymptotics even when keeping one of the invariants fixed. In \cite{hansen-zanoli}, Zanoli and the author confirm, under the additional assumption that the  invariants have comparable rates of growth, a version of \Cref{heuristic_all_inv} for the group $D_4$. 

In this paper, we prove \Cref{conj_gund} for $Q_8$-extensions in full. We develop a more careful analysis of the resulting character sums, which allows us to obtain an additional power saving without imposing any additional condition on the sizes of the parameters.   

We fix some notation. Let  $\overline{\QQ}$  be an algebraic closure of $\QQ$. Let $\Gamma_\QQ := \operatorname{Gal}(\overline{\QQ}/\QQ)$. Let $\operatorname{Epi}_{\operatorname{cts}}(\Gamma_\QQ, G),$ $\operatorname{Hom}_{\operatorname{cts}}(\Gamma_\QQ, G)$ denote the continuous epimorphisms,  homomorphism, respectively.   \\
We denote by a $Q_8$-extension $\varphi$ a continuous, epimorphisms  $\varphi \in \operatorname{Epi}_{\operatorname{cts}}(\Gamma_\QQ, Q_8)$. 
Here 
\[
Q_8 = \langle i,j,k,-1 \ |\ (-1)^2 = 1,  i^2=j^2=k^2=ijk = -1\rangle
\]
is the quaternion group of $8$ elements. Such an epimorphisms gives rise to a $Q_8$-Galois field $M/\QQ$ inside $\overline{\QQ}$ together with an isomorphism $\sigma:\operatorname{\Gal}(M/\QQ) \overset{\simeq}{\to} Q_8$.

We give a precise definition of Gundlach's multiple invariants in \Cref{def:multi}. For now, we note that we obtain a height function for  each non-trivial cyclic subgroup (up to conjugation). For $Q_8$ we have the non-trivial cyclic subgroups $\langle i\rangle, \langle j\rangle, \langle k\rangle,\langle -1\rangle$. Thus, we have four height functions that we denote 
\[
h_{i}, h_{j}, h_k, h_{-1} : \operatorname{Epi}_{\operatorname{cts}}(\Gamma_\QQ, Q_8) \to \RR_{>0}.
\]

\begin{theorem} \label{thm1}
     Let $h_{-1}, h_i, h_j, h_k:\operatorname{Epi}_{\operatorname{cts}}(\Gamma_\QQ, Q_8) \to \mathbb{R}_{>0}$ be the four associated height functions given by \Cref{def:multi}. Let $X_{-1}, X_i, X_j, X_k \in \mathbb{R}_{>0}$. Then 
    \begin{align*}
    &\#\{ \varphi \in \operatorname{Epi}_{\operatorname{cts}}(\Gamma_\QQ, Q_8) \ | \ h_x(\varphi)\leq X_x \text{ for }x\in \{-1,i,j,k\}\}= \\
    &\hspace{1cm}= \frac{7}{8} \prod_{p>2} \left(1-\frac{1}{p}\right)^4 \left(1+\frac{4}{p}\right) X_iX_jX_kX_{-1}+ E\,,
    \end{align*}
    with an error term 
    \[
    E \ll X_{-1}X_iX_jX_k \min\{ X_{-1},X_i,X_j,X_k\}^{-\Delta}
    \]
    for some $\Delta>0$.
\end{theorem}

The multi-heights of an extension is a product of tamely ramified primes with specific ramification type. For $G=Q_8$, the only prime that can wildly ramify is $p=2$. We also obtain a counting result where we impose $2$-adic data. Fix an algebraic closure $\overline{\QQ_2}$ of $\QQ_2$, and fix an embedding 
\[
\overline{\QQ} \hookrightarrow \overline{\QQ_2}.
\]
This induces an inclusion $\Gamma_{\QQ_2} = \operatorname{Gal}(\overline{\QQ_2}/\QQ_2) \hookrightarrow \Gamma_\QQ $, as well as a restriction map $\operatorname{Epi}_{\text{cts}}(\Gamma_\QQ , Q_8) \to \operatorname{Hom}_{\text{cts}}(\Gamma_{\QQ_2}, Q_8)$ of continuous homomorphisms. The restriction of an epimorphism need not be surjective in general.

\begin{theorem}\label{thm2_local}
Let $\Phi : \Gamma_{\QQ_2} \to Q_8$ be a continuous homomorphism. Let $X_i, X_j, X_k, X_{-1}\in \RR_{>0}$. Then 
\begin{align*}
   & \# \{\varphi \in \operatorname{Epi}_{\operatorname{cts}}(\Gamma_\QQ ,  Q_8) : h_{x}(\varphi) \leq X_x \text{ for }x\in\{i,j,k,-1\} , \varphi_{|_{\Gamma_{\QQ_2}}} = \Phi \} \\
   &\hspace{2cm }= \frac{1}{256} \prod_{p>2}\left(1-\frac{1}{p}\right)^4 \left(1+\frac{4}{p} \right) X_iX_jX_kX_{-1}+ E\,,
    \end{align*}
    with an error term 
    \[
    E \ll X_{-1}X_iX_jX_k \min\{ X_{-1},X_i,X_j,X_k\}^{-\Delta}
    \]
    for some $\Delta>0$.
\end{theorem}
\begin{remark}
    In total, there are 224 (continuous) homomorphisms $\Phi:\Gamma_{\QQ_2}\to Q_8$, which recovers \Cref{thm1}.
\end{remark}

To show the strength of these theorems, we provide an application to counting $Q_8$-extensions by the product of ramified primes.  \Cref{thm1} allows us to count in arbitrary skew boxes and obtain a power saving in the smallest side length. A version of a generalized hyperbola method, developed by Pieropan--Schindler \cite{pieropan-schindler}, allows us to establish a result for counting $Q_8$-extensions by the product of tamely ramified primes.
 For a $Q_8$-extension $\varphi \in \operatorname{Epi}_{\operatorname{cts}}(\Gamma_\QQ, Q_8)$ denote its fixed field by $M$, we define 
 \[
 H (M) :=  \prod_{p \text{ ramified in }M}p,
 \]
 the product of ramified primes.  
\begin{corollary} \label{corollary1}
    Let $X\in \RR_{>0}$.  
    \begin{align*}
    &\#\{(M, \sigma) \ | \ \sigma:  \operatorname{Gal}(M/\QQ)\overset{\simeq}{\to} Q_8: H(M)\leq X\} \\
    &\qquad \qquad =  \frac{29}{384}  \prod_{p>2} \left(1-\frac{1}{p}\right)^4 \left(1+\frac{4}{p}\right) X (\log X)^3 + O(X(\log X)^2 (\log\log X)^4)\,.
    \end{align*}
\end{corollary}

In \cite{ellenberg_satriano_zureick-brown}, \cite{DardaYasuda2024} a notion of the Manin conjecture for stacks is introduced and the authors show that Malle's conjecture can be read as a special version of Manin's conjecture for stacks. In this setting, Loughran and Santens \cite{loughran_santens2025malle} make a precise prediction for the leading constant in Malle's conjecture, when ordering the field extensions by the product of ramified primes. The conjecture on the constant resembles the constant predicted by Peyre \cite{peyre95} for counting rational points in the setting of Manin's conjecture. Peyre developed a multi-height approach to counting rational points (\cite{peyre-multi-heights}). Gundlach's multi-height approach can be interpreted as an analogue to Peyre's multi-height setting. 

 \begin{remark}
     The leading constant in \Cref{thm1}, \Cref{thm2_local} agree with the predicted constant from \cite{loughran_santens2025malle, loughran2026leadingconstantmallesconjecture} adapted to a multi-height setting. Simirlarly, the leading constant in \Cref{corollary1} agrees  with the predicted constant from \cite{loughran_santens2025malle, loughran2026leadingconstantmallesconjecture}.
 \end{remark}

\subsection*{Acknowledgments} I would like to greatly thank my advisor Christopher Frei for his guidance on this project and many helpful discussions. I am very grateful to Nick Rome for encouraging this project and his help throughout. I would like to thank Fabian Gundlach, Jürgen Klüners, Carlo Pagano, Anna Zanoli for helpful discussions. This research was funded in whole or in part by the Austrian Science Fund (FWF), project 10.55776/DOC183. 
\section{Preliminaries and notation} \label{sec:notation}
We fix a separable closure $\overline{\QQ}$ of
$\QQ$, as well as a separable closure $\overline{\QQ}_p$ of
$\QQ_p$ for every prime $p$. We fix embeddings
\[
    \overline{\QQ}\hookrightarrow \overline{\QQ}_p
\]
for every prime $p$, which induce inclusions of the absolute Galois groups
in the opposite direction
\[
    \operatorname{Gal}(\overline{\QQ}_p/\QQ_p) \hookrightarrow     \operatorname{Gal}(\overline{\QQ}/\QQ).
\]
For a field $k=\QQ $ or $k= \QQ_p$ we write $\Gamma_k=\operatorname{Gal}(\overline{k}/k)$ for the absolute Galois group of $k$.
By a $Q_8$-extension of $\QQ$ we mean a continuous surjective homomorphism $\Gamma_\QQ \to Q_8$. Equivalently, such a homomorphism corresponds to a pair $(M, \sigma)$ where $M/\QQ$ is a Galois extension with $\sigma : \operatorname{Gal}(M/\QQ)\overset{\simeq}{\to} Q_8$. 

Commonly, one counts isomorphism classes of $Q_8$-extensions, where the isomorphisms are given by conjugation of elements of $Q_8$. Let $BQ_8(\QQ)$ be the groupoid whose objects are continuous homomorphisms from $\Gamma_\QQ \to Q_8$, and whose morphisms are given by conjugation by elements in $Q_8$. Let $BQ_8[\QQ]$ be the set of isomorphism classes of objects of $BQ_8(\QQ)$. In particular the isomorphism classes of the groupoid are the orbits of the conjugacy action.

In this paper, we count surjective objects in $BQ_8(\QQ)$. In the context of Malle's conjecture, one counts surjective elements in $BQ_8[\QQ]$. For a surjective class in $\varphi \in BQ_8[\QQ]$ its automorphism group is isomorphic to the stabilizer of the conjugacy action, which is $Z(Q_8) = \{\pm 1\}$. Hence there are $\frac{|Q_8|}{|Z(Q_8)|} = 4$ objects in $BQ_{8}$ corresponding to the isomorphism class given by $\varphi$, by the orbit-stabilizer theorem.

For integers $a,b$ we use the notation $(a,b) = \gcd(a,b), [a,b] =\operatorname{lcm}(a,b)$. For a place $\nu$ of $\QQ$, $(a,b)_\nu$ is the Hilbert symbol at $\nu$. We denote by $\tau(a)$ the number of positive divisors of $a$. 
For an integer $a\neq 0$ we define the notation\[
a' := \text{ odd positive part of }a.
\]
We say that a tuple of integers is pairwise coprime if a prime divides at most one entry and at most one of them is negative. 
\section{Parametrization} \label{sec:para}
Recall the notation of the quaternion group $Q_8 = \langle i,j,k,-1 \ |\ (-1)^2 = 1,  i^2=j^2=k^2=ijk = -1\rangle$. We follow the parametrization of Koymans-Pagano \cite{Koymans-Pagano-nilpotent}; see also \cite{koymans-pagano-archive}, which contains a section of the general parametrization for the case of $2$-group extensions over the base field $\QQ$. 
We start by defining a sequence of groups with 
\[
\{0\} = G_0 , G_1 , G_2, G_3\simeq   Q_8,
\]
and $G_i = (\FF_2\times G_{i-1}, \theta_i)$ obtained as a central extension 
\[
1\to \FF_2\to G_i \to G_{i-1}\to 1
\]
represented by some cocycle $\theta_i :G_{i-1}^2\to  \FF_2$, normalized so that $\theta_i(\operatorname{id}, \operatorname{id})=0$. For $Q_8$ we have $G_1 = \FF_2$, $G_2 = \FF_2 \times \FF_2$. We fix the labels for $G_2 $
\begin{table}[H]
    \centering
    \begin{tabular}{c |c}
      & {$G_2  $}     \\
         \hline
      1&    \,$(0,0)$  \\
       $\overline{i}$ &  \, $(0,1)$  \\
         $\overline{j}$ & \, $(1,1)$  \\
        $\overline{k}$ &  \, $(1,0)$  \\
    \end{tabular}
\end{table}
\noindent
The group $G_3$ is obtained by a central extension 
\[
1 \to \FF_2 \to G_3 \to G_2 \to 1
\]
represented by a cocycle $\theta : G_2^2 \to \FF_2$. Explicitly, let 
\[
\theta( (x_1, x_2), (y_1, y_2)) = x_1y_1+ x_2y_2+x_2y_1. 
\]
With this choice, $G_3\simeq Q_8$. Indeed,
\[
(a, x)^2 = (\theta(x,x), (0,0)) = \begin{cases}
(0,0,0) \quad \text{if } x=(0,0), \\
(1,0,0) \quad \text{else}.
\end{cases}
\]
We identify  $G_3 = (\FF_2 \times \FF_2^2, \ast_\theta)$  with $Q_8$ via
\begin{table}[H]
    \centering
 \begin{tabular}{c|c}
      $Q_8$ & $G_3$  \\
         \hline
          $1$   & $(0,\,0,\, 0)$ \\
          $i$   & $(0,\,0,\,1)$ \\
          $j$   & $(0,\,1,\,1)$ \\
          $k$   & $(0,\,1,\,0)$ \\
          $-1$  & $(1,\,0,\,0)$ \\
          $-i$  & $(1,\,0,\,1)$ \\
          $-j$  & $(1,\,1,\,1)$ \\
          $-k$  & $(1,\,1,\,0)$
    \end{tabular}
\end{table}
We now specify the general parametrization of nilpotent extensions from  \cite{Koymans-Pagano-nilpotent} to the case $G=Q_8$.  Let
\[
    \chi_a:\Gamma _{\QQ}\longrightarrow \FF_2
\]
denote the quadratic character associated to a
square-free integer $a$. As
\[
    G_2 =\FF_2^2,
\]
every continuous homomorphism $\Gamma_\QQ \to G_2$ is of the form
\[
    (\chi_b,\chi_a):\Gamma_{\QQ}\longrightarrow G_2 
\]
for quadratic characters $\chi_b, \chi_a : \Gamma_\QQ\to \FF_2$ for $a,b$ square-free. This map is surjective if and only if $a, b\neq 1$ and $a\neq b$. In other words, the corresponding field $\QQ(\sqrt{a}, \sqrt{b})$ is a biquadratic field. 

Such a map can be lifted to an epimorphism $\Gamma_\QQ\to Q_8$ if and only if the corresponding embedding problem is solvable. The embedding problem is solvable if and only if the Hilbert symbol $(a,b)_\nu(ab,-1)_\nu = +1$ for every place $\nu$ of $\QQ$, see \cite[Theorem 4]{kiming}. 

Whenever this holds, the map $(\chi_b, \chi_a) : \Gamma_\QQ  \to G_2$ admits a lift to an epimorphism 
\[
(\theta_{a,b}, \chi_b, \chi_a): \Gamma_\QQ\to G_3 = Q_8.
\]
The coordinate map $\theta_{a,b}$ is constructed explicitly in \cite[Prop. 2.6]{Koymans-Pagano-nilpotent} (see also \cite[Prop. 2.2]{koymans-pagano-archive} for the case of 2-groups and $k=\QQ$) such that the resulting extension is minimally ramified. In particular, the extension is unramified outside all ramified places of the biquadratic extension $\QQ(\sqrt{a}, \sqrt{b})$. 

Every other lift is obtained by twisting $\theta_{a,b}$ with a quadratic character $\chi_t$ for some square-free $t$. So every lift is of the form $  (\theta_{a,b} + \chi_t, \chi_b, \chi_a) : \Gamma_\QQ \to Q_8$.

Thus every $Q_8$-extension is parametrized by a triple $(t,b,a)$ of square-free integers satisfying 
\begin{equation}\label{list:conditions}
\begin{aligned}
&\text{(1)}\qquad a,b\neq 1,\quad a\neq b,\\
&\text{(2)}\qquad (a,b)_\nu(ab,-1)_\nu=+1
    \quad\text{for every place }\nu\text{ of }\QQ. \hspace{4.5cm}
\end{aligned}
\end{equation}

To make the parametrization more accessible to analytic arguments, we keep track of the various common divisors between $a,b$ and $t$. 
Following the notation of Koymans-Pagano, we write the square-free $a,b,t$ as products of $|Q_8|-1 = 7$ new variables $v_{x}$, which are pairwise coprime and indexed by the group elements $x\in Q_8\setminus\{1\}$. Given a tuple $(t,b,a)$ of square-free integers, we associate the $7$-tuple $\boldsymbol{v}=(v_i, v_{-i}, v_j, v_{-j}, v_{k}, v_{-k}, v_{-1})$ of square-free, \textit{pairwise coprime} integers via
\begin{align*}
    &v_i = \prod_{p \mid a , \  p\nmid b, \  p\nmid t} p, \quad  
    v_k = \prod_{p \nmid a , \  p\mid b, \  p\nmid t} p, \quad  
    v_j = \prod_{p \mid a , \  p\mid b, \  p\nmid t} p, \\  
    &v_{-i} = \prod_{p \mid a , \  p\nmid b, \  p\mid t} p, \quad  
    v_{-k} = \prod_{p \nmid a , \  p\mid b, \  p\mid t} p, \quad  
    v_{-j} =\prod_{p \mid a , \  p\mid b, \  p\mid t} p, \\
   & v_{-1} = \prod_{p \nmid a , \  p\nmid b, \  p\mid t} p.
\end{align*}
In the above we include $p=-1$ in the product, defining $-1\mid a$ if and only if $a$ is negative. 
 In these new variables   
\begin{align*}
    a &= v_jv_{-j}v_iv_{-i},\\
    b &= v_jv_{-j}v_kv_{-k}, \numberthis \label{eq:def_vs}\\
    t &= v_{-1}v_{-i}v_{-j}v_{-k}.
\end{align*}

 Let $\mu_x \in \{0,1\}$ be the 2-adic valuation of $v_xv_{-x}$ for $x\in \{i,j,k\}$ and write 
    \begin{align}
    v_xv_{-x} = 2^{\mu_x} b_x \qquad \text{ for }x\in \{i,j,k\},
    \label{eq:bx_vx}    
    \end{align}
    where $b_x$ is the odd part of $v_xv_{-x}$. We write $\legendre{n}{m}$ for the Jacobi symbol of $n$ and $m$. We define 
       \begin{align} \label{eq:def_ell}
              \ell_{\boldsymbol{\mu}, b_i, b_j, b_k}\! := \!\prod_{\substack{p\mid b_ib_jb_k  }}\! \frac{1}{2}\left(1+ \legendre{ -2^{\mu_i+\mu_j} b_ib_j}{p}\right)   \left(1+ \legendre{  -2^{\mu_i+\mu_k} b_ib_k }{p}\right) \left(1+ \legendre{ -2^{\mu_j+\mu_k} b_jb_k}{p}\! \right)
    \end{align}
    which, since $b_i, b_j, b_k$ are pairwise coprime,  takes values $0$ or $1$.

\begin{lemma} \label{lemma:solvable}
    Every epimorphism $\varphi \in \operatorname{Epi}_{\operatorname{cts}}(\Gamma_\QQ, Q_8)$ is uniquely parametrized by a 7-tuple of square-free, pairwise coprime integers $\boldsymbol{v}  = (v_i, v_{-i}, v_j, v_{-j}, v_k, v_{-k}, v_{-1}) \in (\ZZ\setminus\{0\})^7$  which, in the notation \eqref{eq:bx_vx} and \eqref{eq:def_ell} satisfies the following properties:
\begin{alignat}{2}
&\hbox{\text{\tiny $\bullet$}} \hspace{1cm} && 2^{\mu_x+\mu_y} b_xb_y \neq 1
    \text{, for all }x\neq y\in \{i,j,k\} \hspace{5cm}
    \tag{S1}\label{cond:S1}\\
&\hbox{\text{\tiny $\bullet$}} \ && (2^{\mu_i+\mu_j}b_ib_j, 2^{\mu_j+\mu_k} b_jb_k)_2 ( -1, 2^{\mu_i+\mu_k} b_ib_k)_2 = 1
    \tag{S2}\label{cond:S2}\\
&\hbox{\text{\tiny $\bullet$}} \ && v_i, v_{-i}, v_j, v_{-j}, v_k, v_{-k}>0
    \tag{S3}\label{cond:S3}\\
&\hbox{\text{\tiny $\bullet$}}  \  && \ell_{\boldsymbol{\mu}, b_i, b_j, b_k} = 1
    \tag{S4}\label{cond:S4}.
\end{alignat}
\end{lemma}
\begin{proof}
We have established the parametrization of $Q_8$-extensions via triples $(t,b,a)$ satisfying \eqref{list:conditions}. We now rewrite these in the variables $\boldsymbol{v}$. Since every prime divisor of $a,b,t$ is assigned to precisely one
of the variables $v_x$ according to the subset of $\{a,b,t\}$ in which it occurs, the integers
\[
v_i,v_{-i},v_j,v_{-j},v_k,v_{-k},v_{-1}
\]
are square-free and pairwise coprime. Conversely, every such $7$-tuple uniquely
determines the triple $(t,b,a)$ through \eqref{eq:def_vs}.

In the notation \eqref{eq:bx_vx}, we have 
\[
a = 2^{\mu_i+\mu_j}b_ib_j , \qquad b = 2^{\mu_j+\mu_k}b_jb_k , \qquad ab \equiv 2^{\mu_i+\mu_k}b_ib_k  \text{ in } \QQ^\times /\QQ^{\times2}.
\]
The condition 
\[
a,b\neq 1, \qquad a\neq b,
\]
is hence equivalent to \eqref{cond:S1}. The condition on the local Hilbert symbols 
\[
(a,b)_{\nu} (-1, ab)_{\nu} = +1
\]
In terms of the variables $2^{\mu_x}b_x$, this reads
\[
(2^{\mu_i+\mu_j}b_ib_j , 2^{\mu_j+\mu_k}b_jb_k)_\nu(-1, 2^{\mu_j+\mu_k}b_jb_k)_\nu =+1. 
\]
We verify this condition for all places $\nu$ of $\QQ$. For the place $\nu =2$ we directly arrive at the condition \eqref{cond:S2}. The remaining local computations are essentially identical to those in \cite[Lemma 5]{fouvry}. By the properties of the Hilbert symbol, using $(a,b)_\nu (-1, ab)_\nu = (-a,-b)_\nu(-1, -1)_\nu$, the above is equivalent to
\[
(-2^{\mu_i+\mu_j}b_ib_j , -2^{\mu_j+\mu_k}b_jb_k)_\nu(-1, -1)_\nu =+1. 
\]
Recall 
\begin{align*}
    (-1, -1)_\nu = \begin{cases}
        -1 &\text{ if }\;\nu=2,\\
        +1 &\text{ if }\;\nu=p \neq 2, \\
        -1 &\text{ if }\;\nu=\infty.
    \end{cases}
\end{align*}
For the place $\nu= \infty$ the condition is fulfilled if and only if $b_ib_j , b_ib_k>0$ which by coprimality forces $v_i, v_{-i}, v_j, v_{-j}, v_k, v_{-k}>0$, which gives condition \eqref{cond:S3}. For a prime $p \mid b_i, p\nmid b_jb_k$, we have 
\[
(-2^{\mu_i+\mu_j}b_ib_j , -2^{\mu_j+\mu_k}b_jb_k)_\nu = 1 \iff \legendre{-2^{\mu_j+\mu_k }b_jb_k}{p} = 1.
\]
 For a prime $p \mid b_j, p\nmid b_ib_k$, we have 
\[
(-2^{\mu_i+\mu_j}b_ib_j , -2^{\mu_j+\mu_k}b_jb_k)_\nu = 1 \iff \legendre{-2^{\mu_i+\mu_k} b_ib_k}{p} = 1.
\] 
For a prime $p \mid b_k, p\nmid b_ib_j$, we have 
\[
(-2^{\mu_i+\mu_j}b_ib_j , -2^{\mu_j+\mu_k}b_jb_k)_\nu = 1 \iff \legendre{-2^{\mu_i+\mu_j} b_ib_j}{p} = 1.
\]
These conditions are exactly picked up by the indicator function $\ell_{\boldsymbol{\mu} ,b_i,b_j,b_k} = +1$, giving \eqref{cond:S4}.
Finally, for any odd prime $p\nmid b_ib_jb_k$, we always have 
\[
(-2^{\mu_i+\mu_j}b_ib_j , -2^{\mu_j+\mu_k}b_jb_k)_\nu = 1.
\]
\end{proof}

\section{Multi-heights for \texorpdfstring{$Q_8$}{}-extensions}
We recall the definition of the family of multi-heights (multi-invariants) introduced by Gundlach \cite{gundlach2022mallesconjecturemultipleinvariants}. We give the definition in the special case of extensions over $\QQ$. 
\begin{definition} \label{def:multi}
Let $G$ be a finite group. Let $\mathcal{C}$ be the set of conjugacy classes of non-trivial cyclic subgroups of $G$. Let $\varphi \in \operatorname{Epi}_{\operatorname{cts}}(\Gamma_\QQ, G)$ be a $G$-extension. Let $(M, \sigma)$ be the associated number field with isomorphism $\operatorname{Gal}(M/\QQ) \overset{\sigma}{\simeq} G$. For a prime $\mathfrak p \mid p$, let $I_{\mathfrak{p}} \subseteq G$ be the image of the inertia subgroup at $\mathfrak p$ under $\sigma$. For each conjugacy class of non-trivial cyclic subgroup $c \in \mathcal{C}$, define a height function $h_c :  \operatorname{Epi}_{\operatorname{cts}}(\Gamma_\QQ, G) \to \RR_{>0}$ by 
\[
h_c(\varphi)=h_c(M, \sigma)  := \prod_{\substack{p \text{ tamely ramified in $M$}\\I_{\mathfrak p} \text{ conjugate to } c }} p
\]
\end{definition}
For $G=Q_8$ the non-trivial cyclic subgroups are
\[
\langle -1\rangle , \quad 
\langle i\rangle , \quad
\langle j\rangle , \quad
\langle k\rangle, 
\]
all of which are normal. Hence we obtain 4 height functions denoted by
\[
h_{-1},h_i, h_j, h_k : \{Q_8\text{-extensions}\} \to \RR_{>0}. 
\]
Let $(M, \sigma)$ be a $Q_8$-extension parametrized by the tuple of square-free numbers $(t,b,a)$. We read off the values of the height functions in terms of $(t,b,a)$. Proposition 2.9 of  \cite{Koymans-Pagano-nilpotent} describes the images of generators of tame inertia in terms of the variables $v_{x}$ introduced in \eqref{eq:def_vs}. In particular for an extension $(M, \sigma)$ parametrized by $(t,b,a)$, using the notation \eqref{eq:def_vs}, we have
\[
h_i (M , \sigma) = (v_iv_{-i})' , \quad 
h_j (M , \sigma) = (v_jv_{-j})' , \quad 
h_k (M , \sigma) = (v_kv_{-k})' , \quad 
h_{-1} (M , \sigma) = (v_{-1})' ,
\]
where $a'$ denotes the odd positive part of an integer $a$. Since the multi-heights only record tame ramification, the prime $2$ does not appear.

\subsection{Local conditions at the wildly ramified place} \label{sec:local2}
The wildly ramified places do not appear in the definition of the multi-heights. For $Q_8$ the only prime with wild ramification is $2$. In this section, we study when a global epimorphism $\varphi \in \operatorname{Epi}_{\operatorname{cts}}(\Gamma_\QQ, Q_8)$ agrees with a local specification  given by a continuous homomorphism $\Gamma_{\QQ_2} \to Q_8$ as in \Cref{thm2_local}.

Our fixed inclusions  $\Gamma_{\QQ_p} \hookrightarrow \Gamma_\QQ$ give rise to restriction maps 
\[
\operatorname{Epi}_{\text{cts}}(\Gamma_\QQ, Q_8) \to \operatorname{Hom}_{\text{cts}}(\Gamma_{\QQ_p}, Q_8).
\]

Given a continuous homomorphism $\Phi : \Gamma_{\QQ_2} \to Q_8$, not necessarily surjective, we study those global extensions $\varphi : \Gamma_\QQ \to Q_8 $, such that the restriction $\varphi_{|{\Gamma_{\QQ_2}}}$ equals $ \Phi$. 

Let
\[
\Phi=(\Phi_3,\Phi_2,\Phi_1):\Gamma_{\QQ_2}\longrightarrow Q_8
\]
be a continuous homomorphism, here we identify $Q_8$ with the group $G_3$ as in \Cref{sec:para} and write $\Phi_{1,2,3}$ for the coordinate maps. Composing with the quotient maps
\[
Q_8\twoheadrightarrow Q_8/\langle k\rangle
\qquad\text{and}\qquad
Q_8\twoheadrightarrow Q_8/\langle i\rangle,
\]
directly gives the maps $\Phi_1$ and $\Phi_2$ respectively. Hence $\Phi_2$ and $\Phi_1$ are themselves homomorphisms $\Gamma_{\QQ_2} \to \FF_2$, and we write
\[
\Phi_1=\chi_\alpha,\qquad
\Phi_2=\chi_\beta,
\]
for some square-classes
\[
\alpha,\beta\in \QQ_2^\times/\QQ_2^{\times 2}.
\]
\begin{remark} \label{rmk:local_hilbert2}
   Let $\Phi=(\Phi_3,\Phi_2,\Phi_1) \in \operatorname{Hom}_{\operatorname{cts}}(\Gamma_{\QQ_2}, Q_8)$ with $\Phi_1=\chi_\alpha$, $\Phi_2=\chi_\beta$ for some $\alpha,\beta\in \QQ_2^\times/\QQ_2^{\times 2}$. Then $\alpha, \beta$ always satisfy the local embedding problem 
\[
(\alpha, \beta)_2(-1, \alpha \beta)_2 = +1.
\]
On the other hand, any $\alpha, \beta$ that solve the above local embedding problem arise as the restriction of a global epimorphism by the Grunwald-Wang theorem. 
\end{remark}

Now let
\[
\varphi=(\varphi_3,\varphi_2,\varphi_1):\Gamma_{\QQ}\longrightarrow Q_8
\]
be a global epimorphism parametrized by a 7-tuple $\boldsymbol{v}$ of square-free, pairwise coprime integers . Recall that the parametrization gives
\[
\varphi_1=\chi_a,\qquad
\varphi_2=\chi_b,\qquad
\varphi_3=\theta_{a,b}+\chi_t.
\]
where $t,b,a$ are obtained from $\boldsymbol{v}$ via \eqref{eq:def_vs} and $\theta_{a,b}$ constructed as a \textit{minimal} lift which depends on $a,b$.

For the remainder of this section we use the notation introduced in \eqref{eq:bx_vx}. Thus, for a tuple $\boldsymbol{v}$ we write 
\[
v_x v_{-x} = 2^{\mu_x} b_x \qquad \text{ for }x\in \{i,j,k\}
\]
and moreover 
\[
v_{-i}v_{-j}v_{-k}v_{-1 } = 2^{\mu_t} \delta_t v_{-i}'v_{-j}'v_{-k}'v_{-1 }' 
\]
with $\delta_t\in \{\pm 1\}$, $\mu_t\in \{0,1\}$ and $v_x'$ the odd positive part of $v_x$.

We check the condition
\[
\varphi_{|\Gamma_{\QQ_2}}=\Phi
\]
coordinate-wise.  

Comparing the $\varphi_1$ and $\varphi_2$ coordinates yields 
\begin{align}
    \label{eq:local_b1}
2^{\mu_i +\mu_j} b_ib_j\equiv\alpha,\qquad
2^{\mu_k +\mu_j} b_kb_j\equiv\beta
\quad\text{in}\quad
\QQ_2^\times/\QQ_2^{\times2}.
\end{align}

Restricting the first coordinate of $\varphi$ to $\Gamma_{\QQ_2}$ gives
\[
\varphi_3|_{\Gamma_{\QQ_2}}
=
{\theta}_{2^{\mu_i+\mu_j}b_ib_j, 2^{\mu_j+\mu_k}b_jb_k}|_{\Gamma_{\QQ_2}}
+\chi_t|_{\Gamma_{\QQ_2}}.
\]
Hence the remaining compatibility condition is
\[
\chi_t|_{\Gamma_{\QQ_2}}
=
\Phi_3+{\theta}_{2^{\mu_i+\mu_j}b_ib_j, 2^{\mu_j+\mu_k}b_jb_k}|_{\Gamma_{\QQ_2}}.
\]
Now
\[
\Phi_3+{\theta}_{2^{\mu_i+\mu_j}b_ib_j, 2^{\mu_j+\mu_k}b_jb_k}|_{\Gamma_{\QQ_2}}
\]
is a homomorphism $\Gamma_{\QQ_2}\to \FF_2$, and we write 
\begin{align} \label{eq:def_sc_gamma}
\Phi_3+{\theta}_{2^{\mu_i+\mu_j}b_ib_j, 2^{\mu_j+\mu_k}b_jb_k}|_{\Gamma_{\QQ_2}} =: \chi_{\gamma_{\mu_i, \mu_j, \mu_k, b_i, b_j, b_k}}
\end{align}
with some square class
\[
\gamma_{\mu_i, \mu_j, \mu_k, b_i, b_j, b_k}\in
\QQ_2^\times/\QQ_2^{\times2}.
\]

We collect this information in the following lemma which characterizes when a global epimorphism restricts to a given 2-adic homomorphism. 
\begin{lemma} \label{lemma:local_comp}
    Let $\Phi \in \operatorname{Hom}_{\operatorname{cts}}(\Gamma_{\QQ_2}, Q_8)$ with coordinate maps $\Phi = (\Phi_3, \chi_\beta, \chi_\alpha)$ for $\alpha, \beta \in \QQ_2^\times /\QQ_2^{\times 2}$.  Let $\varphi \in \operatorname{Epi}_{\operatorname{cts}}(\Gamma_\QQ, Q_8)$ be an epimorphism parametrized by tuple $\boldsymbol{v} \in (\ZZ\setminus \{0\})^7$ as in \Cref{lemma:solvable} and let $\gamma_{\mu_i, \mu_j, \mu_k, b_i, b_j, b_k}\in
\QQ_2^\times/\QQ_2^{\times2}$ be given by \eqref{eq:def_sc_gamma}.  Write 
\[
\alpha \equiv 2^{\mu_\alpha} \alpha' , \qquad 
\beta \equiv 2^{\mu_\beta} \beta' , \qquad 
\gamma_{\mu_i, \mu_j, \mu_k, b_i, b_j, b_k} \equiv 2^{\mu_\gamma} \gamma_{\mu_i, \mu_j, \mu_k, b_i, b_j, b_k}'  \quad \text{in }\QQ_2^\times /\QQ_2^{\times 2} 
\]
with $\mu_\alpha, \mu_\beta, \mu_\gamma \in \{0,1\}$ and $\alpha', \beta', \gamma' \in (\ZZ/8\ZZ)^\times$. Then the exponents $\mu_\alpha ,\mu_\beta, \mu_\gamma$ depend only on $\Phi$ and  $\varphi_{|_{\Gamma_{\QQ_2}}} = \Phi$ if and only if 
\begin{alignat}{2}
&\hbox{\text{\tiny $\bullet$}} \hspace{1cm} && b_ib_j \equiv \alpha' \bmod 8,\qquad
b_jb_k \equiv \beta' \bmod 8
    \tag{L1} \hspace{6.5cm}\label{cond:L1}\\
&\hbox{\text{\tiny $\bullet$}} \ && \mu_j = \mu_\alpha\mu_\beta,\quad
\mu_i = \mu_\alpha- \mu_j,\quad
\mu_k = \mu_\beta- \mu_j,\quad
\mu_t = \mu_\gamma
    \tag{L2}\label{cond:L2}\\
&\hbox{\text{\tiny $\bullet$}} \ && \delta_t v_{-i}'v_{-j}'v_{-k}'v_{-1}'
\equiv \gamma_{\mu_i,\mu_j,\mu_k,b_i,b_j,b_k}'
\pmod{8}
    \tag{L3}\label{cond:L3}.
\end{alignat}
\end{lemma}
\begin{proof}
By the above discussion we have $\varphi_{|_{\Gamma_{\QQ_2}}} = \Phi$ if and only if \eqref{eq:local_b1} and 
\[
\gamma_{\mu_i, \mu_j, \mu_k, b_i, b_j, b_k} \equiv 2^{\mu_t} \delta_t v_{-i}'v_{-j}'v_{-k}'v_{-1 }'  \quad \text{ in }\QQ_2^\times /\QQ_2^{\times 2}.
\]
This is equivalent to the above three conditions \eqref{cond:L1}, \eqref{cond:L2}, \eqref{cond:L3}. Further, clearly $\mu_\alpha , \mu_\beta$ depend only on $\Phi$. Finally $\mu_\gamma$ is independent of $\mu_i, \mu_j, \mu_k, b_i, b_j, b_k$, by the construction of the minimal lift ${\theta}_{2^{\mu_i+\mu_j}b_ib_j, 2^{\mu_j+\mu_k}b_jb_k}$ in \cite[Prop. 2.6]{Koymans-Pagano-nilpotent} . 
\end{proof}

\section{Setting up a counting function} 
Fix a continuous homomorphism
\[
\Phi \in \operatorname{Hom}_{\mathrm{cts}}(\Gamma_{\QQ_2},Q_8).
\]
We study the number of global $Q_8$-extensions having prescribed local behaviour at $2$ and bounded multi-heights. For positive real numbers $\boldsymbol{X} = (X_i,X_j,X_k,X_{-1})\in \RR_{>0}^4$ define
\begin{align}
M_\Phi(\boldsymbol{X})
=
\#\left\{
\varphi\in\operatorname{Epi}_{\mathrm{cts}}(\Gamma_\QQ,Q_8):
\begin{array}{l}
h_x(\varphi)\le X_x,\ x\in\{i,j,k,-1\},\\
\varphi|_{\Gamma_{\QQ_2}}=\Phi
\end{array}
\right\}.
\label{eq:defMPhi}    
\end{align}

We establish an intermediate result to express $M_\Phi(\boldsymbol{X})$ analytically.
\begin{lemma} \label{lemma:M_first}
    Let $\boldsymbol{\mu} = (\mu_i, \mu_j, \mu_k) \in \{0,1\}^3$ and $\alpha', \beta' \in (\ZZ/8\ZZ)^\times$ be specified by $\Phi$ as in \Cref{lemma:local_comp}. Then 
    \begin{align*}
    &M_{\Phi}(\boldsymbol{X})  \\
    &= \!\!\!\sum_{ \substack{(\varepsilon_i, \varepsilon_j, \varepsilon_k) \in ((\ZZ/8\ZZ)^\times)^3 \\ \varepsilon_i\varepsilon_j \equiv \alpha' \bmod 8 \\ \varepsilon_j\varepsilon_k\equiv \beta'\bmod 8}} \;\sum_{\substack{b_x\leq  X_x \text{ for }x\in \{i,j,k\} \\
    b_x \equiv \varepsilon_x \bmod 8 \\
    \eqref{cond:S1} \text{ holds } } }\mu^2(2b_ib_jb_k) \, \ell_{\boldsymbol{\mu}, b_i, b_j, b_k}   \sum_{d \mid  \pm b_ib_jb_k} \sum_{\substack{b_{-1} \leq X_{-1} \\ 
d b_{-1} \equiv \gamma' \bmod 8}} \hspace{-.1cm}\mu^2(2b_{-1}b_ib_jb_k).
\end{align*}
with $\ell_{\boldsymbol{\mu}, b_i, b_j, b_k}  $ given by \eqref{eq:def_ell}.
\end{lemma}
\begin{proof}
We fix the first three heights and vary only
$h_{-1}$.
For fixed integers $b_i,b_j,b_k$ define
\begin{align*}
    M_{\Phi, b_i, b_j, b_k}( X_{-1}) :=  \# \left\{\varphi \in \operatorname{Epi}_{\operatorname{cts}}(\Gamma_\QQ, Q_8) : \begin{array}{l} h_{x}(\varphi) = b_x \text{ for all }x\in \{i,j,k\}, h_{-1}(\varphi) \leq X_{-1} ,  \\
    \varphi_{|\Gamma_{\QQ_2}} = \Phi \end{array} \right\}.
\end{align*}
The multi-heights take odd, square-free and pairwise coprime values, hence
\begin{align*}
     M_{\Phi}(\boldsymbol{X}) = \sum_{b_x\leq X_x \text{ for }x\in \{i,j,k\}} \mu^2(2b_ib_jb_k) M_{\Phi, b_i, b_j, b_k}(X_{-1}).
\end{align*}

By \Cref{sec:para}, an epimorphism counted by $M_{\Phi, b_i, b_j, b_k}( X_{-1})$ is parametrized by a tuple $\boldsymbol{v}$ of square-free pairwise coprime integers satisfying the conditions in \Cref{lemma:solvable} and \Cref{lemma:local_comp}. We check all conditions \eqref{cond:S1}-\eqref{cond:S4}, \eqref{cond:L1}-\eqref{cond:L3}.

The local homomorphism $\Phi$ specifies uniquely a parity vector $(\mu_i, \mu_j, \mu_k, \mu_t) \in \{0,1\}^4$ via \eqref{cond:L2}.  
Thus, if  \eqref{cond:L1} and \eqref{cond:L2} hold, then
\[
\alpha \equiv 2^{\mu_i+\mu_j} b_ib_j , \qquad \beta \equiv 2^{\mu_j+\mu_k} b_jb_k \qquad \text{ in } \QQ_2^\times /\QQ_2^{\times 2},
\]
and hence the condition \eqref{cond:S2} holds automatically by \Cref{rmk:local_hilbert2}. Writing out \eqref{cond:L1} and \eqref{cond:S1} we have 
\begin{align} \label{eq:admissableM}
   M_{\Phi}(\boldsymbol{X}) =   \sum_{ \substack{(\varepsilon_i, \varepsilon_j, \varepsilon_k) \in ((\ZZ/8\ZZ)^\times)^3 \\ \varepsilon_i\varepsilon_j \equiv \alpha' \bmod 8 \\ \varepsilon_j\varepsilon_k\equiv \beta'\bmod 8}} \sum_{\substack{b_x\leq  X_x \text{ for }x\in \{i,j,k\} \\
    b_x \equiv \varepsilon_x \bmod 8 \\ \eqref{cond:S1} \text{ holds}} }\mu^2(2b_ib_jb_k)   M_{\Phi, b_i, b_j, b_k}(X_{-1}).
\end{align}

The remaining conditions to check are 
\[
\eqref{cond:S3}, \eqref{cond:S4} \quad \text{ and } \eqref{cond:L3}.
\]
Checking these amounts to observing that every $M_{\Phi, b_i, b_j, b_k}( X_{-1})$ in \eqref{eq:admissableM} equals the number of tuples  $\boldsymbol{v}=(v_i, v_{-i}, v_j, v_{-j}, v_k, v_{-k}, v_{-1}) \in \ZZ^7$ such that 
\begin{itemize}
    \item $\mu^2(v_i \dots v_{-k}v_{-1})=1$,
    \item  $v_{x}>0$ for $x\in Q_8\setminus \{\pm 1\}$,
    \item $v_xv_{-x} = 2^{\mu_x}b_x$ for all $x\in \{i,j,k\}$,
    \item $\ell_{\boldsymbol{\mu}, b_i, b_j, b_k} = 1$,
    \item $v_{-1}'\leq X_{-1}$,
    \item $v_{-1}v_{-j}v_{-k}v_{-i} =\delta_t 2^{\mu_t} v_{-1}'v_{-j}'v_{-k}'v_{-i}' $,
    \item $v_{-1}'v_{-j}'v_{-k}'v_{-i}' \equiv \delta_t \gamma'$ for $\gamma' \equiv  \gamma'_{\mu_i, \mu_j, \mu_k, b_i, b_j, b_k} \in (\ZZ/8\ZZ)^\times$.
\end{itemize}
The above conditions uniquely determine the parities of the individual $v_x$, $x\in Q_8\setminus\{1\}$, hence we restrict to counting odd integers. Since the function $\ell_{\boldsymbol{\mu}, b_i, b_j, b_k}$ takes values $0$ or $1$, we obtain for any $ M_{\Phi, b_i, b_j, b_k}( X_{-1})$ counted in \eqref{eq:admissableM} 
\begin{align*}
    M_{\Phi, b_i, b_j, b_k}( X_{-1}) &= \ell_{\boldsymbol{\mu}, b_i, b_j, b_k} \sum_{\delta_t \in \{\pm1\}} \;\sum_{v_{-i}'v_{-j}'v_{-k}' \mid   b_ib_jb_k}  \sum_{\substack{v_{-1}'\leq X_{-1} \\ 
\delta_t v_{-i}'v_{-j}'v_{-k}' v_{-1}' \equiv \gamma' \bmod 8}}  \hspace{-.6cm}\mu^2(2v_{-1}'b_ib_jb_k)  \\
&=   \ell_{\boldsymbol{\mu}, b_i, b_j, b_k} \sum_{d \mid  \pm b_ib_jb_k} \sum_{\substack{v_{-1}' \leq X_{-1} \\ 
d v_{-1}' \equiv \gamma' \bmod 8}} \hspace{-.1cm}\mu^2(2v_{-1}'b_ib_jb_k).
\end{align*}
Substituting this into \eqref{eq:admissableM} shows the result.
\end{proof}

We have the following result which allows us to rewrite the term $\ell_{\boldsymbol{\mu}, b_i, b_j, b_k}$. 
\begin{lemma} \label{lemma:ell_rewrite}
    Let $b_i, b_j, b_k$ be odd, square-free, positive and pairwise coprime integers. Let $(\mu_i, \mu_j, \mu_k)\in \{0,1\}^3$ with $\mu_i+\mu_j+\mu_k\leq 1$. Let $\ell_{\boldsymbol{\mu}, b_i, b_j, b_k}$ be given by \eqref{eq:def_ell}. Then 
    \begin{align*}
        \ell_{\boldsymbol{\mu}, b_i, b_j, b_k} = \frac{1}{\tau(b_ib_jb_k)} \sum_{d_xf_x = b_x \text{ for }x\in \{i,j,k\}} u_{\boldsymbol{\mu}}(d_i, d_j, d_k) \legendre{f_i}{d_jd_k}\legendre{f_j}{d_id_k} \legendre{f_k}{d_id_j},
    \end{align*}
    where $u_{\boldsymbol{\mu}}(d_i, d_j, d_k)$ is a function supported on odd integers given by 
    \[ 
        u_{\boldsymbol{\mu}}(d_i, d_j, d_k) := (-2^{\mu_i+ \mu_j}, d_k)_2 (-2^{\mu_i+\mu_k}, d_j)_2 (-2^{\mu_j+\mu_k}, d_i)_2 (d_i, d_j)_2 (d_i, d_k)_2 (d_j, d_k)_2.
    \]
\end{lemma} 
\begin{proof}
    By multiplying out the product in the definition of $\ell_{\boldsymbol{\mu}, b_i, b_j, b_k}$ we have
     
\begin{align*}
    \ell_{\boldsymbol{\mu}, b_i, b_j, b_k} =  \frac{1}{\tau(b_ib_jb_k)} \sum_{d_xf_x = b_x \text{ for }x \in \{i,j,k\}} \legendre{-2^{\mu_i+\mu_j} b_ib_j}{d_k} \legendre{-2^{\mu_j+\mu_k} b_jb_k}{d_i} \legendre{-2^{\mu_i+\mu_k} b_ib_k}{d_j} .  
\end{align*}

We simplify the above quadratic symbols. Using $\legendre{n}{d} = (n, d)_2$ for $d$ odd and $n\in \{\pm1, \pm2\}$, we have
\begin{align*}
    &\legendre{ - 2^{\mu_i+\mu_j}b_ib_j}{d_k} \legendre{ -2^{\mu_j+\mu_k}b_jb_k}{d_i}\legendre{ -2^{\mu_i+\mu_k}b_ib_k}{d_j}  \\
    &= (-2^{\mu_i+ \mu_j}, d_k)_2 (-2^{\mu_i+\mu_k}, d_j)_2 (-2^{\mu_j+\mu_k}, d_i)_2 \legendre{d_if_id_jf_j}{d_k} \legendre{d_jf_jd_kf_k}{d_i} \legendre{d_if_id_kf_k}{d_j} .
\end{align*}
Applying quadratic reciprocity, $\legendre{d}{d'}\legendre{d'}{d} = (d,d')_2$ for odd $d,d'$,  the above equals 
\begin{align*}
      u_{\boldsymbol{\mu}}(d_i, d_j, d_k) \cdot \legendre{f_i }{d_jd_k} \legendre{f_j}{d_i d_k} \legendre{f_k}{d_i d_j},
\end{align*}
with $u_{\boldsymbol{\mu}}(d_i, d_j, d_k)$ supported on odd integers as in the statement of the lemma.
\end{proof}

 We collect some useful properties of the function $u_{\boldsymbol{\mu}}$ above. 
 \begin{lemma}\label{lemma:properties_u}
     Let $\boldsymbol{\mu} = (\mu_i, \mu_j, \mu_k) \in \{0,1\}^3$ with $\mu_i+\mu_j+\mu_k \leq 1$. Let $u_{\boldsymbol{\mu}}(d_i, d_j, d_k)$ be the function supported on odd integers introduced above. Then the following holds.
     \begin{enumerate}
         \item[(i)]  Let $\{x,y,z\} = \{i,j,k\}$. For fixed variables $d_x, d_y$, the function $u_{\boldsymbol{\mu}}$ is, up to a fixed sign, a character modulo 8 in the remaining variable $d_z$.  This character is the principal if and only if $\mu_x+\mu_y = 0$ and $d_x d_y \equiv -1 \bmod 4$.
        \item[(ii)] We have
         \[
         u_{\boldsymbol{\mu}}(d_i, d_j, d_k) = (A,B)_2(AB, -1)_2,
         \]
        with $A = 2^{\mu_j+\mu_i} d_id_j$ and $B= 2^{\mu_j+\mu_k}d_jd_k$   .
     \end{enumerate}
 \end{lemma}
 \begin{proof}
     $(i)$ follows immediately from the definition of $u_{\boldsymbol{\mu}}$. Indeed if $d_x, d_y$ are fixed, the resulting function from $u_{\boldsymbol{\mu}}$ in $d_z$ takes the value 
    \[
     \pm (-2^{\mu_x+\mu_y} d_xd_y, d_z)_2  
     \]
     for $d_z$ odd. This is a character modulo $8$ and it is principal if and only if $\mu_x+\mu_y = 0$ and $d_x d_y \equiv -1 \bmod 4$. \\
     For $(ii)$, we set $A = 2^{\mu_j+\mu_i} d_id_j$ and $B= 2^{\mu_j+\mu_k}d_jd_k$, and expand the expression 
     \[
     (A,B)_2(AB, -1)_2.
     \]
     Using the properties of the Hilbert symbol, we have
 \begin{align*}
     &(A,B)_2(AB,-1)_2 = \\
     &= \Big[ (2^{\mu_j+\mu_i}, 2^{\mu_j+\mu_k})_2(2^{\mu_i+\mu_j}, d_j)_2(2^{\mu_i+\mu_j}, d_k)_2(d_id_j, 2^{\mu_j+\mu_k})_2(d_id_j, d_jd_k)_2 \Big] (AB, -1)_2 \\
     &= \Big[ (2^{\mu_j+\mu_i}, 2^{\mu_j+\mu_k})_2 (2^{\mu_i+\mu_k}, d_j)_2 (2^{\mu_i+\mu_j}, d_k)_2 (d_i, 2^{\mu_j+\mu_k})_2(-1, d_j)_2(d_i, d_j)_2(d_i, d_k)_2 (d_j, d_k)_2 \Big] \\
     &\qquad \cdot \Big[ (2^{\mu_i+\mu_k}, -1)_2(d_i, -1)_2 (d_k, -1)_2\Big] \\
     &= (-2^{\mu_i+\mu_k}, d_j)_2 (-2^{\mu_i+\mu_j},d_k)_2(-2^{\mu_j+ \mu_k}, d_i)_2 (d_i, d_j)_2(d_i, d_k)_2 (d_j, d_k)_2 \\
     &\qquad \cdot (2^{\mu_j+\mu_i}, 2^{\mu_j+\mu_k})_2 (2^{\mu_i+\mu_k}, -1)_2\\
     &= u_{\boldsymbol{\mu}}(d_i, d_j, d_k)   (2^{\mu_j+\mu_i}, 2^{\mu_j+\mu_k})_2 (2^{\mu_i+\mu_k}, -1)_2.
 \end{align*}
 Finally, notice that  $ (2^{\mu_j+\mu_i}, 2^{\mu_j+\mu_k})_2\cdot (2^{\mu_i+\mu_k}, -1)_2 =1$ for any of the admissible choices for $\boldsymbol{\mu}$. 
 \end{proof}

Continuing the analysis of $M_{\Phi}(\boldsymbol{X})$ from \Cref{lemma:M_first}, the non-degeneracy condition \eqref{cond:S1} in the new variables $d_xf_x$  reads
\begin{align}
      2^{\mu_x+\mu_y}d_xf_x d_yf_y \neq 1\text{, for all }x\neq y\in \{i,j,k\} \label{eq:condition_G}. 
\end{align}
\begin{lemma} \label{lemma:M_second}
   Let $\boldsymbol{\mu} = (\mu_i, \mu_j, \mu_k) \in \{0,1\}^3$, $\alpha', \beta' \in (\ZZ/8\ZZ)^\times$ be specified by $\Phi \in \operatorname{Hom}_{\operatorname{cts}}(\Gamma_{\QQ_2}, Q_8)$ as in \Cref{lemma:local_comp}. For $\boldsymbol{X} = (X_i, X_j, X_k, X_{-1}) \in \RR_{>0}^4$, we have
\begin{align*}
  M_{\Phi} (\boldsymbol{X})  &= \frac{1}{2} \sum_{ \substack{(\varepsilon_i, \varepsilon_j, \varepsilon_k )\in ((\ZZ/8\ZZ)^\times)^3\\ (\varepsilon_i\varepsilon_j, \varepsilon_j\varepsilon_k)\equiv (\alpha', \beta') \bmod 8}} N_{\boldsymbol{\mu}, \boldsymbol{\varepsilon}}(\boldsymbol{X})   + O\Bigg(\sum_{\substack{b_x \leq X_x \text{ for }x \in \{i,j,k\} }} \mu^2(b_ib_jb_k)   M^{\chi_1}_{\Phi, b_i, b_j, b_k}( X_{-1}) \Bigg)  
\end{align*}
with 
\begin{align}
    N_{\boldsymbol{\mu}, \boldsymbol{\varepsilon}}(\boldsymbol{X})  &:= \hspace{-.6cm} \sum_{\substack{(d_i, f_i, d_j, f_j, d_k, f_k, v_{-1}) \in \ZZ_{>0}^7 \\ d_x f_x \leq X_x \\ v_{-1}\leq X_{-1}\\ d_xf_x \equiv \varepsilon_x \bmod 8 \\\eqref{eq:condition_G} \text{ holds}}} \hspace{-.7cm}{\mu^2(2d_if_id_jf_jd_kf_kv_{-1})} u_{\boldsymbol{\mu}}(d_i,d_j, d_k)    \legendre{f_i }{d_jd_k} \legendre{f_j}{d_i d_k} \legendre{f_k}{d_i d_j} \label{eq:def_N_main}.
\end{align}
and 
\begin{align} \label{eq:M_chi1}
   M^{\chi_1}_{\Phi, b_i, b_j, b_k} (X_{-1}) :=  \Bigg| \prod_{p\mid b_ib_jb_k} (1+\chi_1(p) ) \sum_{\substack{v_{-1} \leq X_{-1}  \\ (v_{-1}, 2b_ib_jb_k)=1}} \chi_1(v_{-1}) \mu^2(v_{-1}) \Bigg|
\end{align}
for the non-trivial character $\chi_1 \bmod 8$ with $\chi_{-1}(-1)=1$.
\end{lemma}
\begin{proof}
 We detect the congruence condition $d v_{-1}' \equiv \gamma' \bmod 8$ appearing in \Cref{lemma:M_first}  via a sum over the character modulo 8. 
We have
    \begin{align*}
 & \sum_{d \mid  \pm b_ib_jb_k} \sum_{\substack{b_{-1} \leq X_{-1} \\ 
d v_{-1}' \equiv \gamma' \bmod 8}} \hspace{-.3cm}\mu^2(2b_{-1}b_ib_jb_k) \\
&\hspace{2cm}= \frac{\mu^2(2b_ib_jb_k)}{4}\sum_{\chi \bmod 8} \overline{\chi}(\gamma')\Bigg(     \sum_{d\mid \pm b_ib_jb_k} \chi(d)  \!\! \sum_{\substack{v_{-1}' \leq X_{-1}  \\ (v_{-1}', 2b_ib_jb_k)=1}} \!\!\!\chi(v_{-1} ) \mu^2(v_{-1}) \Bigg) .
    \end{align*}
For square-free, pairwise coprime $b_i, b_j, b_k$ we can rewrite 
\[
 \sum_{d\mid \pm b_ib_jb_k} \chi(d)   = (1+\chi(-1)) \prod_{p\mid b_ib_jb_k} \left(1+\chi(p)\right).
\]
In particular the two characters modulo 8 with $\chi(-1)=-1$ identically vanish. Denote by $\chi_1$ the unique non-principal character modulo 8 with $\chi_1(-1)=1$ and let $\chi_0$ be the principal character modulo 8. We separate the contribution coming from these two characters by setting 
\begin{align*}
    \sum_{d \mid  \pm b_ib_jb_k} \sum_{\substack{v_{-1}' \leq X_{-1} \\ 
d v_{-1}' \equiv \gamma' \bmod 8}} \hspace{-.3cm}\mu^2(2b_{-1}b_ib_jb_k) = M^{\chi_1}_{ b_i, b_j, b_k} (X_{-1})  + M^{\chi_0}_{ b_i, b_j, b_k} (X_{-1})
\end{align*}
with 
\begin{align} 
   M^{\chi_1}_{ b_i, b_j, b_k} (X_{-1}) &=  \frac{\overline{\chi_1(\gamma')}}{2}\prod_{p\mid b_ib_jb_k} (1+\chi_1(p) ) \sum_{\substack{v_{-1} \leq X_{-1}  \\ (v_{-1}, 2b_ib_jb_k)=1}} \chi_1(v_{-1}) \mu^2(v_{-1}) \\
   & \ll \Bigg| \prod_{p\mid b_ib_jb_k} (1+\chi_1(p) ) \sum_{\substack{v_{-1} \leq X_{-1}  \\ (v_{-1}, 2b_ib_jb_k)=1}} \chi_1(v_{-1}) \mu^2(v_{-1}) \Bigg|
\end{align}
and 
\begin{align} 
    M^{\chi_0}_{b_i, b_j, b_k} (X_{-1})  :=   \frac{ \tau(b_ib_jb_k)}{2} \sum_{\substack{v_{-1} \leq X_{-1}  \\ (v_{-1}, 2b_ib_jb_k)=1}} \mu^2(v_{-1}). 
\end{align}

Substituting this back into \Cref{lemma:M_first},  we have 

\begin{align*}
	 M_{\Phi}(X_i, X_j, X_k , X_{-1}) &= \sum_{ \substack{(\varepsilon_i, \varepsilon_j, \varepsilon_k) \in ((\ZZ/8\ZZ)^\times)^3 \\ \varepsilon_i\varepsilon_j \equiv \alpha' \bmod 8 \\ \varepsilon_j\varepsilon_k\equiv \beta'\bmod 8}} \sum_{\substack{b_x\leq  X_x \text{ for }x\in \{i,j,k\} \\
    b_x \equiv \varepsilon_x \bmod 8 \\ \eqref{cond:S1} \text{ holds}} } \mu^2(2b_ib_jb_k) \ell_{\boldsymbol{\mu}, b_i, b_j, b_k}  M^{\chi_0}_{ b_i, b_j, b_k}( X_{-1}) \\ 
      &\qquad +O\Bigg(\sum_{\substack{b_x \leq X_x \text{ for }x \in \{i,j,k\}   }} \mu^2(b_ib_jb_k)  \left| M^{\chi_1}_{ b_i, b_j, b_k}( X_{-1}) \right| \Bigg) .
\end{align*}

Inserting the expression obtained for $\ell_{\boldsymbol{\mu}, b_i, b_j, b_k}$ from \Cref{lemma:ell_rewrite} yields
\[
\sum_{\substack{b_x\leq  X_x \text{ for }x\in \{i,j,k\} \\
    b_x \equiv \varepsilon_x \bmod 8 \\ \eqref{cond:S1} \text{ holds}} } \mu^2(2b_ib_jb_k) \ell_{\boldsymbol{\mu}, b_i, b_j, b_k}   M^{\chi_0}_{ b_i, b_j, b_k}( X_{-1})  = \frac{1}{2}  N_{\boldsymbol{\mu}, \boldsymbol{\varepsilon}}(\boldsymbol{X}) 
\]
with $ N_{\boldsymbol{\mu}, \boldsymbol{\varepsilon}}(\boldsymbol{X}) $ defined in the Lemma, and the result follows. 
\end{proof}

We postpone until a later section the proof that the contribution from $ M^{\chi_1}_{\Phi, b_i, b_j, b_k} (X_{-1}) $ is negligible.
 Before continuing the analysis of the counting function, we collect some general Lemmas to deal with the character sums that appear in $N_{\boldsymbol{\mu}, \boldsymbol{\varepsilon}}$.

\section{Analytic lemmas}
A main analytic tool is an estimate on bilinear sums of Friedlander--Iwaniec \cite{friedlander-iwaniec}.
 \begin{lemma}[\texorpdfstring{\cite[Lemma 2]{friedlander-iwaniec}}{}] \label{lemma:large_sieve}
     Let $\alpha_m, \beta_n$ be a sequence of complex numbers with $|\alpha_m|, |\beta_n|\leq 1$ supported on odd integers, then 
     \begin{align*}
         \sum_{\substack{n\leq N \\ m\leq M}} \alpha_m \beta_n \legendre{n}{m} \ll NM \left( N^{-1/6}+M^{-1/6}\right) \log(NM)^{7/6}\,.
     \end{align*}
 \end{lemma}
Next, we prove a result on linear character sums over square-free integers with multiplicative coefficients. We study sums of the form 
\[
\sum_{n\leq X} \mu^2(n) \chi(n) f(n),
\]
where $f$ is multiplicative and ``close to 1''. A naive application of P\'olya--Vinogradov would give a bound of roughly size $(qX)^{1/2+\varepsilon}$, where $q$ is the size of the modulus of $\chi$. The goal is to follow the approach of Munsch \cite[Lemma 2.3]{munsch_character_sums} to obtain a non-trivial upper bound of size roughly $q^{1/4+\varepsilon} X^{1/2+\varepsilon}$.
\begin{lemma}\label{lem:square_free_character_sum_g}
    Let $g$ be a function prescribed on primes, with $|g(p)  -1| \leq C/p$ for some constant $C$. Let $G(n)=\prod_{p\mid n} g(p)$.  Let $\chi$ be a non-principal character of modulus $q$, let $m \in \ZZ$ with $(q, m)=1$. Then for any $\varepsilon>0$
    \[
    \sum_{\substack{n\leq X\\(n,m)=1}} \mu^2(n)\chi(n)G(n) \ll_{C, \varepsilon}  \tau(m)^{1/2} q^{1/4} (\log q)^{1/2} X^{1/2+\varepsilon} .
    \]
\end{lemma}
\begin{proof}
We have
    \begin{align}
         \sum_{\substack{n\leq X\\(n,m)=1}} \mu^2(n)\chi(n)G(n) &=   \sum_{d\leq X^{1/2}} \mu(d) \sum_{\substack{n\leq X\\ {d^2\mid n} \\ (n,m) = 1}} \chi(n) G(n) \nonumber \\
         &\ll   \Big| \sum_{d\leq H} \mu(d) \sum_{\substack{n\leq X\\ {d^2\mid n} \\ (n,m) = 1}} \chi(n) G(n) \Big| + \sum_{H<d\leq X^{1/2}} \frac{X^{1+\varepsilon}}{d^2} \label{eq:split_H}
    \end{align}
    for an auxiliary parameter $H$ to be chosen later. The last step uses $G(n) \ll n^\varepsilon$.\\
    Define a multiplicative function $F = G\ast \mu$. Then $G(n) =  (F\ast 1)(n)= \sum_{e\mid n} F(e)$ and we have 
    \[
    F(p)= G(p)-G(1) = g(p)-1, \qquad F(p^k)=G(p^k)-G(p^{k-1})= 0 \text{ for $k\geq 2$}.
    \]
    In particular $|F(p)| \leq C/p$. With this, the first sum in \eqref{eq:split_H} can be manipulated further 
    \begin{align*}
 \sum_{\substack{d\leq H  }} \mu(d) \sum_{\substack{n\leq X \\ d^2 \mid n \\ (n,m)=1}} \chi(n) G(n)     &= \sum_{\substack{d\leq H \\ (d, m)=1 }} \mu(d) \sum_{\substack{n\leq X \\ d^2 \mid n \\ (n,m)=1}}  \chi(n) \sum_{e\mid n} F(e) \\       
        &= \sum_{\substack{d\leq H \\(d,m)=1 }} \mu(d) \sum_{\substack{e\leq X \\ (e , m)=1}} F(e) \chi([d^2,e])  \sum_{\substack{n\leq X/[d^2,e] \\ (n,m)=1}}  \chi(n) \,. \numberthis \label{eq:character_sum_manipulated}
\end{align*}
Now using a Möbius inversion to pick up the coprimality condition in the last sum and applying P\'olya--Vinogradov gives that 
\eqref{eq:character_sum_manipulated} is bounded by
\begin{align*}\ll  \sum_{\substack{d\leq H \\(d,m)=1 }} \sum_{\substack{e\leq X \\ (e , m)=1}} |F(e)| \tau(m) q^{1/2} \log q.
\end{align*}
Using $|F(p)|\leq C/p$ and $F(p^k)=0$ for $k\geq 2$, we have 
\[
\sum_{\substack{e\leq X \\ (e , m)=1}} |F(e)|  \ll \prod_{p\leq X} \left(1+|F(p)|\right)\ll_{C, \varepsilon}  X^\varepsilon .
\]
Hence \eqref{eq:character_sum_manipulated} is bounded by $O(\tau(m) q^{1/2} \log q \cdot H \cdot X^{\varepsilon})$.  Plugging this into \eqref{eq:split_H}, we have 
\[
  \sum_{\substack{n\leq X\\(n,m)=1}}  \mu^2(n)\chi(n)G(n) \ll_{C, \varepsilon} \tau(m) H X^{\varepsilon} q^{1/2}  (\log q) + \frac{X^{1+\varepsilon}}{H}\,.
\]
Choosing $H= q^{-1/4} (\tau(m) \log q)^{-1/2} X^{1/2-\varepsilon/2} $ establishes the bound.
\end{proof}
The next lemma is crucial for the bounds of our character sum. For a similar result in the context of counting $D_4$-extension by conductor, see \cite[Lemma 5.1]{d4_conductor}.
\begin{lemma} \label{lemma:non_square_contribution}
    Let $X_1,X_2 \in \mathbb{R}_{>0}$, $m\in \ZZ\setminus \{0\}$, $\delta>0$. Let $g$ be a function prescribed for primes $p$ with $|g(p)-1|\leq C/p$ for some $C\geq  0$ and let $G(n)=\prod_{p\mid n} g(p)$. Then  for any $\varepsilon>0$ 
        \begin{align*}
   & \sum_{\substack{ d_1\leq X_1 \\ (d_1, m) = 1}} \mu^2(d_1)G(d_1) \sum_{\substack{d_2 \leq X_2 \\  (d_2, md_1)=1}} \mu^2(d_2) G(d_2)  \sum_{\substack{f \leq (X_1X_2)^{1/2+\delta}\\ f \neq \square}}  \frac{G(mf)}{f}\legendre{d_1d_2}{f} \\
    & \hspace{5cm}\ll_{C, \varepsilon}  \tau(m) (X_1X_2)^{3/4 +\delta/2 +\varepsilon}  . 
\end{align*}

\end{lemma}
\begin{proof}
    We start by interchanging the summation 
    \begin{align*}
     &\sum_{\substack{ d_1\leq X_1 \\ (d_1, m) = 1}} \mu^2(d_1)G(d_1) \sum_{\substack{d_2 \leq X_2 \\  (d_2, md_1)=1}} \mu^2(d_2) G(d_2)  \sum_{\substack{f \leq (X_1X_2)^{1/2+\delta}\\ f \neq \square}} \frac{G(mf)}{f}\legendre{d_1d_2}{f}  \\
     & \qquad\qquad = \sum_{\substack{f \leq (X_1X_2)^{1/2+\delta}\\ f \neq \square}}  \frac{G(mf)}{f}  \sum_{\substack{ d_1\leq X_1 \\ (d_1, m) = 1}} \mu^2(d_1)G(d_1)\legendre{d_1}{f}\sum_{\substack{d_2 \leq X_2 \\  (d_2, md_1)=1}} \mu^2(d_2) G(d_2)  \legendre{d_2}{f}.
    \end{align*}
  Lifting the coprimality condition $(d_1, d_2)=1$ with a Möbius sum over $e\mid (d_1, d_2)$ and taking absolute values gives
    \begin{align*}
        \ll \!\!\!\!\! \sum_{\substack{f \leq (X_1X_2)^{1/2+\delta}\\ f \neq \square}}  \frac{|G(mf)|}{f} \sum_{\substack{e\leq X_1 \\ (e, fm)=1}} |G(e)|^2 \Bigg|\sum_{\substack{ d_1\leq X_1/e \\ (d_1, em) = 1}} \mu^2(d_1)G(d_1)\legendre{d_1}{f} \Bigg| \Bigg|\sum_{\substack{d_2 \leq X_2/e \\  (d_2, em)=1}} \mu^2(d_2) G(d_2)  \legendre{d_2}{f} \Bigg|.
    \end{align*}
    We are now in a position to apply \Cref{lem:square_free_character_sum_g} to the sums over $d_1$ and $d_2$. This allows us to bound the above by 
    \begin{align*}
        &\ll \sum_{f\leq (X_1X_2)^{1/2+\delta}} \frac{|G(mf)| }{f}\cdot \left( f^{1/2}(\log f) \tau(m) (X_1X_2)^{1/2+\varepsilon} \right)  \sum_{e\leq X_1} \frac{\tau(e)  |G(e) |^2}{e} . \numberthis \label{eq:lemma_lastbound}
    \end{align*}
    Using the crude bound
    \begin{align*}
        G(n) = \prod_{p\mid n} g(p) \ll \prod_{p\mid n} \left( 1+\frac{C}{p} \right) \ll (\log n)^{C} ,
    \end{align*}
    we have that \eqref{eq:lemma_lastbound} is $O_{C, \varepsilon}( \tau(m) (X_1X_2)^{3/4 +\delta/2+\varepsilon})$ as desired.
\end{proof}

\begin{lemma} \label{lemma:L_func_bound}
   Let $g$ be a function prescribed on primes, with $|g(p)  -1| \leq C/p$ for a constant $C>0$. Let $G(n)=\prod_{p\mid n} g(p)$. Let $\chi$ be a non-trivial character modulo $q$, then 
   \[
   \sum_{n\leq X} \frac{\mu^2(n)G(n)\chi(n)}{n} \ll_C \sqrt{q}\log q.
   \]
\end{lemma}
\begin{proof}
   We express $\mu^2(n)G(n)$ in terms of a  Dirichlet convolution. Define a multiplicative function by $F := \mu^2G\ast \mu$. Then $\mu^2(n)G(n) = \sum_{d\mid n} F(d)$ and we evaluate  
   \[
   F(p) = g(p)-1   , \qquad F(p^2) = - g(p), \qquad F(p^k)=0 \text{ for } k\geq 3.
\]
By hypothesis, it follows $|F(p)| \leq C/p$ and $|F(p^2)| \leq 1+ \frac{C}{p}$.

Returning to our main sum, we obtain
\[
\sum_{n \leq X} \frac{\mu^2(n)G(n)\chi(n)}{n} = \sum_{n \leq X} \frac{\chi(n)}{n} \sum_{d \mid n} F(d).
\]
Interchanging the summation using the multiplicativity of the character yields
\[
\sum_{d \leq X} \frac{F(d)\chi(d)}{d} \sum_{m \leq X/d} \frac{\chi(m)}{m}.
\]
By partial summation and  Pólya--Vinogradov, we have 
\[
\left| \sum_{m \leq Y} \frac{\chi(m)}{m} \right|\ll \sqrt{q}\log q,
\]
uniformly for any $Y \geq 1$.

Taking absolute values of our rearranged sum and applying this bound, we obtain
\[
\left| \sum_{n \leq X} \frac{\mu^2(n)G(n)\chi(n)}{n} \right| \leq \sum_{d \leq X} \frac{|F(d)|}{d} \left| \sum_{m \leq X/d} \frac{\chi(m)}{m} \right| \ll \sqrt{q}\log q \sum_{d=1}^{\infty} \frac{|F(d)|}{d}.
\]
We finish the proof by showing that the sum over $d$ converges. 
 Because $F$ is multiplicative, we can expand it as an Euler product
\[
\sum_{d=1}^{\infty} \frac{|F(d)|}{d} = \prod_{p} \left( 1 + \frac{|F(p)|}{p} + \frac{|F(p^2)|}{p^2} \right) \ll \prod_{p} \left( 1 + \frac{2C+1}{p^2}\right)   \ll_C 1 ,
\]
 which shows convergence.
\end{proof}
Finally, we state a lemma that allows us to give an asymptotic expression for sums over square-free numbers. 

\begin{lemma} \label{lemma:square-free-asymp}
   Let $g$ be a function prescribed on primes, with $|g(p)  -1| \leq C/p$, for a constant $C>0$. Let $G(n)=\prod_{p\mid n} g(p)$. Let $m,a,q\in \ZZ$ with $(a,q) = (m,q)=1$. Then for any $\varepsilon>0$ we have
   \begin{align*}
   \sum_{\substack{n\leq X\\ (n,m)=1 \\ n\equiv a \bmod q}} {\mu^2(n)G(n)} & = \frac{1}{\varphi(q)} \prod_{p} \left(1-\frac{1}{p}\right) \left(1+\frac{g(p)}{p}\right)\prod_{p\mid mq} \left(1+\frac{g(p)}{p}\right)^{-1} X\\
   &\qquad + O\left(\tau(m) X^{1/2+\varepsilon}q^{1/4+\varepsilon}\right)  .  
   \end{align*}
   
\end{lemma}
\begin{proof}
    Using orthogonality of the Dirichlet characters modulo $q$, we lift the condition $n\equiv a\bmod q$ via 
    \[
    1_{n\equiv a \bmod q} = \frac{1}{\varphi(q)} \sum_{\chi \bmod q} \overline{\chi(a)} \chi(n).
    \]
    Every non-principal character gives an error term of $O(\tau(m)X^{1/2+\varepsilon}q^{1/4+\varepsilon})$ using \Cref{lem:square_free_character_sum_g}, hence 
    \[
      \sum_{\substack{n\leq X\\ (n,m)=1 \\ n\equiv a \bmod q}} {\mu^2(n)G(n)}  =   \frac{1}{\varphi(q)}\sum_{\substack{n\leq X\\ (n,mq)=1 }} {\mu^2(n)G(n)}  +O(\tau(m)X^{1/2+\varepsilon}q^{1/4+\varepsilon}).
    \]
    Define the auxiliary function $F = \mu^2G \ast \mu$, so that $\mu^2(n)G(n) = \sum_{d\mid n} F(d)$. We calculate
    \[
    F(p)=G(p)-1 = g(p)-1 , \qquad F(p^2)= -G(p)=-g(p) , \qquad F(p^k) = 0 \text{ for }k>2.
    \]
    In particular $|F(p)|\leq C/p$ and $|F(p^2)|\leq C+1$. Interchanging summation gives
    \[
     \sum_{\substack{n\leq X\\ (n,qm)=1  }} {\mu^2(n)G(n)} = \sum_{\substack{d\leq X \\ (d,qm)=1}} F(d) \sum_{\substack{n\leq X/d\\ (n,qm)=1}} 1 = \sum_{\substack{d\leq X \\ (d,qm)=1}} F(d) \left( \frac{\varphi(mq)}{mq}  \frac{X}{d} +O(\tau(qm))\right).
    \]
    To estimate the error term, we use that $F$ is supported on cube-free numbers. Write $d=a^2b$ with $a$ and $b$ square-free and coprime. Then $F(d)=F(a^2) F(b)$ and 
    \[
     |F(a^2)| = \prod_{p\mid a} |g(p)| \leq \tau(a)^{O_C(1)}  , \qquad |F(b) | = \prod_{p\mid b} |g(p)-1| \leq  \frac{{\tau(b)^{O_C(1)}}}{b}.
    \]
    Hence 
    \begin{align}\label{eq:F_estimate}
     \tau(qm)   \sum_{d\leq X} |F(d)| &\ll \tau(qm) \sum_{a\leq X^{1/2}} {\tau(a)}^{O_C(1)} \sum_{b\leq X/a^2} \frac{\tau(b)^{O_C(1)}}{b}\ll \tau(qm) X^{1/2+\varepsilon} .    
    \end{align}
    
    Coming to the main term, we complete the sum 
    \[
    \frac{\varphi(qm)}{qm} \sum_{\substack{d\leq X \\ (d,qm)=1}}  \frac{F(d)}{d}  = \frac{\varphi(qm)}{qm}  \sum_{\substack{d = 1 \\ (d,qm)=1}}^\infty  \frac{F(d)}{d} + O\left( \sum_{\substack{d >X }}  \frac{F(d)}{d} \right).
    \]
    Using partial summation and \eqref{eq:F_estimate}, the tail is bounded by 
    \[
    \sum_{\substack{d >X }}  \frac{|F(d)|}{d} \ll  X^{-1/2+\varepsilon}.
    \]
    Finally, we write the complete sum as an Euler product, to get
    \begin{align*}
   \frac{\varphi(qm)}{mq}  \sum_{\substack{d= 1 \\ (d,qm)=1}}^\infty  \frac{F(d)}{d}  &= \prod_{p\mid qm} \frac{p-1}{p} \prod_{p\nmid qm} \left(1+\frac{F(p)}{p} +\frac{F(p^2)}{p^2}\right) \\
    &= \prod_{p\mid qm} \left(1+\frac{g(p)}{p}\right)^{-1} \prod_{p}\left(1-\frac{1}{p}\right)\left(1+\frac{g(p)}{p}\right).  
    \end{align*}
    Multiplying this with $\frac{1}{\varphi(q)} X$ gives the result. 
\end{proof}

\section{Evaluating the counting function}

We now return to the study of $M_{\Phi} (\boldsymbol{X})$ given by \Cref{lemma:M_second}.  We show that the contribution coming from $M_{\Phi}^{\chi_1} $ is negligible. 
\begin{lemma} \label{lemma:M1small}
Let $X_i, X_j, X_k, X_{-1} \in \RR_{>0}$.  Let $M_{\Phi, b_i, b_j, b_k}^{\chi_1}(X_{-1})$ be given by \eqref{eq:M_chi1}. For any $\varepsilon>0$, we have 
\begin{align*}
    \sum_{\substack{b_x \leq X_x \text{ for }x \in \{i,j,k\} }}  \mu^2(b_ib_jb_k)\big |  M^{\chi_1}_{ \Phi, b_i, b_j, b_k}( X_{-1}) \big| \ll X_iX_jX_kX_{-1}^{1/2+\varepsilon}.
\end{align*}
\end{lemma}
\begin{proof}
For an integer $n$ and some $T\in \RR_{>0}$, let 
\[
n_{|\leq T} := \prod_{\substack{p\mid n \\ p\leq T}} p
\]
Then 
\[
(v_{-1}', 2b_ib_jb_k) =1 \iff (v_{-1}', (2b_ib_jb_k)_{|\leq X_{-1}} ) =1 
\]
for $v_{-1}'\leq X_{-1}$.  By \Cref{lem:square_free_character_sum_g}, we have 
 \begin{align*}
     \mu^2(b_ib_jb_k) \big |  M^{\chi_1}_{ \Phi, b_i, b_j, b_k}( X_{-1}) \big|  &\ll  \mu^2(b_ib_jb_k) X_{-1}^{1/2+\varepsilon} \tau\big((b_ib_jb_k)_{|\leq X_{-1}}\big)  \prod_{p\mid b_ib_jb_k} \left(1+\chi_1(p)\right)\\
     &\ll X_{-1}^{1/2+\varepsilon} \prod_{x\in \{i,j,k\}} \Bigg(\tau({b_x}_{|\leq X_{-1}}) \prod_{p\mid b_x} (1+\chi_1(p))\Bigg)
 \end{align*}
 We show
 \[
\sum_{n\leq Y} \tau(n_{|\leq T})\prod_{p\mid n} \left(1+\chi_1(p)\right)\ll Y \log T, 
 \]
 from which the result of the lemma easily follows by setting $Y=X_i,X_j,X_k$, $T=X_{-1}$. Let 
 \[
 f(n):= \tau(n_{|\leq T}) \prod_{p\mid n} \left(1+\chi_1(p)\right).
 \]
 We define an auxiliary multiplicative function $g$ via 
 \[
 g(p^a) = \begin{cases}
 0 &\text{ if }a=1, p>T, \\
 -\chi_1(p)^2 &\text{ if }a=2, p>T, \\
 0 &\text{ if }a\geq 3, p>T, 
 \end{cases} \quad \text{ and } \quad 
g(p^a) = \begin{cases}
     1+\chi_1(p) & \text{ if }a=1, p\leq T, \\
- \chi_1(p) -2\chi_1(p)^2&\text{ if }a=2, p\leq T, \\
     0& \text{ if }a\geq 3, p\leq T. \\
     \end{cases}
 \]
This function satisfies $f = 1\ast \chi_1\ast g$. Indeed, for $p>T$
\begin{align*}
 (1\ast \chi_1 \ast g)(p^a) &= (1\ast \chi_1)(p^{a}) -  ((1\ast \chi_1)(p^{a-2}) \chi_1(p^2)  = \sum_{j=0}^a \chi_1(p^j) - \chi_1(p^2) \sum_{j=0}^{a-2} \chi_1(p^j)\\
 &= 1+\chi_1(p) = f(p^a),
  \end{align*}
 and for $p\leq T$
  \begin{align*}
 (1\ast \chi_1 \ast g)(p^a) &= \sum_{j=0}^a \chi_1(p^j) + (1+\chi_1(p))\sum_{j=0}^{a-1} \chi_1(p^j)  - (\chi_1(p) +2\chi(p^2) )\sum_{j=0}^{a-2} \chi_1(p^j)\\
 &= 2(1+\chi_1(p)) = f(p^a),
  \end{align*}
and thus $f = 1\ast \chi_1\ast g$. Substituting gives 
\begin{align*}
    \sum_{n\leq Y} f(n) = \sum_{n\leq Y} \sum_{d\mid n} g(d) (1\ast \chi_1)(n/d) = \sum_{d\leq Y} g(d) \sum_{n\leq Y/d} (1\ast \chi_1(d))(n)
\end{align*}
The inner sum is 
\[
\sum_{n\leq Y/d} (1\ast \chi_1(d))(n) = \sum_{n\leq Y/d}    \sum_{a\mid n} \chi_1(a) = \frac{Y}{d}\sum_{a\leq Y/d} \frac{\chi_1(a)}{a} + O(Y/d) \ll Y/d ,
\]
and hence 
\[
 \sum_{n\leq Y} f(n)  \ll Y \sum_{d =1 }^\infty \frac{|g(d)|}{d}.
\]
Expressing the infinite series as an Euler-product and using our definition of $g$, we have
\begin{align*}
    \sum_{d =1 }^\infty \frac{|g(d)|}{d}  = \prod_{p\leq T} \left( 1+\frac{1+\chi_1(p)}{p} + O(p^{-2})\right)\cdot \prod_{p\geq T} \left( 1 + O(p^{-2})\right) \ll \log T,
\end{align*}
by Mertens' theorem.
\end{proof}

Combining the previous \Cref{lemma:M1small} and \Cref{lemma:M_second},  we have 
\begin{align*}
  M_{\Phi} ( X_i, X_j, X_k , X_{-1})  &= \frac{1}{2} \sum_{ \substack{(\varepsilon_i, \varepsilon_j, \varepsilon_k )\in ((\ZZ/8\ZZ)^\times)^3\\ (\varepsilon_i\varepsilon_j, \varepsilon_j\varepsilon_k)\equiv (\alpha', \beta') \bmod 8}}
N_{\boldsymbol{\mu}, \boldsymbol{\varepsilon}}(X_i, X_j, X_k, X_{-1})   + O\big(X_iX_jX_kX_{-1}^{1/2+\varepsilon}\big)  \numberthis \label{eq:analytic_MPhi_new}
\end{align*}
with $N_{\boldsymbol{\mu}, \boldsymbol{\varepsilon}}(\boldsymbol{X})$ given by \eqref{eq:def_N_main}. The following proposition is the main counting result.

\begin{proposition}\label{prop_main_N}
    Let $ N_{\boldsymbol{\mu}, \boldsymbol{\varepsilon}}(\boldsymbol{X}) = N_{\boldsymbol{\mu}, \boldsymbol{\varepsilon}}(X_i, X_j, X_k, X_{-1})$ be given by \eqref{eq:def_N_main}, then 
    \[
     N_{\boldsymbol{\mu}, \boldsymbol{\varepsilon}}(\boldsymbol{X})  =  \frac{u_{\boldsymbol{\mu}} (\varepsilon_i, \varepsilon_j, \varepsilon_k)+ u_{\boldsymbol{\mu}}(1,1,1)}{1024} \prod_{p>2} \left(1-\frac{1}{p}\right)^4 \left(1+\frac{4}{p}\right)X_{-1}X_iX_jX_k  + E,
    \]
    for an error term
    \[
    E = O\left(X_iX_jX_kX_{-1} \cdot \min \{X_i, X_j, X_k, X_{-1}\}^{-\Delta} \right).
    \]
    for some $\Delta>0$.
\end{proposition}

The definition of the function $N_{\boldsymbol{\mu}, \boldsymbol{\varepsilon}}$ is symmetric in $i,j,k$. Hence without loss of generality, we always assume
    \begin{align}
      X_i\leq X_j\leq X_k. \label{eq:size_comp}    
    \end{align}

    The proof of \Cref{prop_main_N} is split into two main cases with two, respectively three, sub-cases and carried out in the following sections. The cases we consider are 
   \begin{alignat*}{2}
       &\bullet \;  \; &&\text{Case $X_j\geq X_k^{1/5}$:} \hspace{11cm} \\
       & && \quad - \, \text{ Sub-case: $X_{-1}< X_k$  in \Cref{sec:largeXj_smallX_minus};} \\
       & && \quad - \, \text{ Sub-case: $X_{-1}\geq X_k$ in \Cref{ssec:X1Xjlarge};} \\
       & \bullet \; \; &&\text{Case $X_j< X_k^{1/5}$: } \numberthis \label{list:cases}\\ 
       & && \quad - \, \text{ Sub-case: $X_{-1} \geq  X_k$ in \Cref{seq:Xjsmall_Xminus_large};} \\
       & && \quad - \, \text{ Sub-case: $X_j< X_{-1} \leq X_k$ in  \Cref{ssec:middle};} \\
       & && \quad - \, \text{ Sub-case: $X_{-1} \leq X_j $ in \Cref{ssec:last}.} 
   \end{alignat*}
Instead of the exponent $1/5$, any fixed exponent less than $1/4$ could be taken.

    \subsection{The case $X_j\geq  X_k^{1/5}$}
Assume 
\begin{align}
    X_j\geq X_k^{1/5} \label{eq:X_j>X_k large}.
\end{align}
Let $V$ be an auxiliary parameter given by
\begin{align}
\label{eq:def_V}
V=X_k^{\delta}    
\end{align}
for a fixed small $\delta\in (0,1)$ to be chosen later. 
We split the summation of odd square-free $(d_i, d_j, d_k, f_i, f_j, f_k)$ into the following three regions. \newline \ \vspace{0.3cm} \ 
\noindent
\begin{minipage}{\linewidth} \label{eq:1}
\begin{enumerate}
\vspace{0.3cm}
    \item[\fbox{\parbox{0.6cm}{\centering A}}:] \label{eq:regions}
    There exist $x,y \in \{i,j,k\}$, $x\neq y$, such that $d_x, f_y > V$ ;
     \item[\fbox{\parbox{0.6cm}{\centering B1}}:]
    $f_i, f_j, f_k \leq V$; \fbox{\parbox{0.6cm}{\centering B2}}:
    $d_i, d_j, d_k \leq V$;
    \item[\fbox{\parbox{0.6cm}{\centering C}}:]
    $f_x, d_x \leq V$ for some $x\in \{i,j,k\}$, but $(d_i, d_j, d_k, f_i, f_j, f_k)$ does not lie \\ \textcolor{white}{..} in region \fbox{A}, \fbox{B1}, \fbox{B2}.
\end{enumerate}
\end{minipage}
\hfill
 We estimate the contributions coming from region \fbox{A} using the result on bilinear sums from \Cref{lemma:large_sieve};
we show that the contribution coming from the region \fbox{C} is negligible; the regions \fbox{B1} and \fbox{B2} are estimated pointwise using  \Cref{lem:square_free_character_sum_g}, whenever $(f_i, f_j, f_k)\neq (1,1,1)$, $(d_i, d_j, d_k)\neq (1,1,1)$, we obtain a non-trivial saving.\\
To obtain the sufficient error terms, we consider the two sub-cases $X_{-1}>X_k$ and $X_{-1}\leq X_k$. Accordingly, the order of summation over $v_{-1}$ in \eqref{eq:def_N_main} is performed first or respectively last. We address the latter case first.

\subsubsection{Sub-case $X_{-1} \leq X_k$} \label{sec:largeXj_smallX_minus}

We perform the summation over $v_{-1}$ last, at the cost of picking up divisor functions depending on $v_{-1}$ in the inner sums, which causes no difficulty as any logarithms in $X_{-1}$ will be absorbed by sufficient saving in $X_k$. We have 
\begin{align} \label{eq:N_as_T}
    N_{\boldsymbol{\mu}, \boldsymbol{\varepsilon}}(\boldsymbol{X}) = \sum_{v_{-1}\leq X_{-1}} \mu^2(2v_{-1}) T_{v_{-1}, \boldsymbol{\varepsilon}, \boldsymbol{\mu}} (X_i, X_j,X_k) +O(X_{-1}X_k^{1+\varepsilon}),
\end{align}
where 
\begin{align}
T_{v_{-1}, \boldsymbol{\varepsilon}, \boldsymbol{\mu}} (X_i, X_j, X_k) := \hspace{-.5cm}\sum_{\substack{\textcolor{white}{.}\ \\ \ d_i f_i \leq X_i , d_jf_j\leq X_j, d_k f_k \leq X_k  \\ d_xf_x =\varepsilon_x \bmod 8 , x\in \{i,j,k\} \\ (d_if_id_jf_jd_kf_k, v_{-1}) = 1}} \hspace{-.5cm}{\mu^2(d_if_id_jf_jd_kf_k)} u_{\boldsymbol{\mu}}(d_i, d_j, d_k)    \legendre{f_i }{d_jd_k} \legendre{f_j}{d_i d_k} \legendre{f_k}{d_i d_j}\,. \label{eq:def_tv}
\end{align}
Here, we dropped the condition \eqref{eq:condition_G} that is present in $N_{\boldsymbol{\mu}, \boldsymbol{\varepsilon}}(\boldsymbol{X})$. This is justified, since tuples, where $\mathcal{G}$ fails are trivially bounded by
\[
O( X_k^{1+\varepsilon}) 
\]
which is the error term in \eqref{eq:N_as_T}. As $X_j>X_k^{1/5}$ this error is neglibale. We define
\begin{align*}
&T_{v_{-1}, \boldsymbol{\mu}, \boldsymbol{\varepsilon}}^{\text{A}}(X_i, X_j, X_k) := \\
&\qquad  \sum_{\substack{d_i f_i \leq X_i , d_jf_j\leq X_j, d_k f_k \leq X_k  \\ d_xf_x =\varepsilon_x  \bmod 8, x\in \{i,j,k\} \\ (d_if_id_jf_jd_kf_k, 2v_{-1}) = 1 \\ (d_i, d_j, d_k, f_i, f_j, f_k) \text{ lies in region A}}} {\mu^2(d_if_id_jf_jd_kf_k) } u_{\boldsymbol{\mu}}(d_i, d_j, d_k)    \legendre{f_i }{d_jd_k} \legendre{f_j}{d_i d_k} \legendre{f_k}{d_i d_j}\,.   
\end{align*}
and similarly $T_{v_{-1}, \boldsymbol{\mu}, \boldsymbol{\varepsilon}}^{\text{B1}},T_{v_{-1}, \boldsymbol{\mu}, \boldsymbol{\varepsilon}}^{\text{B2}}$ and  $T_{v_{-1}, \boldsymbol{\mu}, \boldsymbol{\varepsilon}}^{\text{C}}$. 

We proceed to establish an error estimate for the region  \fbox{A}.
\begin{lemma}\label{lemma:regionA}
    Let $X_i, X_j, X_k\in \RR_{>0}$ such that \eqref{eq:size_comp} and \eqref{eq:X_j>X_k large} hold. Let $V\in \RR_{>0}$ be given by \eqref{eq:def_V}. In the notation above, for any $\varepsilon>0$ we have
    \[
    T_{v_{-1}, \boldsymbol{\mu}, \boldsymbol{\varepsilon}}^{\text{A}}(X_i, X_j, X_k) \ll X_kX_jX_i V^{-1/6+\varepsilon} \,. 
    \]
\end{lemma}
\begin{proof}
Region \fbox{A} consists of those tuples with $d_x, f_y>V$ for some $x,y \in \{i,j,k\}$, $x\neq y$. We first consider the case $d_i, f_j>V$: 
We move the sum over the $d_i, f_j>V$ to the inside. Set
\[ 
\alpha_{d_i } = \boldsymbol{1}_{\left\{\substack{d_i \equiv f_i\varepsilon_i \bmod 8\\ (d_i, 2f_i d_j f_k d_k v_{-1})=1}\right\} } \mu^2(d_i) \legendre{f_jf_k}{d_i} \; \; , \qquad \beta_{f_j } = \boldsymbol{1}_{\left\{\substack{f_j \equiv d_j\varepsilon_j \bmod 8\\ (f_j, 2f_i d_j f_kd _k v_{-1})=1}\right\} } \mu^2(f_j) \legendre{f_j}{d_id_k}
\]
Using \Cref{lemma:large_sieve}, we have
\begin{align*}
&\Bigg|\sum_{\substack{d_i f_i \leq X_i , d_jf_j\leq X_j, d_k f_k \leq X_k  \\ d_xf_x =\varepsilon_x  \bmod 8, x\in \{i,j,k\} \\ (d_if_id_jf_jd_kf_k, 2v_{-1}) = 1 \\ d_i, f_j>V}} {\mu^2(d_if_id_jf_jd_kf_k) } u_{\boldsymbol{\mu}}(d_i, d_j, d_k)     \legendre{f_i }{d_jd_k} \legendre{f_j}{d_i d_k} \legendre{f_k}{d_i d_j} \Bigg|\\
&\ll \sum_{d_kf_k \leq X_k} \sum_{\substack{f_i\leq X_i/V\\ d_j\leq X_j/V}}  \Bigg|  \sum_{\substack{V<d_i<X_i/f_i\\ V<f_j<X_j/d_j}} \alpha_{d_i} \beta_{f_j} \legendre{f_j}{d_i} \Bigg| \\
&\ll X_k \log X_k  \sum_{\substack{f_i\leq X_i/V\\ d_j\leq X_j/V}} \frac{X_i X_j}{f_i d_j}\left(\frac{f_i^{1/6}}{X_i^{1/6}}+\frac{d_j^{1/6}}{X_j^{1/6}}\right) \log(X_i X_j)^{7/6}\\
&\ll X_k (\log X_k)  X_i X_j \log(X_i X_j)^{13/6} V^{-1/6} \ll X_k X_i X_j V^{-1/6+\varepsilon},
\end{align*}
where we used that $\log X_k \log(X_i X_j)^{13/6} = O(V^\varepsilon)$ by \eqref{eq:def_V}. The remaining choices $x,y\in \{i,j,k\}$ used in the definition of region \fbox{A} are treated analogously, yielding the same bound.
\end{proof}
We now turn to the region \fbox{C}. 
\begin{lemma}\label{lemma:regionB}
     Let $X_i, X_j, X_k\in \RR_{>0}$ such that \eqref{eq:size_comp} holds. Let $V\in \RR_{>0}$ be given by \eqref{eq:def_V}. For any $\varepsilon>0$, we have
    \[
       T_{v_{-1}, \boldsymbol{\mu}, \boldsymbol{\varepsilon}}^{\text{C}}(X_i, X_j, X_k) \ll {X_k}V^{4+\varepsilon} .
    \]
\end{lemma}
\begin{proof}
Suppose $d_x, f_x \leq V$ for some $x\in \{i,j,k\}$. Since we are not in region \fbox{B1} or \fbox{B2}, there exist indices $y_1,y_2 \in \{i,j,k\}\setminus \{x\}$ such that $d_{y_1}$ and $ f_{y_2} >V$. If $y_1\neq y_2$, we are in region \fbox{A}, hence we must have $y_1 = y_2:=y$ with $d_{y} , f_{y}>V$. This forces $d_z, f_z \leq V$ for the remaining $z \in \{i,j,k\} \setminus \{x,y\}$, for otherwise we would again be in region \fbox{A}. Hence the contribution is trivially bounded by 
\[
O(V^4 X_{y}^{1+\varepsilon})
\]
Since $X_k$ is the largest parameter among $X_i, X_j, X_k$ and $V = X_k^{\delta}$, the result follows.
\end{proof}

 We now turn to the regions \fbox{B1}, \fbox{B2}, that is tuples where $f_i, f_j, f_k\leq V$, or $d_i, d_j, d_k\leq V$ respectively. We use pointwise estimates to show that the contribution gives an error term whenever there is a non-principal character.  Consider first the region \fbox{B1}. We fix $f_i, f_j, f_k\leq V$ and define the sum 
 \begin{align*}
     &U_{f_i, f_j, f_k, v_{-1}} := \sum_{\substack{ d_i \leq X_i/f_i \\ (d_i, f_if_jf_kv_{-1})=1 \\ d_i \equiv f_i\varepsilon_i \bmod 8} }\sum_{\substack{ d_j \leq X_j/f_j \\ (d_j, f_if_jf_kv_{-1})=1 \\ d_j \equiv f_j\varepsilon_j \bmod 8} } \sum_{\substack{ d_k \leq X_k/f_k \\ (d_k, f_if_jf_kv_{-1})=1 \\ d_k \equiv f_k\varepsilon_k \bmod 8} } \mu^2(2d_id_jd_k)  \legendre{f_if_j} {d_k}\legendre{f_if_k}{d_j} \legendre{f_jf_k}{d_i} \,.
 \end{align*}
Similarly, for \fbox{B2}, we fix $d_i, d_j, d_k\leq V$ odd.  We define 
  \begin{align*}
     &U'_{d_i, d_j, d_k, v_{-1}} := \sum_{\substack{ f_i \leq X_i/d_i \\ (f_i, d_id_jd_kv_{-1})=1 \\ f_i \equiv d_i\varepsilon_i \bmod 8} }\sum_{\substack{ f_j \leq X_j/d_j \\ (f_j, f_id_jd_kv_{-1})=1 \\ f_j \equiv d_j\varepsilon_j \bmod 8} } \sum_{\substack{ f_k \leq X_k/d_k \\ (f_k, f_id_jd_kv_{-1})=1 \\ f_k \equiv d_k\varepsilon_k \bmod 8} } \mu^2(2f_if_jf_k) \legendre{f_i}{d_jd_k} \legendre{f_j}{d_id_k} \legendre{f_k} {d_id_j} 
 \end{align*}
 Notice that $u_{\boldsymbol{\mu}}(d_i, d_j, d_k) = u_{\boldsymbol{\mu}}(\varepsilon_i f_i, \varepsilon_j f_j, \varepsilon_k f_k) $ for $d_xf_x\equiv \varepsilon_i \bmod 8$ and $d_x, f_x$ odd. 
 With this notation we have 
 \begin{align}
       T_{v_{-1}, \boldsymbol{\mu}, \boldsymbol{\varepsilon}}^{B_1}(X_i, X_j, X_k) = \sum_{\substack{f_i, f_j, f_k\leq V\\ (f_if_jf_k , v_{-1})=1}}  \mu^2(2f_if_jf_k)  u_{\boldsymbol{\mu}}(\varepsilon_i f_i, \varepsilon_j f_j, \varepsilon_k f_k)U_{f_i, f_j, f_k, v_{-1}}\label{eq:dec_TB1}
      \end{align}
      and
      \begin{align}
          T_{v_{-1}, \boldsymbol{\mu}, \boldsymbol{\varepsilon}}^{B_2}(X_i, X_j, X_k) = \sum_{\substack{d_i, d_j, d_k\leq V\\ (d_id_jd_k , v_{-1})=1 }}    \mu^2(2d_id_jd_k) u_{\boldsymbol{\mu}}(d_i, d_j, d_k) U'_{d_i, d_j, d_k, v_{-1}}.\label{eq:dec_TB2}
 \end{align}
We establish a pointwise bound for $U$ and $U'$.
 \begin{lemma} \label{lemma:bound_U}
     In the notation above, whenever $(f_i, f_j, f_k) \neq (1,1,1)$ we have
     \[
     U_{f_i, f_j, f_k, v_{-1}}  \ll  \tau(v_{-1})V^{1+\varepsilon} \frac{X_iX_j X_k }{f_if_jf_k} \left(X_k^{-1/2} f_k^{1/2} + X_j^{-1/2} f_j^{1/2}\right)  .
     \]
     Similarly, whenever $(d_i, d_j, d_k)\neq (1,1,1)$, we have
         \[
     U'_{d_i, d_j, d_k, v_{-1}}  \ll \tau(v_{-1}) V^{1+\varepsilon} \frac{X_iX_j X_k }{d_id_jd_k} \left( X_k^{-1/2} d_k^{1/2} + X_j^{-1/2} d_j^{1/2}\right)  
     \]
 \end{lemma}
 \begin{proof}
 Let $(f_i, f_j, f_k)\neq (1,1,1)$. We begin with the case $f_i \neq 1$ or $f_j\neq 1$. In the $d_k$ sum in the definition of $ U_{f_i, f_j, f_k, v_{-1}}$, we have 
 \begin{align*}
    \sum_{\substack{ d_k \leq X_k/f_k \\ (d_k,  f_k d_id_jv_{-1} )=1 \\ d_k \equiv f_k\varepsilon_k \bmod 8} } \mu^2(d_k)  \legendre{f_if_j}{d_k}  .
 \end{align*}
 Using \Cref{lemma:properties_u} $(i)$ and \Cref{lem:square_free_character_sum_g} with $g(p)=1$, as well as that the bound on the modulus of the character $1\neq f_if_j \leq (8V)^2$, this is bounded by 
 \[
 O( \tau(d_id_jf_kv_{-1}) \left(X_k/f_k\right)^{1/2+\varepsilon'} \cdot V^{1+\varepsilon'}),
 \]
for any $\varepsilon'>0$. Summing this trivially over the $d_j$- and $d_i$-sums in $U_{f_i, f_j, f_k, v_{-1}}$ yields 
 \[
 U_{f_i, f_j, f_k, v_{-1}} \ll \tau(v_{-1})\frac{({X_iX_j})^{1+\varepsilon'}   X_k^{1/2+\varepsilon'}}{f_if_j f_k^{1/2}}  V^{1+\varepsilon'} .
 \]
 On choosing $\varepsilon'$ small enough in terms of $\delta$ from \eqref{eq:def_V}, we obtain
 \[
 U_{f_i, f_j, f_k, v_{-1}} \ll \tau(v_{-1}) \frac{({X_iX_j})   X_k^{1/2}}{f_if_j f_k^{1/2}}  V^{1+\varepsilon} ,
 \]
 for any $\varepsilon>0$.  In the case where $f_i=f_j=1$ and $f_k\neq 1$, we apply the above argument to the $d_j$ sum and obtain 
 \[
 U_{f_i, f_j, f_k, v_{-1}} \ll \tau(v_{-1})\frac{X_j^{1/2}  (X_i X_k)}{f_if_kf_j^{1/2} }  V^{1+\varepsilon}  .
 \]
 The proof for $U'$ is the same by symmetry.
 \end{proof}
The preceding lemma allows us to evaluate $T_{v_{-1}}^{B_1}$ and $T_{v_{-1}}^{B_2}$.
 \begin{lemma}\label{lemma:regionC} Let $X_i, X_j, X_k\in \RR_{>0}$ satisfying \eqref{eq:size_comp} and \eqref{eq:X_j>X_k large}. Then 
     \[
     T_{v_{-1}, \boldsymbol{\mu}, \boldsymbol{\varepsilon}}^{B_1 }(X_i, X_j, X_k) = u_{\boldsymbol{\mu}}(\varepsilon_i, \varepsilon_j, \varepsilon_k) U_{1,1,1, v_{-1}} + O(\tau(v_{-1})V^{5+\varepsilon} X_iX_kX_j^{1/2});
    \]
    and 
     \[
      T_{v_{-1}, \boldsymbol{\mu}, \boldsymbol{\varepsilon}}^{B_2}(X_i, X_j, X_k) = u_{\boldsymbol{\mu}}(1,1,1)U'_{1,1,1, v_{-1}} + O(\tau(v_{-1})V^{5+\varepsilon} X_iX_kX_j^{1/2}).
     \]
 \end{lemma}
\begin{proof}
    Starting from \eqref{eq:dec_TB1} and \eqref{eq:dec_TB2}, we apply \Cref{lemma:bound_U} to the individual terms $U_{f_i, f_j, f_k, v_{-1}}$ and $U'_{d_i, d_j, d_k, v_{-1}}$ whenever $(f_i, f_j, f_k)\neq (1,1,1)$,  $(d_i, d_j, d_k)\neq (1,1,1)$, respectively. On taking $\varepsilon$ in \Cref{lemma:bound_U} small enough, using \eqref{eq:def_V} and \eqref{eq:X_j>X_k large}, the result follows by summing trivially over all $f_i, f_j, f_k<V$, $d_i, d_j, d_k<V$ respectively. 
\end{proof}
We end this section by giving a first proof of \Cref{prop_main_N} under the size assumptions \eqref{eq:size_comp}, \eqref{eq:X_j>X_k large} and $X_{-1}\leq X_k$.
\begin{lemma} \label{lemma:main_1}
    Let $X_i, X_j, X_k, X_{-1}\in \RR_{>0}$ such that $X_i \leq X_j \leq X_k$, $X_{-1}\leq X_k$ and $X_j\geq X_k^{1/5}$. Then \Cref{prop_main_N} holds with an error term 
    \[
    E \ll X_{-1}X_iX_jX_k \left( X_k^{-1/510 +\varepsilon}+ X_i^{-1/2+\varepsilon} + X_{-1}^{-1/2+\varepsilon}\right) ,
    \]
    for any $\varepsilon>0$.
\end{lemma}
\begin{proof}
    Recall \eqref{eq:N_as_T}. Combining the estimates from \Cref{lemma:regionA}, \Cref{lemma:regionB} and \Cref{lemma:regionC}, as well as the fact that the regions \fbox{B1} and \fbox{B2} intersect in at most $O(V^6)$ indices, we find 
    \begin{align*}
   &N_{\boldsymbol{\mu}, \boldsymbol{\varepsilon}}(\boldsymbol{X}) = \sum_{v_{-1}\leq X_{-1}} \mu^2(2v_{-1}) \Big(u_{\boldsymbol{\mu}}(\varepsilon_i, \varepsilon_j, \varepsilon_k) U_{1,1,1, v_{-1}}+ u_{\boldsymbol{\mu}}(1,1,1) U'_{1,1,1, v_{-1}} +\\
    &\hspace{2.5cm} O\big(X_{k} X_jX_i V^{-1/6+\varepsilon}+ X_kV^{4+\varepsilon}+\tau(v_{-1})V^{5+\varepsilon} X_iX_kX_j^{1/2} +V^6\big)\Big)  + O(X_k^{1+\varepsilon}X_{-1})  \\
    &=(u_{\boldsymbol{\mu}}(\varepsilon_i, \varepsilon_j, \varepsilon_k) + u_{\boldsymbol{\mu}}(1,1,1) ) \sum_{\substack{v_{-1}\leq X_{-1} \\ f_{x} \leq X_x \text{ for }x=i,j,k \\
    f_x\equiv \varepsilon_x \bmod 8}} \mu^2(2v_{-1}f_if_jf_k) \\
    &\qquad  + O\left(X_{-1}X_kX_jX_i V^{-1/6+\varepsilon} + X_{-1}X_kV^{4+\varepsilon}+ X_{-1} X_kX_j^{1/2} X_iV^{5+\varepsilon}+X_{-1}V^6\right) . \numberthis \label{eq:njagjnagf}
    \end{align*}
    Recalling \eqref{eq:def_V} and using \eqref{eq:X_j>X_k large} as well as $X_{-1}\leq X_k$, the error term above is bounded by 
    \[
   O\left( X_{-1}X_jX_i \left( X_k^{1-\delta/6+\varepsilon} + X_k^{4/5+4\delta+\varepsilon }+X_k^{9/10+5\delta +\varepsilon}     +X_k^{6\delta}\right) \right).
    \]
    Choosing $\delta = 1/51$, this is 
    \[
    O\left(X_{-1}X_iX_jX_k^{1-1/510+\varepsilon}\right), \numberthis \label{eq:error_term1}
    \]
    for any $\varepsilon>0$. Lastly, we evaluate the main term 
    \[
    \sum_{\substack{v_{-1}\leq X_{-1} \\ f_{x} \leq X_x \text{ for }x=i,j,k \\
    f_x\equiv \varepsilon_x \bmod 8}} \mu^2(2v_{-1}f_if_jf_k) .
    \]
    We do so by applying \Cref{lemma:square-free-asymp} successively in decreasing order of parameters $X_{k}, X_{-1}, X_j, X_i$. Beginning with the largest, in our case $X_k$, an application of \Cref{lemma:square-free-asymp} to the sum over $f_k$ with $q=8$, $m=2v_{-1}f_if_j$ and $g(p)=1$ gives 
    \begin{align*}
     \sum_{\substack{v_{-1}\leq X_{-1} \\ f_{x} \leq X_x \text{ for }x=i,j,k \\
    f_x\equiv \varepsilon_x \bmod 8}} \mu^2(2v_{-1}f_if_jf_k) &= \frac{X_k}{8} \prod_{p>2} \left(1-\frac{1}{p^2}\right)  \sum_{\substack{ f_{i} \leq X_i\\ f_j\leq X_j \\ v_{-1}\leq X_{-1} \\ (f_i, f_j)\equiv (\varepsilon_i, \varepsilon_j)\bmod 8}} \mu^2(2f_if_jv_{-1}) \prod_{p\mid f_if_jv_{-1}} \frac{p}{p+1} \\
    &\quad +O\left( X_{k}^{1/2+\varepsilon}({X_iX_jX_{-1}})^{1+\varepsilon}\right). 
    \end{align*}
     Successively applying \Cref{lemma:square-free-asymp} three more times in decreasing order of $X_{-1}, X_j, X_i$, we obtain
    \begin{align*}
      \sum_{\substack{v_{-1}\leq X_{-1} \\ f_{x} \leq X_x \text{ for }x=i,j,k \\
    f_x\equiv \varepsilon_x \bmod 8}} \mu^2(2v_{-1}f_if_jf_k) &=  \frac{X_{i}X_jX_kX_{-1}}{1024} \prod_{p>2} \left(1-\frac{1}{p}\right)^{4} \left(1+\frac{4}{p}\right) \\
    &\quad+ O\left(X_{-1}X_iX_jX_k \left( X_i^{-1/2+\varepsilon}+ X_{-1}^{-1/2+\varepsilon}\right)\right).
    \end{align*}
    Combining this with \eqref{eq:njagjnagf} and the error term obtained above in \eqref{eq:error_term1} gives the result.
\end{proof}

\subsubsection{Sub-case $X_{-1} \geq X_k $} \label{ssec:X1Xjlarge}
We illustrate how the above argument is adapted to the case $X_{-1}> X_{k}$, still under the assumption \eqref{eq:size_comp} and \eqref{eq:X_j>X_k large}.
\begin{lemma}  \label{lemma:main_2}
    Let $X_i, X_j, X_k, X_{-1}\in \RR_{>0}$ such that $X_i\leq X_j \leq X_k\leq X_{-1}$ and $X_j\geq X_k^{1/5}$. Then \Cref{prop_main_N} holds with an error term 
    \[
    E \ll X_{-1}X_iX_jX_k \left( X_k^{-1/510 +\varepsilon}+ X_i^{-1/2+\varepsilon} \right) ,
    \]
    for any $\varepsilon>0$.
\end{lemma}

\begin{proof}
We perform the summation over $v_{-1}$ first. We have, using for example \Cref{lemma:square-free-asymp},
\begin{align*}
\sum_{ v_{-1} \leq X_{-1}} \mu^2(2v_{-1} d_id_jd_k f_if_jf_k)   &= \mu^2(2 d_id_jd_k f_if_jf_k) \frac{X_{-1}}{2}\prod_{p>2} \left(1-\frac{1}{p^2}\right) \prod_{p\mid d_id_jd_kf_if_jf_k} \frac{p}{p+1} \\
&\qquad + O\left(X_{-1}^{1/2} \tau(d_id_jd_kf_if_jf_k)\right).    
\end{align*}
Inserting this into \eqref{eq:def_N_main} and summing the error term gives
\begin{align*}
    N_{\boldsymbol{\mu}, \boldsymbol{\varepsilon}}(\boldsymbol{X}) &= \frac{X_{-1}}{2}  \prod_{p>2} \left(1-\frac{1}{p^2}\right) T_{\boldsymbol{\varepsilon}, \boldsymbol{\mu}}(X_i, X_j, X_k) +O\left(X_{-1}^{1/2} (X_iX_jX_k)^{1+\varepsilon}\right) \\
    &= \frac{X_{-1}}{2}  \prod_{p>2} \left(1-\frac{1}{p^2}\right) T_{\boldsymbol{\varepsilon}, \boldsymbol{\mu}}(X_i, X_j, X_k) +O\left(X_{-1}^{1/2+\varepsilon} X_iX_jX_k \right)  \numberthis \label{eq:N_v-1_sum_performed} 
\end{align*}
for any $\varepsilon>0$ and with
\begin{align}
    T_{\boldsymbol{\varepsilon}, \boldsymbol{\mu}}(X_i, X_j, X_k) &=\sum_{\substack{d_i f_i \leq X_i , d_jf_j\leq X_j, d_k f_k \leq X_k  \\  d_xf_x=\varepsilon_x \bmod 8 \text{ for } x\in\{i,j,k\} }} \Bigg( {\mu^2( d_if_id_jf_jd_kf_k)} u_{ \boldsymbol{\mu}}(d_i, d_j, d_k)  \nonumber \\
 &\hspace{2cm}   \legendre{f_i }{d_jd_k} \legendre{f_j}{d_i d_k} \legendre{f_k}{d_i d_j}   \prod_{p\mid d_id_jd_kf_if_jf_k} \frac{p}{p+1}\Bigg)\,. \label{eq:def_T}
\end{align}
We now split this into the regions $ T_{\boldsymbol{\varepsilon}, \boldsymbol{\mu}}^A(X_i, X_j, X_k)$,  $T^{B_{1,2}}_{\boldsymbol{\varepsilon}, \boldsymbol{\mu}}(X_i, X_j, X_k)$,  $T^{C}_{\boldsymbol{\varepsilon}, \boldsymbol{\mu}}(X_i, X_j, X_k)$ as before in \Cref{sec:largeXj_smallX_minus}. On replacing $\alpha_{d_i}, \beta_{f_j}$ in the proof of \Cref{lemma:regionB} by 
\[ 
\alpha_{d_i } = \boldsymbol{1}_{\left\{\substack{d_i \equiv f_i\varepsilon_i \bmod 8\\ (d_i, 2f_id_jf_kd_k)=1}\right\} } \mu^2(d_i) \prod_{p\mid d_i} \frac{p}{p+1}\; \; , \quad \beta_{f_j } = \boldsymbol{1}_{\left\{\substack{f_j \equiv d_j\varepsilon_j \bmod 8\\ (f_j, 2f_id_if_kd_k)=1}\right\} } \mu^2(f_j)\prod_{p\mid f_j} \frac{p}{p+1},
\]
we obtain 
\[
 T_{\boldsymbol{\varepsilon}, \boldsymbol{\mu}}^A(X_i, X_j, X_k) \ll X_kX_jX_iV^{-1/6+\varepsilon}
\]
for any $\varepsilon>0$. The proof of \Cref{lemma:regionB} goes through without any changes, showing 
\[
 T_{\boldsymbol{\varepsilon}, \boldsymbol{\mu}}^A(X_i, X_j, X_k) \ll X_kV^{4+\varepsilon}.
\]
We define 
 \begin{align*}
     &U_{f_i, f_j, f_k} =  \ \sum_{\substack{ d_i \leq X_i/f_i \\ (d_i, f_if_jf_k)=1 \\ d_i \equiv f_i\varepsilon_i \bmod 8} }\sum_{\substack{ d_j \leq X_j/f_j \\ (d_j, f_if_jf_k)=1 \\ d_j \equiv f_j\varepsilon_j \bmod 8} } \sum_{\substack{ d_k \leq X_k/f_k \\ (d_k, f_if_jf_k)=1 \\ d_k \equiv f_k\varepsilon_k \bmod 8} } \mu^2(2d_id_jd_k) \\
     &\hspace{5cm}\cdot \legendre{f_if_j}{d_k} \legendre{f_if_k} {d_j}\legendre{f_jf_k} {d_i}\prod_{p\mid {d_id_jd_kf_if_jf_k}} \frac{p}{p+1},
 \end{align*}
and
  \begin{align*}
     &U'_{d_i, d_j, d_k} =\sum_{\substack{ f_i \leq X_i/d_i \\ (f_i, d_id_jd_k)=1 \\ f_i \equiv d_i\varepsilon_i \bmod 8} }\sum_{\substack{ f_j \leq X_j/d_j \\ (f_j, f_id_jd_k)=1 \\ f_j \equiv d_j\varepsilon_j \bmod 8} } \sum_{\substack{ f_k \leq X_k/d_k \\ (f_k, f_id_jd_k)=1 \\ f_k \equiv d_k\varepsilon_k \bmod 8} } \mu^2(2f_if_jf_k) \\
     &\hspace{5cm}\cdot \legendre{f_k}{d_id_j} \legendre{f_j}{d_id_k} \legendre{f_i}{d_jd_k}\prod_{p\mid {d_id_jd_kf_if_jf_k}} \frac{p}{p+1} .
 \end{align*}
 With this, we have 
\begin{align*}
      T_{\boldsymbol{\varepsilon}, \boldsymbol{\mu}}^{B_1}(X_i, X_j, X_k) = \sum_{\substack{f_i, f_j, f_k\leq V\\ (f_if_jf_k , 2)=1 }}   \mu^2(f_if_jf_k)u_{\boldsymbol{\mu}}(\varepsilon_i f_i, \varepsilon_j f_j, \varepsilon_kf_k)   U_{f_i, f_j, f_k} 
      \end{align*}
      and
      \begin{align*}
          T_{\boldsymbol{\varepsilon}, \boldsymbol{\mu}}^{B_2}(X_i, X_j, X_k) = \sum_{\substack{d_i, d_j, d_k\leq V\\ (d_id_jd_k ,2)=1}}   \mu^2(d_id_jd_k) u_{\boldsymbol{\mu}}(d_i,d_j, d_k)  U'_{d_i, d_j, d_k}.
 \end{align*}
Setting $g(p) = \frac{p}{p+1}$ and $v_{-1}=1$ in the proof of \Cref{lemma:bound_U}, we obtain the same bounds for $U_{f_i, f_j, f_k} $ and $U'_{d_id_jd_k}$ as in \Cref{lemma:bound_U} with $v_{-1}=1$. Then the proof of \Cref{lemma:regionC} shows
\[
T_{\boldsymbol{\varepsilon}, \boldsymbol{\mu}}^{B_1}(X_i, X_j, X_k) = u_{\boldsymbol{\mu}}(\varepsilon_i, \varepsilon_j, \varepsilon_k)U_{1,1,1} + O( V^{5+\varepsilon} X_kX_j^{1/2}X_i) .
\]
and 
\[
T_{\boldsymbol{\varepsilon}, \boldsymbol{\mu}}^{B_2}(X_i, X_j, X_k) = u_{\boldsymbol{\mu}}(1,1,1)U'_{1,1,1} + O( V^{5+\varepsilon} X_kX_j^{1/2}X_i).
\]
Setting $V= X_k^{1/51} $ as in the proof of \Cref{lemma:main_1} shows 
\[
T_{\boldsymbol{\varepsilon}, \boldsymbol{\mu}}(X_i, X_j, X_k) = U_{1,1,1} + U'_{1,1,1} + O\left(X_jX_iX_k^{1-\frac{1}{510} + \varepsilon}\right).
\]
Finally, the main term $$\frac{1}{2} \prod_{p>2} \left(1-\frac{1}{p^2}\right) X_{-1} (U_{1,1,1}+U'_{1,1,1})$$ 
is treated in the same way as before in the proof of \Cref{lemma:main_1}, giving relative power saving of $X_{i}^{-1/2+\varepsilon}$ in the smallest parameter $X_i$.
\end{proof}
\subsection{The case $X_j < X_k^{1/5}$} \label{sec:smallX_j}
    
In this case we will not use any large sieve estimates, but treat the character sum directly. 
As before we have to consider the different size ranges of $X_{-1}$ with regard to $X_k$ and $X_j$. We start with the following sub-case.
\subsubsection{Sub-case $X_{-1} \geq  X_k$} \label{seq:Xjsmall_Xminus_large}
We start by summing over $v_{-1}$ square-free and coprime to $d_if_id_jf_jd_kf_k$ first. As in  \eqref{eq:N_v-1_sum_performed}, performing the summation over $v_{-1}\leq X_{-1}$, we have 
\begin{align*}
    N_{\boldsymbol{\mu}, \boldsymbol{\varepsilon}}(\boldsymbol{X})      &= \frac{1}{2} X_{-1} \prod_{p>2} \left(1-\frac{1}{p^2}\right) T_{\boldsymbol{\varepsilon}, \boldsymbol{\mu}}(X_i, X_j, X_k) + O(X_{-1}^{1/2+\varepsilon} X_iX_jX_k) , 
\end{align*}
with $T_{\boldsymbol{\varepsilon}, \boldsymbol{\mu}}(X_i, X_j, X_k)$ given by
\begin{align*}
    T_{\boldsymbol{\varepsilon}, \boldsymbol{\mu}}(X_i, X_j, X_k) &=\sum_{\substack{d_i f_i \leq X_i , d_jf_j\leq X_j, d_k f_k \leq X_k  \\  d_xf_x=\varepsilon_x \bmod 8 \text{ for } x\in\{i,j,k\} \\ \eqref{eq:condition_G} \text{ holds} }} \Bigg( {\mu^2( d_if_id_jf_jd_kf_k)} u_{ \boldsymbol{\mu}}(d_i, d_j, d_k)  \nonumber \\
 &\hspace{2cm}   \legendre{f_i }{d_jd_k} \legendre{f_j}{d_i d_k} \legendre{f_k}{d_i d_j}   \prod_{p\mid d_id_jd_kf_if_jf_k} \frac{p}{p+1}\Bigg)\,. 
\end{align*}
We rewrite $T_{\boldsymbol{\varepsilon}, \boldsymbol{\mu}}(X_i, X_j, X_k)$ as
\begin{align}
   T_{\boldsymbol{\varepsilon}, \boldsymbol{\mu}}(X_i, X_j, X_k) = \sum_{\substack{d_if_i\leq X_i \\ d_jf_j\leq X_j  \\ d_xf_x = \varepsilon_x \bmod 8 \text{ for }x=i,j } } \mu^2(d_if_id_jf_j) \legendre{f_i}{d_j}\legendre{f_j}{d_i}S_{d_i, d_j, f_i, f_j}(X_i, X_j, X_k)  ,\label{eq:T_and_S}
\end{align}
where 
\begin{align*}
    S_{d_i, d_j, f_i, f_j}(X_i, X_j, X_k) :=   
    \hspace{-.2cm}\sum_{\substack{d_kf_k \leq X_k  \\  (d_kf_k, d_id_jf_if_j)=1\\d_kf_k  \equiv \varepsilon_k \bmod 8 \\ \eqref{eq:condition_G} \text{ holds}}} \mu^2(d_kf_k) u_{\boldsymbol{\mu}}(d_i, d_j, d_k) \legendre{f_k}{d_id_j} \legendre{f_if_j}{d_k} \prod_{p\mid d_id_jf_if_j d_kf_k} \frac{p}{p+1}.
\end{align*}
We have the following result.
\begin{lemma} \label{lemma:boundsS_ddff}
   Let $X_i, X_j, X_k\in \RR_{>0}$ such that \eqref{eq:size_comp} holds. Let $d_i, f_i\leq X_i$, $d_j, f_j\leq X_j$ be four square-free pairwise coprime integers such that $(d_i, d_j)\neq (1,1)$ and $(f_i, f_j)\neq (1,1)$. Then for any $\varepsilon>0$ the following holds
   \begingroup
\allowdisplaybreaks
        \begin{alignat*}{2}  
       &\quad (i) \qquad  &&S_{d_i,d_j,f_i,f_j} = O \left(X_k^{3/4+\varepsilon} (X_{i}X_j)^{1/4}\right);   \\ 
        &\quad(ii) && S_{1,1,f_i,f_j} = \frac{X_k}{8}  \prod_{p>2} \left(1-\frac{2}{p(p+1)}\right)\hspace{-.2cm} \sum_{\substack{d_k = 1 \\ (d_k,   2f_if_j)=1}}^\infty \hspace{-.2cm}\frac{ \mu^2(d_k) u_{\boldsymbol{\mu}}(1, 1, d_k)}{d_k} \legendre{f_if_j}{d_k}\prod_{p\mid f_if_jd_k} \frac{p}{p+2}\\
        & &&  \hspace{2cm }  + O\left( X_k ^{3/4+\varepsilon}(X_iX_j)^{1/4}  \right)   ;  \\
        &\quad(iii)  &&S_{d_i,d_j,1,1} = \frac{X_k}{8}  \prod_{p>2} \left(1-\frac{2}{p(p+1)}\right) \hspace{-.4cm} \sum_{\substack{f_k=1 \\ (f_k, 2d_id_j)=1}}^\infty \hspace{-.4cm}\frac{\mu^2(f_k)u_{\boldsymbol{\mu}}(\varepsilon_i , \varepsilon_j , \varepsilon_k f_k)}{f_k} \legendre{f_k}{d_id_j}\prod_{p\mid d_id_jf_k} \frac{p}{p+2} \\
        & &&  \hspace{2cm }+ O\left(  X_k ^{3/4+\varepsilon}(X_iX_j)^{1/4}\right)   ;  \\
         &\quad(iv)  &&S_{1,1,1,1} = O\left( X_k\right) . 
    \end{alignat*}
    \endgroup
\end{lemma}
\begin{proof}
Let $d_i, f_i\leq X_i$, $d_j, f_j\leq X_j$ be pairwise coprime and square-free, possibly $(d_i, d_j)=(1,1)$ or $(f_i, f_j) = (1,1)$. We split the sum in $S_{d_i, d_j, f_i, f_j}$ according to whether $d_k\leq X_k^{1/2}$ or $f_k\leq X_k^{1/2}$. We have 
 \begin{align*}
     S_{d_i,d_j, f_i, f_j}(X_i, X_j, X_k) = S^{d} (X_i, X_j, X_k)  + S^f(X_i, X_j, X_k),
 \end{align*}
 with 
 \begin{align*}
     S^{d} (X_i, X_j, X_k)& = \sum_{\substack{d_k \leq X_k^{1/2}  \\  (d_k, 2d_id_jf_if_j) = 1\\ \eqref{eq:condition_G} }} \mu^2(d_k) u_{\boldsymbol{\mu}}(d_i, d_j, d_k) \legendre{f_if_j}{d_k}  \\
     &\hspace{2cm}\sum_{\substack{f_k \leq X_k/d_k \\ (f_k, d_if_id_jf_jd_k) =1\\ f_k \equiv d_k \varepsilon_k \bmod 8  \\ \eqref{eq:condition_G}}} \mu^2(f_k)\legendre{f_k}{d_i d_j}\prod_{p\mid d_if_i d_j f_j d_k f_k} \frac{p}{p+1} 
 \end{align*}
 and its complementary range 
\begin{align*}
    S^f(X_i, X_j, X_k) &= \sum_{\substack{f_k \leq X_k^{1/2}  \\  (f_k, 2d_id_jf_if_j) = 1 \\ \eqref{eq:condition_G}  }} \mu^2(f_k) u_{\boldsymbol{\mu}}(\varepsilon_i f_i, \varepsilon_jf_j,\varepsilon_k f_k)\legendre{f_k}{d_id_j}\\
     &\hspace{2cm}  \sum_{\substack{ X_k^{1/2}<  d_k \leq X_k/f_k \\ (d_k, d_if_id_jf_jf_k) =1\\ f_k \equiv d_k \varepsilon_k \bmod 8 \\ \eqref{eq:condition_G} }} \mu^2(d_k)\legendre{f_if_j}{d_k}\prod_{p\mid d_if_i d_j f_j d_k f_k} \frac{p}{p+1} .
\end{align*}
Assume first $(d_i, d_j) \neq (1,1)$. Then we obtain a saving in $ S^{d} (X_i, X_j, X_k)$. In fact, the character $\legendre{\cdot}{d_id_j}$ is non-principal of modulus $O(X_iX_j)$, after lifting the congruence condition via a sum of the character modulo 8, we can apply \Cref{lem:square_free_character_sum_g} to the inner sum in $ S^{d} (X_i, X_j, X_k)$. This gives 
\begin{align*}
    \sum_{\substack{f_k \leq X_k/d_k \\ (f_k, d_if_id_jf_jd_k) =1\\ f_k \equiv d_k \varepsilon_k \bmod 8 }} \mu^2(f_k)\legendre{f_k}{d_i d_j}
\prod_{p\mid d_if_i d_j f_j d_k f_k} \frac{p}{p+1}  \ll (X_k/d_k)^{1/2+\varepsilon} (X_iX_j)^{1/4+\varepsilon}. \numberthis \label{eq:bound_in_dk}
\end{align*}
Summing over $d_k\leq X_k^{1/2}$ yields 
\begin{align*}
   S^{d}(X_i, X_j, X_k) = O\left(X_k^{3/4+\varepsilon} (X_iX_j)^{1/4+\varepsilon}\right) \qquad \text{whenever $(d_i, d_j)\neq (1,1)$}.\numberthis \label{eq:saving_dk}
\end{align*}
The same argument applies to $S^{f}(X_i, X_j, X_k)$ when $(f_i, f_j)\neq (1,1)$. Here the non-principal character is $\legendre{f_if_j}{\cdot}$, and an identical argument gives 
\begin{align*}
   S^{f}(X_i, X_j, X_k) = O\left(X_k^{3/4+\varepsilon} (X_iX_j)^{1/4+\varepsilon}\right) \qquad \text{whenever $(f_i, f_j)\neq (1,1)$}.\numberthis \label{eq:saving_fk}
\end{align*}
Using the assumption \eqref{eq:size_comp}, $(i)$ follows from \eqref{eq:saving_dk} and \eqref{eq:saving_fk}.

In the case where $(d_i, d_j) = (1,1)$, the character appearing in the inner sum of  $S^d(X_i, X_j, X_k) $ is principal. We obtain an asymptotic for the inner sum  using \Cref{lemma:square-free-asymp}
\begin{align*}
 &   \sum_{\substack{f_k \leq X_k/d_k \\ (f_k, d_if_id_jf_jd_k) =1\\ f_k \equiv d_k \varepsilon_k \bmod 8 }} \mu^2(f_k)
\prod_{p\mid f_i  f_j d_k f_k} \frac{p}{p+1}  \\
&\hspace{2cm}= \frac{X_k}{8d_k}  \prod_{p>2} \left(1-\frac{2}{p(p+1)}\right) \prod_{p\mid f_i f_j d_k } \frac{p}{p+2} + O\left(\tau(f_if_j)(X_k/d_k)^{1/2+\varepsilon} \right). \numberthis  \label{eq:asymp1}
\end{align*}
This shows 
\begin{align*}
    S^d (X_i, X_j, X_k)& = \frac{X_k}{8}  \prod_{p>2} \left(1-\frac{2}{p(p+1)}\right) \sum_{\substack{d_k\leq X_k^{1/2} \\ (d_k, 2f_if_j)=1  }} \frac{\mu^2(d_k) u_{\boldsymbol{\mu }}(1,1,d_k)}{d_k} \legendre{f_if_j}{d_k}\prod_{p\mid f_i f_j d_k } \frac{p}{p+2} \\
    &\qquad + O\left( X_k^{3/4+\varepsilon} \right),\numberthis\label{eq:Sd_asympt}
\end{align*}
where we used that $\tau(f_if_j) = O(X_k^{\varepsilon})$ by \eqref{eq:size_comp}.
 If $(f_i, f_j) \neq (1,1)$ the character $\legendre{f_if_j}{\cdot}$ is non-principal and of modulus $O(X_iX_j)$. Applying partial summation and \Cref{lem:square_free_character_sum_g}, we obtain 
 \begin{align*}
      \sum_{\substack{d_k\leq X_k^{1/2} \\ (d_k, 2f_if_j)=1 }}\hspace{-.3cm} \frac{\mu^2(d_k) u_{\boldsymbol{\mu }}(1,1,d_k)}{d_k} \legendre{f_if_j}{d_k}\prod_{p\mid f_i f_j d_k } \frac{p}{p+2}  &= \hspace{-.4cm} \sum_{\substack{d_k = 1 \\ (d_k, 2f_if_j)=1 }}^\infty \frac{\mu^2(d_k) u_{\boldsymbol{\mu }}(1,1,d_k)}{d_k} \legendre{f_if_j}{d_k}\hspace{-.1cm}\prod_{p\mid f_i f_j d_k } \frac{p}{p+2}  \\
      & \qquad + O(X_k^{-1/4+\varepsilon} (X_iX_j)^{1/4}). 
 \end{align*}
Inserting this back into \eqref{eq:Sd_asympt} and combining this with \eqref{eq:saving_fk}  whenever $(f_i, f_j) \neq (1,1)$, we obtain $(ii)$.

Analogously, in the case where $(f_i, f_j) = (1,1)$, the inner sum in $S^f(X_i, X_j, X_k) $ is evaluated using \Cref{lemma:square-free-asymp}. We have
\begin{align*}
     \sum_{\substack{ X_k^{1/2}\leq  d_k \leq X_k/f_k \\ (d_k, d_id_jf_k) =1\\ d_k \equiv f_k \varepsilon_k \bmod 8 }} \mu^2(d_k)\prod_{p\mid d_i d_j d_k f_k} \frac{p}{p+1} &=   \prod_{p>2} \left(1-\frac{2}{p(p+1)}\right) \prod_{p\mid d_i d_j f_k } \frac{p}{p+2} \left(\frac{X_k}{8f_k}- \frac{X_k^{1/2}}{8}\right) \\
     & \qquad + O\left(\tau(d_id_j)(X_k/f_k)^{1/2+\varepsilon} + \tau(d_id_j) X_k^{1/4+\varepsilon} \right). 
\end{align*}
and we obtain 
\begin{align*}
    &S^f (X_i, X_j, X_k) = \frac{X_k}{8}  \prod_{p>2} \left(1-\frac{2}{p(p+1)}\right) \sum_{\substack{f_k\leq X_k^{1/2} \\ (f_k, 2d_id_j)=1 \\ f_k}} \frac{\mu^2(f_k) u_{\boldsymbol{\mu }}(\varepsilon_i,\varepsilon_j,\varepsilon_kf_k)}{f_k} \legendre{f_k}{d_id_j}\prod_{p\mid d_i d_j f_k } \frac{p}{p+2} \\
    &\quad + O\left(  X_k^{3/4+\varepsilon} \right) + O\Bigg(X_k^{1/2} \sum_{\substack{f_k\leq X_k^{1/2} \\ (f_k, 2d_id_j)=1 }} \mu^2(f_k) u_{\boldsymbol{\mu}}(\varepsilon_i,\varepsilon_j,\varepsilon_kf_k) \legendre{f_k}{d_id_j}\prod_{p\mid d_i d_j f_k } \frac{p}{p+2}\Bigg).
\end{align*}
The second error term is $O((d_id_j)^{1/4}X_k^{3/4+\varepsilon})$ by an application of \Cref{lem:square_free_character_sum_g} as $(d_i, d_j)\neq (1,1)$.
As in the previous case, whenever  $(d_i, d_j) \neq (1,1)$ we can apply partial summation to complete the sum. Together with \eqref{eq:saving_dk}, this gives $(iii)$.

Finally, suppose that $(d_i,d_j)=(1,1)$ and $(f_i, f_j)=(1,1)$. Using \eqref{eq:asymp1}, we have 
\begin{align*}
    S^d(X_i, X_j, X_k) = O \Bigg( X_k \sum_{\substack{d_k \leq X_k^{1/2} \\ d_k \text{ odd}}} \frac{\mu^2(d_k) u_{\boldsymbol{\mu}}(1, 1, d_k) }{d_k}\prod_{p\mid d_k} \frac{p}{p+2} \Bigg) + O(X_k^{3/4+\varepsilon}).
\end{align*}
Now crucially, since $d_id_jf_if_j = 1$, we must have $\mu_i=1$ or $\mu_j=1$, for otherwise we would have $2^{\mu_i+\mu_j}d_id_jf_if_j = 1 $ which is excluded by \eqref{eq:condition_G}. With $\mu_i=1$ or $\mu_j =1$, $u_{\boldsymbol{\mu}}(1,1,\cdot)$ is a non-principal character modulo 8 in the last variable by \Cref{lemma:properties_u} $(i)$. An application of \Cref{lemma:L_func_bound} shows 
\[
S^d(X_i, X_j, X_k) = O( X_k).
\]
The same argument shows that 
\begin{align*}
     S^f(X_i, X_j, X_k) &= O\Bigg( X_k \sum_{\substack{f_k \leq X_k^{1/2}   \\ f_k \text{ odd}}} \frac{\mu^2(f_k) u_{\boldsymbol{\mu}}(1, 1, f_k)}{f_k}\prod_{p\mid f_k} \frac{p}{p+2}\Bigg) = O(X_k).
\end{align*}
This establishes $(iv)$.
\end{proof}

Substituting the above result back into \eqref{eq:T_and_S} shows 
\begin{align*}
   T_{\boldsymbol{\varepsilon}, \boldsymbol{\mu}}(X_i, X_j, X_k) &=  D_{\boldsymbol{\varepsilon}, \boldsymbol{\mu}}(X_i, X_j, X_k) +F_{\boldsymbol{\varepsilon}, \boldsymbol{\mu}}(X_i, X_j, X_k) \\
   &\qquad +O\left(X_k^{3/4+\varepsilon} (X_iX_j)^{5/4+\varepsilon} +X_k \right) , \numberthis \label{eq:T_e_asymp}
   \end{align*} 
   with
   \begin{align}
    & D_{\boldsymbol{\varepsilon}, \boldsymbol{\mu}}(X_i, X_j, X_k)   := \sum_{\substack{d_i \leq X_i \\ d_i \equiv \varepsilon_i \bmod 8} } \sum_{\substack{d_j \leq X_j \\ d_j \equiv \varepsilon_j \bmod 8 \\ d_id_j\neq 1} } \mu^2(d_id_j)S_{d_i,d_j,1,1} (X_i, X_j, X_k) \label{eq:def_D}
     \end{align}
     and
     \begin{align}
      &  F_{\boldsymbol{\varepsilon}, \boldsymbol{\mu}}(X_i, X_j, X_k)  := \sum_{\substack{f_i \leq X_i\\  f_i \equiv \varepsilon_i \bmod 8}}\sum_{\substack{f_j \leq X_j\\  f_j \equiv \varepsilon_j \bmod 8 \\ f_if_j\neq 1}}  \mu^2(f_if_j) S_{1,1, f_i, f_j} (X_i, X_j, X_k) .\label{eq:def_F}
\end{align}
The following result establishes the asymptotic for $  D_{\boldsymbol{\mu}, \boldsymbol{\varepsilon}} (X_i,X_j, X_k)$ and $F_{\boldsymbol{\mu}, \boldsymbol{\varepsilon}} (X_i,X_j, X_k)$
\begin{lemma}\label{lemma:asymp_D_F}
Let $X_i, X_j, X_k \in \RR_{>0}$ with $X_i\leq X_j<X_k^{1/5}$. Then in the notation above, we have 
\begin{align*}
     D_{\boldsymbol{\mu}, \boldsymbol{\varepsilon}} (X_i,X_j, X_k) &= \frac{X_iX_jX_k u_{\boldsymbol{\mu}}(\varepsilon_i, \varepsilon_j, \varepsilon_k)}{512} \prod_{p>2} \left(1-\frac{1}{p}\right)^3 \left(1+ \frac{3}{p+1}\right)\\
     &\quad +O\left(X_iX_jX_k \left(X_i^{-1/2 + \varepsilon} + (X_iX_j)^{-1/6+ \varepsilon}\right)\right);
\end{align*}
and
\begin{align*}
    F_{\boldsymbol{\mu}, \boldsymbol{\varepsilon}} (X_i,X_j, X_k) 
    &= \frac{X_iX_jX_k u_{\boldsymbol{\mu}}(1,1,1)}{512} \prod_{p>2} \left(1-\frac{1}{p}\right)^3 \left(1 +\frac{3}{p+1}\right)\\
     &\quad +O\left(X_iX_jX_k \left(X_i^{-1/2 + \varepsilon} + (X_iX_j)^{-1/6+ \varepsilon}\right)\right).
\end{align*}
\end{lemma}
\begin{proof}
    We note that $D_{\boldsymbol{\varepsilon}, \boldsymbol{\mu}}(X_i, X_j, X_k)$ and $F_{\boldsymbol{\varepsilon}, \boldsymbol{\mu}}(X_i, X_j, X_k)$ are completely symmetric up to the term $u_{\boldsymbol{\mu}}(\cdot)$. We will prove the result for $D_{\boldsymbol{\varepsilon}, \boldsymbol{\mu}}(X_i, X_j, X_k)$. The proof for $F_{\boldsymbol{\varepsilon}, \boldsymbol{\mu}}(X_i, X_j, X_k)$ is analogous. \\
    Starting from \eqref{eq:def_D}, we use $(iii)$ of \Cref{lemma:boundsS_ddff} on the individual terms $S_{d_i,d_j, 1,1}$. This shows that
\begin{align*}
  & D_{\boldsymbol{\varepsilon}, \boldsymbol{\mu}}(X_i, X_j, X_k) = \frac{X_k}{8} \prod_{p>2} \left(1-\frac{2}{p(p+1)} \right)\sum_{\substack{d_i \leq X_i \\ d_i \equiv \varepsilon_i \bmod 8} } \sum_{\substack{d_j \leq X_j \\ d_j \equiv \varepsilon_j \bmod 8 \\ d_id_j\neq 1} } \sum_{\substack{f_k =1  \\ (f_k, 2d_id_j)=1 }}^\infty \Bigg(\mu^2(d_id_jf_k) \\
   &\quad\frac{   u_{\boldsymbol{\mu}} (\varepsilon_i, \varepsilon_j, \varepsilon_kf_k )} {f_k}\legendre{f_k}{d_id_j} \prod_{p\mid d_id_jf_k} \frac{p}{p+2} \Bigg)  + O\left( X_k^{3/4+\varepsilon} (X_iX_j)^{5/4+\varepsilon} \right).  \numberthis\label{eq:D_midway}
\end{align*}
The sum over $f_k$ closely resembles $\frac{L(1, \chi)}{L(2,\chi^2)}$ for the non-trivial quadratic character (up to a fixed sign) 
\[
\chi(f) = u_{\boldsymbol{\mu}}(\varepsilon_i, \varepsilon_j, \varepsilon_k  f) \legendre{f}{d_id_j},
\]
which has modulus $\asymp d_id_j$.  Using Möbius inversion to lift $\mu^2(f_k)$, we have
\begin{align*}
    &\sum_{\substack{f_k = 1 \\ f_k\text{ odd} }}^\infty \frac{\mu^2(f_k)}{f_k}  \legendre{f_k}{d_id_j} u_{\boldsymbol{\mu}}(\varepsilon_i, \varepsilon_j, \varepsilon_k f_k) \prod_{p\mid f_k} \frac{p}{p+2} \\
    &= \sum_{\substack{e=1 \\ e \text{ odd} \\ (e, d_id_j)=1}}^\infty \frac{\mu(e)}{e^2} \sum_{\substack{f_k=1 \\ f_k \text{ odd}}}^\infty \frac{1}{f_k}\legendre{f_k}{ d_id_j} u_{\boldsymbol{\mu}}(\varepsilon_i, \varepsilon_j, \varepsilon_k f_k) \prod_{p\mid ef_k} \frac{p}{p+2} .\numberthis \label{eq:sum_fk2}
\end{align*}
Notice, that we may restrict the summation over the $f_k$ to the range $f_k \leq (X_iX_j)^{1/2+\delta}$ for some small $\delta>0$. In fact
\begin{align*}
   & \sum_{\substack{f_k=1 \\ (f_k, 2d_id_j)=1}}^\infty \frac{1}{f_k}\legendre{f_k}{ d_id_j} u_{\boldsymbol{\mu}}(\varepsilon_i, \varepsilon_j, \varepsilon_kf_k) \prod_{p\mid ef_k} \frac{p}{p+2} \\
     &\qquad \qquad = \sum_{\substack{f_k\leq (X_iX_j)^{1/2+\delta} \\  (f_k, 2d_id_j)=1}} \frac{1}{f_k}\legendre{f_k}{ d_id_j} u_{\boldsymbol{\mu}}(\varepsilon_i, \varepsilon_j, \varepsilon_k f_k) \prod_{p\mid ef_k} \frac{p}{p+2} +O( (X_iX_j)^{-\delta+\varepsilon}) ,
\end{align*}
which again follows from partial summation and P\'olya--Vinogradov to a character of modulus $O(d_id_j) = O(X_iX_j) $.
We substitute this back into \eqref{eq:sum_fk2}, to obtain 
\begin{align*}
    &\sum_{\substack{f_k  =1 \\ (f_k, 2d_id_j) =1 }}^\infty\frac{\mu^2(f_k)}{f_k}  \legendre{f_k}{d_id_j} u_{\boldsymbol{\mu}}(\varepsilon_i, \varepsilon_j, \varepsilon_k f_k) \prod_{p\mid d_id_jf_k} \frac{p}{p+2}  \\
   & =  \prod_{p\mid d_id_j} \frac{p}{p+2} \!\sum_{\substack{e=1 \\ e \text{ odd} \\ (e, d_id_j)=1}}^\infty \frac{\mu(e)}{e^2}  \!\!\sum_{\substack{f_k\leq (X_iX_j)^{1/2+\delta} \\ (f_k, 2d_id_j)=1}} \frac{1}{f_k}\legendre{f_k}{ d_id_j} u_{\boldsymbol{\mu}}(\varepsilon_i, \varepsilon_j, \varepsilon_k f_k) \prod_{p\mid ef_k} \frac{p}{p+2} +O\left( (X_iX_j)^{-\delta+\varepsilon}\right) ,
\end{align*}
where we used that the summation over $\frac{|\mu(e)|\tau(e)}{e^2}$ is absolutely convergent. 
Substituting this back into \eqref{eq:D_midway}, we obtain
\begin{align*}
    D_{\boldsymbol{\mu}, \boldsymbol{\varepsilon}}(X_i, X_j, X_k) &= \frac{X_k}{8} \prod_{p>2} \left(1-\frac{2}{p(p+1)}\right)\sum_{\substack{e=1 \\ e \text{ odd} }} \frac{\mu(e)}{e^2} \sum_{\substack{d_i \leq X_i \\ d_i \equiv \varepsilon_i \bmod 8 \\ (d_i, e)=1}} \sum_{\substack{d_j \leq X_j \\ d_j \equiv \varepsilon_j \bmod 8 \\ (d_j, d_i e)=1 \\ d_id_j\neq 1}} \mu^2(d_id_j) \prod_{p\mid d_id_j} \frac{p}{p+2} \\
    &\qquad  \sum_{\substack{f_k\leq (X_iX_j)^{1/2+\delta} \\ f_k \text{ odd}}} \frac{1}{f_k}\legendre{f_k}{ d_id_j} u_{\boldsymbol{\mu}}(\varepsilon_i, \varepsilon_j, \varepsilon_k f_k) \prod_{p\mid ef_k} \frac{p}{p+2} \\
    & + O\left(X_k^{3/4+\varepsilon}(X_iX_j)^{5/4+\varepsilon} + X_k (X_iX_j)^{1-\delta +\varepsilon}\right). \numberthis \label{eq:comp_sum}
\end{align*}
We distinguish the contribution coming from $f_k = \square$ and $f_k \neq \square$. \Cref{lemma:non_square_contribution} establishes that the  terms with $f_k\neq \square$ contribute at most $O(X_k(X_iX_j)^{3/4+\delta/2 + \varepsilon})$. The main term contribution comes from the $f_k = \square$. Writing $f_k=a^2$ we get 
\begin{align*}
   D_{\boldsymbol{\mu}, \boldsymbol{\varepsilon}}(X_i, X_j, X_k)   &= \frac{X_k u_{\boldsymbol{\mu}}(\varepsilon_i, \varepsilon_j, \varepsilon_k)}{8} \prod_{p>2} \left(1-\frac{2}{p(p+1)}\right) \sum_{\substack{e=1 \\ e \text{ odd} } }^{\infty}  \frac{\mu(e) }{e^2} \sum_{\substack{a = 1\\ a \text{ odd}}}^\infty \frac{1}{a^2} \prod_{p \mid ae} \frac{p}{p+2} \\
   &\qquad \qquad \qquad \sum_{\substack{ d_i\leq X_i \\ {d_i \equiv \varepsilon_i \bmod 8} \\ (d_i, ea) = 1}} \mu^2(d_i)\prod_{p\mid d_i} \frac{p}{p+2} \sum_{\substack{d_j \leq X_j \\ d_j \equiv \varepsilon_j \bmod 8 \\ (d_j, a e d_i)=1}} \mu^2(d_j) \prod_{p\mid d_j} \frac{p}{p+2}   \\
   &+ O(X_k^{3/4+\varepsilon} (X_iX_j)^{5/4+\varepsilon}+X_k (X_iX_j)^{1-\delta+\varepsilon} + X_k (X_iX_j)^{3/4+\delta/2+\varepsilon}),
\end{align*}
where we used $u_{\boldsymbol{\mu}}(\varepsilon_i, \varepsilon_j, \varepsilon_k a^2) = u_{\boldsymbol{\mu}}(\varepsilon_i, \varepsilon_j, \varepsilon_k)$. 
All these sums are now evaluated using for example \Cref{lemma:square-free-asymp}. Starting from the innermost sum, we have 
\begin{align*}
  \sum_{\substack{d_j \leq X_j \\ d_j \equiv \varepsilon_j \bmod 8\\ (d_j, a e d_i)=1}} \mu^2(d_j) \prod_{p\mid d_j} \frac{p}{p+2}=  
\frac{X_j}{8}   \prod_{p>2}\left(1 -\frac{3}{p(p+2)}\right)\prod_{p\mid a e d_i}\left(\frac{p+2}{p+3}\right)+O(X_j^{1/2+\varepsilon} \tau(d_i a e)).
\end{align*}
Summing the error term gives $O(X_kX_j^{1/2+\varepsilon} X_i^{1+\varepsilon})$. Using \Cref{lemma:square-free-asymp} to evaluate the summation over the $d_i$, we have
\begin{align*}
    \sum_{\substack{ d_i\leq X_i \\ {d_i \equiv \varepsilon_i \bmod {8} } \\ (d_i, ea) = 1}} \mu^2(d_i)\prod_{p\mid d_i} \frac{p}{p+3} \prod_{p\mid ae} \frac{p+2}{p+3} = \frac{X_i}{8} \prod_{p\mid ae} \frac{p+2}{p+4}\prod_{p>2} \left( 1-      \frac{4}{p(p+3)}\right)+O(X_i^{1/2+\varepsilon}\tau(ea)).
\end{align*}
Finally the sum over $a$ and $e$ is absolutely convergent. Writing as an Euler product, it is clear
\begin{align*}
    \sum_{\substack{e=1 \\ e \text{ odd} } }^{\infty}  \frac{\mu(e) }{e^2} \sum_{\substack{a = 1\\ a \text{ odd}}}^\infty \frac{1}{a^2} \prod_{p \mid ae} \frac{p}{p+4}  = 1.
\end{align*}
Combining, we find 
\begin{align*}
     &D_{\boldsymbol{\mu}, \boldsymbol{\varepsilon}} (X_i,X_j, X_k) = \frac{X_iX_jX_k u_{\boldsymbol{\mu}}(\varepsilon_i, \varepsilon_j, \varepsilon_k)}{512} \prod_{p>2} \left(1-\frac{2}{p(p+1)}\right) \left(1 -\frac{3}{p(p+2)}\right) \left( 1-      \frac{4}{p(p+3)}\right) \\
     &+O(X_k^{3/4+\varepsilon} (X_iX_j)^{5/4+\varepsilon}+X_k (X_iX_j)^{1-\delta+\varepsilon} + X_k (X_iX_j)^{3/4+\varepsilon+\delta/2} + X_kX_j^{1/2+\varepsilon} X_i^{1+\varepsilon} \\
     &\qquad \quad + X_kX_jX_i^{1/2+\varepsilon}).
     \end{align*}
     Choosing $\delta =\frac{1}{6}$ and using $X_i,X_j \leq X_k^{1/5}$, we find that the above error term is  
     \begin{align*}
       O\left( X_k (X_jX_i)^{5/6+\varepsilon} + X_kX_jX_i^{1/2+\varepsilon} \right).
     \end{align*}
     The result follows from the observation that
     \[
      \prod_{p>2} \left(1-\frac{2}{p(p+1)}\right) \left(1 -\frac{3}{p(p+2)}\right) \left( 1-      \frac{4}{p(p+3)}\right) = \prod_{p>2} \left(1-\frac{1}{p}\right)^3  \left(1 +\frac{3}{p+1}\right).
     \]
\end{proof}
\begin{lemma} \label{lemma:main_3}
    Let $X_i, X_j, X_k, X_{-1} \in \RR_{>0}$ with $X_i\leq X_j<X_k^{1/5}$ and $X_k \leq X_{-1}$, then \Cref{prop_main_N} holds with an error term 
    \[
    E \ll X_iX_jX_kX_{-1}\left(X_{i}^{-1/2+\varepsilon} + (X_iX_j)^{-1/6+\varepsilon} \right)
    \]
    for any $\varepsilon>0$.
\end{lemma}
\begin{proof}
 Starting from \eqref{eq:N_v-1_sum_performed} and  \eqref{eq:T_e_asymp}, the result follows from using the previous \Cref{lemma:asymp_D_F}, and the calculation 
 \[
 \prod_{p>2} \left(1-\frac{1}{p^2}\right)\left(1-\frac{1}{p}\right)^3  \left(1 +\frac{3}{p+1}\right) =\prod_{p>2} \left(1-\frac{1}{p}\right)^4  \left(1 +\frac{4}{p}\right).
 \]
\end{proof}

\subsubsection{Sub-case $X_{j}\leq X_{-1} \leq X_k$} \label{ssec:middle}
We sum in order $X_k, X_{-1}, X_{j}, X_i$. Otherwise, we follow the same techniques as in the previous case. We have 
\begin{align}
    N_{\boldsymbol{\mu}, \boldsymbol{\varepsilon}}(\boldsymbol{X}) = \sum_{v_{-1}\leq X_{-1}} \mu^2(2v_{-1}) T_{v_{-1}, \boldsymbol{\varepsilon}, \boldsymbol{\mu}} (X_i, X_j,X_k) ,\label{eq:N_decomp_T}
\end{align}
where 
\[
 T_{v_{-1}, \boldsymbol\mu, \boldsymbol{\varepsilon}}(X_i, X_j, X_k) = \!\!\!\!\!\!\!\sum_{\substack{d_i f_i \leq X_i , d_jf_j\leq X_j, d_k f_k \leq X_k  \\ d_xf_x =\varepsilon_x \bmod 8 , x\in \{i,j,k\} \\ (d_if_id_jf_jd_kf_k, v_{-1}) = 1\\ \eqref{eq:condition_G} \text{ holds}}} \!\!{\mu^2(d_if_id_jf_jd_kf_k)} u_{\boldsymbol{\mu}}(d_i, d_j, d_k )    \legendre{f_i }{d_jd_k} \legendre{f_j}{d_i d_k} \legendre{f_k}{d_i d_j}\,.
\]
Setting
\begin{align*}
&S_{d_i, d_j, f_i, f_j, v_{-1}}(X_i, X_j, X_k) := \sum_{\substack{d_kf_k\leq X_k\\ (d_kf_k, f_if_jv_{-1} )=1\\ d_kf_k \equiv \varepsilon_k \bmod 8 \\ \eqref{eq:condition_G} \text{ holds}}} \mu^2(d_id_jd_kf_k) u_{\boldsymbol{\mu}}(d_i,d_j,d_k) \legendre{f_k}{d_id_j} \legendre{f_if_j}{d_k},    
\end{align*}
we obtain 
\begin{align}
 T_{v_{-1}, \boldsymbol\mu, \boldsymbol{\varepsilon}}(X_i, X_j, X_k) = \sum_{\substack{d_if_i\leq X_i \\  d_jf_j\leq X_j \\ (d_if_id_jf_j, v_{-1}) = 1 \\ d_xf_x \equiv \varepsilon_x \bmod 8}} \mu^2(d_if_id_jf_j) \legendre{f_i}{d_j}\legendre{f_j}{d_i} S_{d_i, d_j, f_i, f_j, v_{-1}}(X_i, X_j, X_k)\,. \label{eq:T_{v-1}_andS}
\end{align}
As in \Cref{lemma:boundsS_ddff} the terms that contribute to the main term come from either $(f_i, f_j)=(1,1)$ or $(d_i, d_j)=(1,1)$. We show an analogue of \Cref{lemma:boundsS_ddff}

\begin{lemma}\label{lemma:boundsS_ddff_vminus}
 Let $X_i,X_j,X_k, X_{-1}\in\RR_{>0}$ satisfy \eqref{eq:size_comp} and $X_{-1}\leq X_k$. Let
$d_i,f_i, d_j, f_j, v_{-1}$ be square-free pairwise coprime integers with $d_i, f_i\leq X_i$, $d_j, f_j\leq X_j$ and $(d_i, d_j)\neq (1,1)$ and $(f_i, f_j)\neq (1,1)$. Then for any $\varepsilon>0$ we have
\begin{alignat*}{2}
 &\quad (i)\qquad &&S_{d_i,d_j,f_i,f_j,v_{-1}}
   =O\!\left(X_k^{3/4+\varepsilon}(X_iX_j)^{1/4}\right); \hspace{ 8cm}
\end{alignat*}
\begin{alignat*}{2}
&\quad (ii)\qquad
&&S_{1,1,f_i,f_j,v_{-1}} = \hspace{12cm}
\\
& && \;=\frac{X_k}{8}
 \prod_{\substack{p>2}}
 \left(1-\frac1{p^2}\right) \!\!\!
 \sum_{\substack{d_k = 1\\(d_k,2f_if_jv_{-1})=1}}^\infty
  \!\!\!\!\!\frac{\mu^2(d_k)u_{\boldsymbol{\mu}}(1,1,d_k)}{d_k}
 \legendre{f_if_j}{d_k}
 \prod_{p\mid d_kf_if_jv_{-1}}\frac{p}{p+1} \hspace{ 1cm}
\\
&&&\hspace{2cm}
+O\!\left(X_k^{3/4+\varepsilon} (X_iX_j)^{1/4}\right);
\end{alignat*}
\begin{alignat*}{2}
&\quad (iii)\qquad
&&S_{d_i,d_j,1,1,v_{-1}} = \hspace{11.9cm}\\
& && 
 =\frac{X_k}{8}
 \prod_{\substack{p>2}}
 \left(1-\frac1{p^2}\right) \!\!\!\!\!
 \sum_{\substack{f_k =1 \\(f_k,2d_id_jv_{-1})=1}}^\infty \!\!\!\!\!
 \frac{\mu^2(f_k)u_{\boldsymbol{\mu}}(\varepsilon_i,\varepsilon_j,\varepsilon_kf_k)}{f_k}
 \legendre{f_k}{d_id_j}
 \prod_{p\mid d_id_jf_kv_{-1}}\frac{p}{p+1}
\\
&&&\hspace{2cm}
+O\!\left(X_k^{3/4+\varepsilon} (X_iX_j)^{1/4}\right);
\end{alignat*}
\begin{alignat*}{2}
&\quad (iv)\qquad
&&\sum_{v_{-1}\leq X_{-1}} \mu^2(2v_{-1}) S_{1,1,1,1,v_{-1}}
 = O\!\left( X_{-1}X_k\right). \hspace{7.1cm}
\end{alignat*}
\end{lemma}

\begin{proof}
The proof of $(i), (ii), (iii)$ is analogous to that of \Cref{lemma:boundsS_ddff}. We again split \\
$S_{d_i,d_j,f_i,f_j,v_{-1}}(X_i, X_j , X_k)$ according to whether $d_k\le X_k^{1/2}$ or $f_k\le X_k^{1/2}$
\[
S_{d_i,d_j,f_i,f_j,v_{-1}}(X_i, X_j , X_k)
=S^d(X_i, X_j , X_k)+S^f(X_i, X_j , X_k),
\]
where
 \begin{align*}
     S^{d} (X_i, X_j, X_k)& := \hspace{-.3cm}\sum_{\substack{d_k \leq X_k^{1/2}  \\  (d_k, 2d_id_jf_if_jv_{-1}) = 1}} \mu^2(d_k) u_{\boldsymbol{\mu}}(d_i, d_j, d_k) \legendre{f_if_j}{d_k}  \sum_{\substack{f_k \leq X_k/d_k \\ (f_k, d_if_id_jf_jd_k v_{-1}) =1\\ f_k \equiv d_k \varepsilon_k \bmod 8 \\
     \eqref{eq:condition_G}}} \mu^2(f_k)\legendre{f_k}{d_i d_j}
 \end{align*}
 and the complementary range 
\begin{align*}
    S^f(X_i, X_j, X_k) &:= \hspace{-.5cm}\sum_{\substack{f_k \leq X_k^{1/2}  \\  (f_k, 2d_id_jf_if_jv_{-1}) = 1}}  \hspace{-.3cm}\mu^2(f_k) u_{\boldsymbol{\mu}}(\varepsilon_i f_i, \varepsilon_jf_j,\varepsilon_k f_k)\legendre{f_k}{d_id_j} \sum_{\substack{X_k^{1/2}\leq d_k \leq X_k/f_k \\ (d_k, d_if_id_jf_jf_kv_{-1}) =1\\ f_k \equiv d_k \varepsilon_k \bmod 8 \\ \eqref{eq:condition_G} }} \!\!\!\mu^2(d_k)\legendre{f_if_j}{d_k}.
\end{align*}

If $(d_i,d_j)\neq(1,1)$, then the character
$\legendre{\cdot}{d_id_j}$ occurring in the inner sum defining $S^d$ is
non-principal. Applying \Cref{lem:square_free_character_sum_g} yields
\[
S^d
=O\!\left(\tau(v_{-1})X_k^{3/4+\varepsilon}(X_iX_j)^{1/4+\varepsilon}\right) = O\left(X_k^{3/4+\varepsilon}(X_iX_j)^{1/4}\right), \quad \text{whenever } (d_i, d_j)\neq (1,1) .
\]
Similarly, if $(f_i,f_j)\neq(1,1)$, then
\[
S^f
=O\!\left(X_k^{3/4+\varepsilon}(X_iX_j)^{1/4}\right),
\]
which proves $(i)$.

Suppose now that $(d_i,d_j)=(1,1)$. Then the character in the inner sum defining
$S^d$ is principal, so we apply \Cref{lemma:square-free-asymp}
\[
  \sum_{\substack{f_k \leq X_k/d_k \\ (f_k, d_if_id_jf_jd_k v_{-1}) =1\\ f_k \equiv d_k \varepsilon_k \bmod 8 }} \mu^2(f_k) = \frac{X_k}{8d_k} \prod_{p>2} \left(1-\frac{1}{p^2}\right) \prod_{p\mid f_if_j d_k v_{-1}} \frac{p}{p+1} + O(\tau(f_if_jv_{-1}) (X_k/d_k)^{1/2+\varepsilon}).
\]
Summing over $d_k\le X_k^{1/2}$, we have 
\begin{align*}
    S^d(X_i, X_j, X_k) = \frac{X_k}{8} \prod_{p>2} \left(1-\frac{1}{p^2}\right) \hspace{-.5cm}\sum_{\substack{d_k\leq X_k^{1/2}\\ (d_k, 2f_if_jv_{-1})=1}}\!\!\!\! \frac{u_{\boldsymbol{\mu}}(1,1,d_k)}{d_k} \legendre{f_if_j}{d_k} \prod_{p\mid f_if_jd_kv_{-1}} \frac{p}{p+1} + O\left(X_k^{3/4+\varepsilon}\right).
\end{align*}
\Cref{lem:square_free_character_sum_g} and partial summation shows 
\begin{align*}
     S^d(X_i, X_j, X_k) &= \frac{X_k}{8} \prod_{p>2} \left(1-\frac{1}{p^2}\right) \sum_{\substack{d_k = 1 \\ (d_k, 2f_if_jv_{-1}) = 1}}^\infty \frac{u_{\boldsymbol{\mu}}(1,1,d_k)}{d_k} \legendre{f_if_j}{d_k} \prod_{p\mid f_if_jd_kv_{-1}} \frac{p}{p+1} \\
     &\qquad + O\left(X_k^{3/4+\varepsilon} (X_iX_j)^{1/4}\right).
\end{align*}
whenever $(f_i, f_j)\neq (1,1)$. This establishes $(ii)$.
\noindent
The proof of $(iii)$ follows the same steps as in \Cref{lemma:boundsS_ddff}.

Finally, if $(d_i,d_j)=(1,1)$ and $(f_i,f_j)=(1,1)$, we have 
\begin{align*}
    S_{1,1,1,1, v_{-1}} &= \frac{X_k}{8} \prod_{p>2}\left(1-\frac{1}{p^2}\right) \Bigg(\sum_{\substack{ d_k \leq X_k^{1/2 } \\ (d_k, v_{-1})=1 }} \frac{\mu^2(d_k)u_{\boldsymbol{\mu}}(1,1,d_k)}{d_k}\prod_{p\mid d_k v_{-1}} \frac{p}{p+1}  \\
    &\qquad \qquad + \sum_{\substack{ f_k \leq X_k^{1/2 } \\ (f_k, v_{-1})=1 }} \frac{\mu^2(f_k)u_{\boldsymbol{\mu}}(1,1, f_k)}{f_k}\prod_{p\mid f_k v_{-1}} \frac{p}{p+1} \Bigg) + O(\tau(v_{-1})X_{k}^{3/4+\varepsilon}) \\
    &=  \Tilde{c} X_k \sum_{d\leq X_k^{1/2}}\mu^2(2v_{-1} d)  \frac{\chi(d)}{d}\prod_{p\mid v_{-1} d} \frac{p}{p+1} + O(\tau(v_{-1})X_{k}^{3/4+\varepsilon}) ,
\end{align*} for some constant $\Tilde{c}>0$ and $\chi(d) = u_{\boldsymbol{\mu}}(1,1,d)$ a non-principal character modulo 8. Indeed, $\chi(d)$ is non-principal, as the principal case is ruled out by \eqref{eq:condition_G}. We have 
\[
\sum_{v_{-1}\leq X_{-1}} \tau(v_{-1})X^{3/4+\varepsilon} \ll X_k^{3/4+\varepsilon} X_{-1}^{1+\varepsilon} = O(X_{k}X_{-1}),
\]
as $X_{-1}\leq X_k$. Thus, to show $(iv)$ it suffices to show the estimate 
   \begin{align*}
       \sum_{v_{-1} \leq X_{-1}} \sum_{d\leq X_k^{1/2}} \mu^2(2v_{-1} d)  \frac{\chi(d)}{d} \prod_{p\mid v_{-1} d} \frac{p}{p+1}\ll X_{-1}.
\end{align*}
We complete the inner sum above using partial summation and \Cref{lem:square_free_character_sum_g} 
   \begin{align*}
       &\sum_{v_{-1} \leq X_{-1}} \sum_{d\leq X_k^{1/2}} \mu^2(2v_{-1} d)  \frac{\chi(d)}{d} \prod_{p\mid v_{-1} d} \frac{p}{p+1} \\
       & \hspace{3cm}=  \sum_{v_{-1} \leq X_{-1}} \sum_{d=1}^\infty \mu^2(2v_{-1} d)  \frac{\chi(d)}{d} \prod_{p\mid v_{-1} d} \frac{p}{p+1} + O(X_{-1}X_k^{-1/4+\varepsilon}).
\end{align*}
Writing the inner sum as an Euler product, the completed sum equals 
\begin{align*}
    \prod_{p>2} \left(1 + \frac{\chi(p)}{p+1}\right)\sum_{v_{-1}\leq X_{-1}} \mu^2(2v_{-1}) \prod_{p\mid v_{-1}} \frac{p}{p+1+\chi(p)} ,
\end{align*}
which is $O(X_{-1})$ since $ |\frac{p}{p+1+\chi(p)}| \leq 1$. Hence 
\[
 \sum_{v_{-1 }\leq X_{-1}} \mu^2(2v_{-1}) S_{1,1,1,1, v_{-1}} = O(X_kX_{-1}),
\]
as claimed in $(iv)$.
\end{proof}

Analogous to \eqref{eq:def_D}, \eqref{eq:def_F}, we introduce 
 \begin{align}
    & D_{v_{-1},\boldsymbol{\varepsilon}, \boldsymbol{\mu}}(X_i, X_j, X_k)  := \sum_{\substack{d_i \leq X_i \\ d_i \equiv \varepsilon_i \bmod 8 \\ (d_i, v_{-1}) =1} } \sum_{\substack{d_j \leq X_j \\ d_j \equiv \varepsilon_j \bmod 8 \\ (d_j, v_{-1}) = 1 \\ d_id_j \neq 1} } \mu^2(d_id_j)S_{d_i, d_j, 1, 1, v_{-1}} \label{eq:def_D_v-1}
     \end{align}
     and
     \begin{align}
     &  F_{v_{-1}, \boldsymbol{\varepsilon}, \boldsymbol{\mu}}(X_i, X_j, X_k)   := \sum_{\substack{f_i \leq X_i\\  f_i \equiv \varepsilon_i \bmod 8 \\ (f_i, v_{-1})=1}}\sum_{\substack{f_j \leq X_j\\  f_j \equiv \varepsilon_j \bmod 8 \\ (f_j, v_{-1}) = 1 \\ f_if_j \neq 1} } \mu^2(f_if_j) S_{1,1,f_i, f_j ,v_{-1}} .\label{eq:def_F_v-1}
\end{align}
We now perform the summation over $v_{-1}$ in  \eqref{eq:N_decomp_T}. We have

\begin{lemma} \label{lemma:perform_v_sum_D_F}
Let $X_i, X_j,X_k, X_{-1} \in \RR_{>0}$ such that \eqref{eq:size_comp} and $X_j\leq X_{-1}\leq X_k$ holds. Let $\varepsilon>0$. In the above notation we have 
\begin{align*}
\sum_{\substack{v_{-1} \leq X_{-1} \\ v_{-1} \text{ odd }}} \mu^2(v_{-1} ) D_{v_{-1},\boldsymbol{\varepsilon}, \boldsymbol{\mu}}(X_i, X_j, X_k)&=\frac{X_{-1}}{2} \prod_{p>2} \left(1-\frac{1}{p^2}\right) D_{\boldsymbol{\varepsilon}, \boldsymbol{\mu}}(X_i, X_j, X_k) \\
&\qquad + O\left(X_{-1}^{1/2+\varepsilon} X_k X_iX_j+ X_{-1} X_k^{3/4+\varepsilon} (X_iX_j)^{5/4} \right);
\end{align*}
and
\begin{align*}
\sum_{\substack{v_{-1} \leq X_{-1} \\ v_{-1} \text{ odd }}} \mu^2(v_{-1} ) F_{v_{-1},\boldsymbol{\varepsilon}, \boldsymbol{\mu}}(X_i, X_j, X_k)&=\frac{X_{-1}}{2} \prod_{p>2} \left(1-\frac{1}{p^2}\right) F_{\boldsymbol{\varepsilon}, \boldsymbol{\mu}}(X_i, X_j, X_k) \\
&\qquad + O\left(X_{-1}^{1/2+\varepsilon} X_k X_iX_j+ X_{-1} X_k^{3/4+\varepsilon} (X_iX_j)^{5/4} \right);
\end{align*}
with $D_{\boldsymbol{\varepsilon}, \boldsymbol{\mu}}(X_i, X_j, X_k)$ and $F_{\boldsymbol{\varepsilon}, \boldsymbol{\mu}}(X_i, X_j, X_k)$ given by \eqref{eq:def_D} and \eqref{eq:def_F}.
\end{lemma}
\begin{proof}
Starting from the definition of \eqref{eq:def_D_v-1}, we pull in the summation over $v_{-1}$. We have 
\begin{align*}
    &\sum_{\substack{v_{-1} \leq X_{-1} \\ v_{-1} \text{ odd }}} \mu^2(v_{-1} ) D_{v_{-1},\boldsymbol{\varepsilon}, \boldsymbol{\mu}}(X_i, X_j, X_k) \\
    &\qquad\qquad =   \sum_{\substack{d_i \leq X_i \\ d_i \equiv \varepsilon_i \bmod 8  } }\sum_{\substack{d_j \leq X_j \\ d_j \equiv \varepsilon_j \bmod 8  \\ d_id_j \neq 1} } \mu^2(d_id_j) \sum_{\substack{v_{-1} \leq X_{-1} \\ (v_{-1}, 2d_id_j)=1 }} \mu^2(v_{-1})S_{d_i,d_j,1,1, v_{-1}}.
\end{align*}
Define  
\[
\chi(f_k) := u_{\boldsymbol{\mu}}(\varepsilon_i, \varepsilon_j, \varepsilon_k  f_k) \legendre{f_k}{d_id_j},
\]
which is (up to a fixed sign) a non-principal character of modulus $O(X_iX_j)$. Moreover $\chi(p)=0$ for any $p\mid 2d_id_j$. By \Cref{lemma:boundsS_ddff_vminus} $(iii)$, the above equals
\begin{align*}
&\sum_{\substack{v_{-1} \leq X_{-1} \\ v_{-1} \text{ odd }}} \mu^2(v_{-1} ) D_{v_{-1},\boldsymbol{\varepsilon}, \boldsymbol{\mu}}(X_i, X_j, X_k)=    \frac{X_k}{8} \prod_{p>2}\left(1-\frac{1}{p^2}\right) \sum_{\substack{d_i \leq X_i \\ d_i \equiv \varepsilon_i \bmod 8 } } \sum_{\substack{d_j \leq X_j \\ d_j \equiv \varepsilon_j \bmod 8 \\ d_id_j \neq 1} } \mu^2(d_id_j) \\
 & \quad\sum_{\substack{v_{-1} \leq X_{-1} \\ (v_{-1}, 2d_id_j)=1 }} \mu^2( v_{-1})\sum_{\substack{f_k=1 \\  (f_k, 2d_id_j v_{-1})=1 }}^\infty \frac{\mu^2(f_k) \chi(f_k)}{f_k}  \prod_{p\mid d_id_jf_kv_{-1}} \frac{p}{p+1} + O\left( X_k^{3/4+\varepsilon}X_{-1} (X_iX_j)^{5/4}\right). \numberthis \label{eq:v_D_s}
\end{align*}
    Writing the inner infinite sum over $f_k$ as an Euler product, we obtain 
\begin{align*}
   &\sum_{\substack{v_{-1} \leq X_{-1} \\ (v_{-1}, 2d_id_j)=1 }} \mu^2( v_{-1})\sum_{\substack{f_k=1 \\  (f_k, 2d_id_j v_{-1})=1 }}^\infty \frac{\mu^2(f_k) \chi(f_k)}{f_k}  \prod_{p\mid d_id_jf_kv_{-1}} \frac{p}{p+1}\\
   &\hspace{2.5cm} = \prod_{p\mid d_id_j} \frac{p}{p+1} \prod_{p\nmid 2d_id_j} \left(1+\frac{\chi(p)}{p+1}\right) \sum_{\substack{v_{-1} \leq X_{-1} \\ (v_{-1}, 2d_id_j)=1 }} \mu^2( v_{-1}) \prod_{p\mid v_{-1}} \frac{p}{p+1+\chi(p)} .
\end{align*}
Applying \Cref{lemma:square-free-asymp} to the remaining sum shows that the above equals
\begin{align*}
      &=\frac{X_{-1}}{2}\prod_{p\mid d_id_j} \frac{p-1}{p+1} \prod_{p\nmid 2d_id_j} \left(1+\frac{\chi(p)}{p+1}\right) \left(1-\frac{1}{p}\right) \left(1+\frac{1}{p+1+\chi(p)}\right) + O\left(X_{-1}^{1/2+\varepsilon} \right),
\end{align*} 
where we used $\tau(d_i d_j) \ll X_{-1}^{\varepsilon}$ as $X_i, X_j\leq X_{-1}$.
Simplifying the Euler factors, using that $\chi(p)=0$ for all $p\mid 2d_id_j$, this is 
\begin{align*}
    &=  \frac{X_{-1}}{2}\prod_{p>2} \left(1-\frac{2}{p(p+1)}\right)  \prod_{p \mid d_id_j} \frac{p}{p+2} \prod_{p\nmid 2d_id_j}\left(1+\frac{\chi(p)}{p+2}\right) + O\left(X_{-1}^{1/2+\varepsilon}  \right).
\end{align*}
Finally, expanding the Euler product back into a series gives
\begin{align*}
   &\sum_{\substack{v_{-1} \leq X_{-1} \\ (v_{-1}, 2d_id_j)=1 }} \mu^2( v_{-1})\sum_{\substack{f_k = 1 \\  (f_k, 2d_id_j v_{-1})=1 }}^\infty \frac{\mu^2(f_k)\chi(f_k)}{f_k}  \prod_{p\mid d_id_jf_kv_{-1}} \frac{p}{p+1} \\
   & \qquad \qquad = \frac{1}{2} X_{-1} \prod_{p>2} \left(1-\frac{2}{p(p+1)}\right)  \sum_{\substack{f_k = 1 \\ (f_k,2d_id_j)=1 }}^\infty \frac{\mu^2(f_k) \chi(f_k)}{f_k}  \prod_{p \mid d_id_jf_k}   \frac{p}{p+2}  + O\left(X_{-1}^{1/2+\varepsilon} \right).
\end{align*}
Inserting this back into \eqref{eq:v_D_s} shows
\begin{align*}
    &\sum_{\substack{v_{-1} \leq X_{-1} \\ v_{-1} \text{ odd }}} \mu^2(v_{-1} ) D_{v_{-1},\boldsymbol{\varepsilon}, \boldsymbol{\mu}}(X_i, X_j, X_k)= \frac{X_{-1}X_k}{16} \prod_{p>2} \left(1-\frac{1}{p^2}\right)\left(1-\frac{2}{p(p+1)}\right) \\
    &\qquad\cdot \sum_{\substack{d_i \leq X_i \\ d_i \equiv \varepsilon_i \bmod 8} } \sum_{\substack{d_j \leq X_j \\ d_j \equiv \varepsilon_j \bmod 8 \\ d_id_j \neq 1} } \mu^2(d_id_j)  \sum_{\substack{f_k = 1 \\ (f_k,2d_id_j)=1 }}^\infty \frac{\mu^2(f_k) \chi(f_k)}{f_k}  \prod_{p \mid d_id_jf_k}   \frac{p}{p+2}  \\
    &\qquad + O\left(X_{-1}^{1/2+\varepsilon} X_k X_iX_j+ X_{-1} X_k^{3/4+\varepsilon} (X_iX_j)^{5/4} \right).
\end{align*}
Using the definition of $D_{\boldsymbol{\varepsilon}, \boldsymbol{\mu}}(X_i, X_j, X_k)$ in \eqref{eq:def_D} and \Cref{lemma:boundsS_ddff} $(iii)$ this is exactly 
\begin{align*}
\frac{X_{-1}}{2} \prod_{p>2} \left(1-\frac{1}{p^2}\right) D_{\boldsymbol{\varepsilon}, \boldsymbol{\mu}}(X_i, X_j, X_k)+ O\left(X_{-1}^{1/2+\varepsilon} X_k X_iX_j+ X_{-1} X_k^{3/4+\varepsilon} (X_iX_j)^{5/4} \right),
\end{align*}
as claimed. The proof for $F_{v_{-1}, \boldsymbol{\varepsilon}, \boldsymbol{\mu}}(X_i, X_j, X_k)$ is analogous. 
\end{proof}

 The remaining sums $D_{\boldsymbol{\mu}, \boldsymbol{\varepsilon}}(X_i, X_j, X_k)$ and $F_{\boldsymbol{\mu}, \boldsymbol{\varepsilon}}(X_i, X_j, X_k)$ have been studied in \Cref{lemma:asymp_D_F} in \Cref{seq:Xjsmall_Xminus_large}. In fact, we obtain the following result.
 \begin{lemma} \label{lemma:main_4}
     Let $X_i, X_j, X_k, X_{-1}\in \RR_{>0}$ such that $X_i\leq X_j \leq X_{-1} \leq X_k$ and $X_j\leq X_k^{1/5}$. Then \Cref{prop_main_N} holds with an error term 
    \[
    E \ll X_iX_jX_kX_{-1}\left(X_{i}^{-1/2+\varepsilon} + (X_iX_j)^{-1/6+\varepsilon}  \right)
    \]
     for any $\varepsilon>0$.
 \end{lemma}
 \begin{proof}
     Recalling  \eqref{eq:N_decomp_T}, \Cref{lemma:boundsS_ddff_vminus} shows 
     \begin{align*}
     N_{\boldsymbol{\mu}, \boldsymbol{\varepsilon}}(X_i, X_j, X_k, X_{-1}) &= \sum_{v_{-1}\leq X_{-1}} \mu^2(2v_{-1}) \left( D_{v_{-1}, \boldsymbol{\varepsilon}, \boldsymbol{\mu}}(X_i, X_j, X_k) + F_{v_{-1}, \boldsymbol{\varepsilon}, \boldsymbol{\mu}}(X_i, X_j, X_k)\right) \\
     & \qquad + O(X_{-1}X_k^{3/4+\varepsilon} (X_iX_j)^{5/4+\varepsilon} + X_kX_{-1}).
     \end{align*}
     By \Cref{lemma:perform_v_sum_D_F}, we obtain
     \begin{align*}
    N_{\boldsymbol{\mu}, \boldsymbol{\varepsilon}}(X_i, X_j, X_k, X_{-1})  &=\frac{X_{-1}}{2} \prod_{p>2} \left(1-\frac{1}{p^2}\right) \left( D_{\boldsymbol{\varepsilon}, \boldsymbol{\mu}}(X_i, X_j, X_k) + F_{ \boldsymbol{\varepsilon}, \boldsymbol{\mu}}(X_i, X_j, X_k)\right) \\
     & \qquad + O\left(X_{-1}X_k^{3/4+\varepsilon} (X_iX_j)^{5/4+\varepsilon} +X_{-1}^{1/2+\varepsilon} X_kX_{i}X_j+ X_kX_{-1}\right).  
     \end{align*}
     The result then follows from \Cref{lemma:asymp_D_F}.
 \end{proof}
\subsubsection{Sub-case $X_{-1} \leq X_j$} \label{ssec:last} In this last case, sum over $v_{-1}$ last, retaining the coprimality to $v_{-1}$ condition on all the variables inside. The techniques are similar to the ones developed in the previous two cases. We write $N_{\boldsymbol{\mu}, \boldsymbol{\varepsilon}}$ as in \eqref{eq:N_decomp_T}. Then \Cref{lemma:boundsS_ddff_vminus} shows 
     \begin{align*}
     N_{\boldsymbol{\mu}, \boldsymbol{\varepsilon}}(X_i, X_j, X_k, X_{-1}) &= \sum_{v_{-1\leq X_{-1}}} \mu^2(2v_{-1}) \left( D_{v_{-1}, \boldsymbol{\varepsilon}, \boldsymbol{\mu}}(X_i, X_j, X_k) + F_{v_{-1}, \boldsymbol{\varepsilon}, \boldsymbol{\mu}}(X_i, X_j, X_k)\right) \\
     & \qquad + O(X_{-1}X_k^{3/4+\varepsilon} (X_iX_j)^{5/4+\varepsilon} + X_kX_{-1}) \numberthis \label{eq:nasasfsdf},
     \end{align*}
with $D_{v_{-1}, \boldsymbol{\varepsilon}, \boldsymbol{\mu}}(X_i, X_j, X_k)$, $ F_{v_{-1}, \boldsymbol{\varepsilon}, \boldsymbol{\mu}}(X_i, X_j, X_k)$ given by \eqref{eq:def_D_v-1} and \eqref{eq:def_F_v-1}. The following result 
establishes asymptotics for the individual terms $D_{v_{-1}, \boldsymbol{\varepsilon}, \boldsymbol{\mu}}(X_i, X_j, X_k)$ and $F_{v_{-1}, \boldsymbol{\varepsilon}, \boldsymbol{\mu}}(X_i, X_j, X_k)$.

\begin{lemma} \label{lemma:asymp_D_v}
       Let $X_i, X_j, X_k, X_{-1}\in \RR_{>0}$ such that $X_i\leq X_j  \leq X_k^{1/5}$ and $X_{-1}\leq X_j$. Let $v_{-1}\leq X_{-1}$ be odd and square-free. Then for any $\varepsilon>0$ we have
\begin{align*}
      D_{v_{-1}, \boldsymbol{\mu}, \boldsymbol{\varepsilon}}(X_i, X_j, X_k)   &=  \frac{X_iX_jX_k u_{\boldsymbol{\mu}}(\varepsilon_i, \varepsilon_j, \varepsilon_k)}{512} \prod_{p>2} \left(1-\frac{1}{p}\right)^3 \left(1+\frac{3}{p}\right) \prod_{p\mid v_{-1}} \frac{p}{p+3} \\
      &\quad +O\left(X_iX_jX_k \left( X_i^{-1/2+\varepsilon} + (X_iX_j)^{-1/6+\varepsilon}\right) \right);
\end{align*}
and 
\begin{align*}
      F_{v_{-1}, \boldsymbol{\mu}, \boldsymbol{\varepsilon}}(X_i, X_j, X_k)   &=  \frac{X_iX_jX_k u_{\boldsymbol{\mu}}(1,1,1)}{512} \prod_{p>2} \left(1-\frac{1}{p}\right)^3 \left(1+\frac{3}{p}\right) \prod_{p\mid v_{-1}} \frac{p}{p+3} \\
        &\quad+O\left(X_iX_jX_k\left( X_i^{-1/2+\varepsilon} + (X_iX_j)^{-1/6+\varepsilon} \right) \right).
\end{align*}
\end{lemma}

 \begin{proof}

We show the the case for $D_{v_{-1}, \boldsymbol{\varepsilon}, \boldsymbol{\mu}}(X_i, X_j, X_k)$, the proof for $F_{v_{-1}, \boldsymbol{\varepsilon}, \boldsymbol{\mu}}(X_i, X_j, X_k)$ is analogous. \\
Using \Cref{lemma:boundsS_ddff_vminus} $(iii)$, we have 
\begin{align*}
D_{v_{-1}, \boldsymbol{\varepsilon}, \boldsymbol{\mu}}&(X_i, X_j, X_k)= \frac{X_k}{8}  \prod_{p>2} \left(1-\frac{1}{p^2}\right)\sum_{\substack{d_i\leq X_i \\ (d_i, 2v_{-1})=1 \\ d_i \equiv \varepsilon_i \bmod 8}} \sum_{\substack{d_j\leq X_j \\ (d_j, 2d_iv_{-1})=1 \\ d_j \equiv \varepsilon_j \bmod 8\\d_id_j \neq 1}}
\sum_{\substack{f_k =1  \\ (f_k, 2v_{-1})=1}}^\infty \Bigg( \\
&\mu^2(d_id_jf_k) \legendre{f_k}{d_id_j} \frac{u(\varepsilon_i, \varepsilon_j, \varepsilon_kf_k)}{f_k}\prod_{p\mid {v_{-1}d_id_jf_k}} \frac{p}{p+1} \Bigg) + O\left(X_k^{3/4+\varepsilon}(X_iX_j)^{5/4}\right) .\numberthis\label{eq:D_v_last}
\end{align*}
The proof from this point is analogous to the proof of \Cref{lemma:asymp_D_F}. We briefly go through the steps. Applying Möbius inversion on $\mu^2(f_k)$, we have
\begin{align*}
    &\sum_{\substack{f_k = 1 \\ (f_k, 2v_{-1})=1}}^\infty \frac{\mu^2(f_k)}{f_k}  \legendre{f_k}{d_id_j} u_{\boldsymbol{\mu}}(\varepsilon_i, \varepsilon_j, \varepsilon_k f_k) \prod_{p\mid f_k} \frac{p}{p+1} \\
    &= \sum_{\substack{e=1 \\ e \text{ odd} \\ (e, v_{-1} d_id_j)=1}}^\infty \frac{\mu(e)}{e^2} \sum_{\substack{f_k=1 \\ (f_k, 2 v_{-1})=1}}^\infty \frac{1}{f_k}\legendre{f_k}{ d_id_j} u_{\boldsymbol{\mu}}(\varepsilon_i, \varepsilon_j, \varepsilon_k f_k) \prod_{p\mid ef_k} \frac{p}{p+1} .\numberthis \label{eq:sum_fk3}
\end{align*}
We restrict the summation over the $f_k$ to the range $f_k \leq (X_iX_j)^{1/2+\delta}$ for some small $\delta>0$. Using partial summation and P\'olya--Vinogradov, we have
\begin{align*}
   & \sum_{\substack{f_k=1 \\ f_k \text{ odd} \\ (f_k, v_{-1})=1}}^\infty \frac{1}{f_k}\legendre{f_k}{ d_id_j} u_{\boldsymbol{\mu}}(\varepsilon_i, \varepsilon_j, \varepsilon_kf_k) \prod_{p\mid ef_k} \frac{p}{p+1} \\
     &\qquad \qquad = \sum_{\substack{f_k\leq (X_iX_j)^{1/2+\delta} \\ f_k \text{ odd} \\ (f_k, v_{-1})=1}} \frac{1}{f_k}\legendre{f_k}{ d_id_j} u_{\boldsymbol{\mu}}(\varepsilon_i, \varepsilon_j, \varepsilon_k f_k) \prod_{p\mid ef_k} \frac{p}{p+1} +O( \tau(v_{-1}) (X_iX_j)^{-\delta+\varepsilon}) .
\end{align*}
Substituting this back into the expression for $D_{v_{-1}, \boldsymbol{\varepsilon}, \boldsymbol{\mu}}(X_i, X_j, X_k)$ in \eqref{eq:D_v_last}, we obtain
\begin{align*}
    D_{v_{-1}, \boldsymbol{\mu}, \boldsymbol{\varepsilon}}(X_i, X_j, X_k) &= \frac{X_k}{8} \prod_{p>2} \left(1-\frac{1}{p^2)}\right)\sum_{\substack{e=1 \\ e \text{ odd} \\ (e, v_{-1}) }} \frac{\mu(e)}{e^2} \sum_{\substack{d_i \leq X_i \\ d_i \equiv \varepsilon_i \bmod 8 \\ (d_i, ev_{-1})=1}} \sum_{\substack{d_j \leq X_j \\ d_j \equiv \varepsilon_j \bmod 8 \\ (d_j, d_i ev_{-1})=1}}  \\
    & \mu^2(d_id_j) \prod_{p\mid d_id_j} \frac{p}{p+1}  \sum_{\substack{f_k\leq (X_iX_j)^{1/2+\delta} \\ f_k \text{ odd} \\ (f_k, v_{-1})=1}} \frac{1}{f_k}\legendre{f_k}{ d_id_j} u_{\boldsymbol{\mu}}(\varepsilon_i, \varepsilon_j, \varepsilon_k f_k) \prod_{p\mid ef_kv_{-1}} \frac{p}{p+2} \\
    & + O( \tau(v_{-1}) X_k^{3/4+\varepsilon}(X_iX_j)^{5/4+\varepsilon} + \tau(v_{-1})X_k (X_iX_j)^{1-\delta +\varepsilon})\,. 
\end{align*}
The analysis of the remaining sum now follows the same steps as the analysis following \eqref{eq:comp_sum} in the proof of \Cref{lemma:asymp_D_F} and in fact, following this approach gives 
\begin{align*}
      D_{v_{-1}, \boldsymbol{\mu}, \boldsymbol{\varepsilon}}(X_i, X_j, X_k)   &=  \frac{X_iX_jX_k u_{\boldsymbol{\mu}}(\varepsilon_i, \varepsilon_j, \varepsilon_k)}{512} \prod_{p>2} \left(1-\frac{1}{p}\right)^3 \left(1+\frac{3}{p}\right) \prod_{p\mid v_{-1}} \frac{p}{p+3} \\
      &\quad +O\left(\tau(v_{-1})X_iX_jX_k \left( X_i^{-1/2+\varepsilon} + (X_iX_j)^{-1/6+\varepsilon}\right) \right).
\end{align*}
Since $X_{-1}\leq X_j$ by assumption, the term $\tau(v_{-1})$ is absorbed in the $X_j^{\varepsilon}$, as desired.
 \end{proof}

We conclude this section by giving the proof of \Cref{prop_main_N} in the last remaining case.      
 \begin{lemma} \label{lemma:main_5}
     Let $X_i, X_j, X_k, X_{-1}\in \RR_{>0}$ such that $X_i\leq X_j  \leq X_k^{1/5}$ and $X_{-1}\leq X_j$. Then \Cref{prop_main_N} holds with an error term 
    \[
    E \ll X_iX_jX_kX_{-1}\left(X_{-1}^{-1/2+\varepsilon} + X_{i}^{-1/2+\varepsilon} + (X_iX_j)^{-1/6+\varepsilon} \right)
    \]
     for any $\varepsilon>0$.
 \end{lemma}
 \begin{proof}
  Inserting the previous \Cref{lemma:asymp_D_v} into \eqref{eq:nasasfsdf}, we have 
  \begin{align*}
       N_{\boldsymbol{\mu}, \boldsymbol{\varepsilon}}(\boldsymbol{X}) &= \frac{u_{\boldsymbol{\mu}}(1,1,1)+u_{\boldsymbol{\mu}}(\varepsilon_i, \varepsilon_j, \varepsilon_k)}{512}   \prod_{p>2} \left(1-\frac{1}{p}\right)^3\left(1+\frac{3}{p}\right)X_iX_jX_k \\
       &  \cdot\sum_{v_{-1} \leq X_{-1}}\hspace{-.2cm} \mu^2(2v_{-1})\prod_{p\mid v_{-1}} \frac{p}{p+3} 
        + O\left(X_{-1}X_iX_jX_k\left( X_i^{-1/2+\varepsilon} + (X_iX_j)^{-1/6+\varepsilon} \right) \right).
  \end{align*}
  An application of \Cref{lemma:square-free-asymp} then shows 
    \begin{align*}
       N_{\boldsymbol{\mu}, \boldsymbol{\varepsilon}}(\boldsymbol{X}) &= \frac{u_{\boldsymbol{\mu}}(1,1,1)+u_{\boldsymbol{\mu}}(\varepsilon_i, \varepsilon_j, \varepsilon_k)}{1024} \prod_{p>2} \left(1-\frac{1}{p}\right)^4\left(1+\frac{4}{p}\right)X_iX_jX_kX_{-1} \\
       &\qquad +O\left(X_{-1}X_iX_jX_k\left( X_{-1}^{-1/2+\varepsilon} +X_{i}^{-1/2+\varepsilon}+(X_iX_j)^{-1/6+\varepsilon} \right) \right).
  \end{align*}
  as claimed.
 \end{proof}
\begin{proof}[\texorpdfstring{Proof of \Cref{prop_main_N}}{}]
    Combining \Cref{lemma:main_1},  \Cref{lemma:main_2} as well as \Cref{lemma:main_3}, \Cref{lemma:main_4}, \Cref{lemma:main_5} shows that \Cref{prop_main_N} holds for all the cases introduced in \Cref{list:cases}.\\
    Since $N_{\boldsymbol{\mu}, \boldsymbol{\varepsilon}}(X_i, X_j, X_k, X_{-1})$ is symmetric in $i,j,k$, this shows that \Cref{prop_main_N} holds for all ranges $X_i, X_j, X_k, X_{-1}$ sufficiently large with an error that gives a power saving in the smallest of $X_i, X_j, X_k, X_{-1}$.
\end{proof}

\section{\texorpdfstring{Proof of \Cref{thm1}, \Cref{thm2_local} and \Cref{corollary1}}{}}
\begin{proof}[Proof of \Cref{thm2_local}]
    Recall \eqref{eq:analytic_MPhi_new}.  For any $\alpha', \beta' \in (\ZZ/8\ZZ)^\times$ there are four choices  $(\varepsilon_i, \varepsilon_j, \varepsilon_k) \in ((\ZZ/8\ZZ)^\times)^3$ with 
    \[
    \varepsilon_i\varepsilon_j \equiv \alpha' \bmod 8, \qquad \varepsilon_j\varepsilon_k \equiv \beta' \bmod 8.
    \]
    Moreover we have 
    \[
     2^{\mu_i+\mu_j} \varepsilon_i\varepsilon_j \equiv \alpha , \qquad  2^{\mu_j+\mu_k }\varepsilon_j\varepsilon_k \equiv \beta \quad \text{in }\QQ_2^\times /\QQ_2^{\times 2},
    \]
    with $\alpha, \beta$ the square-classes specified by $\Phi$ via \Cref{lemma:local_comp}. By \Cref{rmk:local_hilbert2} and \Cref{lemma:properties_u} (ii) we have for any $\boldsymbol{\varepsilon}$ as above
    \[
    u_{\boldsymbol{\mu}}(1, 1, 1 )= u_{\boldsymbol{\mu}}(\varepsilon_i, \varepsilon_j, \varepsilon_k) =  1.
    \]
    By \Cref{prop_main_N}, we have for any of the four choices of $\boldsymbol{\varepsilon}$, we have 
    \begin{align*}
              N_{\boldsymbol{\mu}, \boldsymbol{\varepsilon}}(X_i, X_j, X_k, X_{-1}) &= \frac{1}{512} \prod_{p>2} \left(1-\frac{1}{p}\right)^4\left(1+\frac{4}{p}\right)X_iX_jX_kX_{-1} + E ,
    \end{align*}
    with an error term 
    \[
    E \ll X_iX_jX_kX_{-1} \min\{X_i,X_j, X_k, X_{-1}\}^{-\Delta}
    \]
    for some $\Delta>0$. Hence it follows directly from  \eqref{eq:analytic_MPhi_new}
    \[
    M_{\Phi}(X_i, X_j, X_k,X_{-1}) =  \frac{1}{256} \prod_{p>2} \left(1-\frac{1}{p}\right)^4\left(1+\frac{4}{p}\right)X_iX_jX_kX_{-1} + E 
    \]
    with the same power saving error term $E$. 
\end{proof}

\begin{lemma} \label{lemma:local2_masses}
    We have 
    \[
     \sum_{\Phi \in \operatorname{Hom}_{\operatorname{cts}}(\Gamma_{\QQ_2}, Q_8)}1 = 224.
    \]
    Let $\operatorname{H}_2 :\operatorname{Hom}_{\operatorname{cts}}(\Gamma_{\QQ_2}, Q_8) \to \RR_{>0}$ be the height function $
        \Phi \mapsto \begin{cases}
            1 &\text{ if } \Phi \text{ unramified}, \\
            2 &\text{ if } \Phi \text{ ramified}.
        \end{cases}$\\ Then 
    \[
     \sum_{\Phi \in \operatorname{Hom}_{\operatorname{cts}}(\Gamma_{\QQ_2}, Q_8)} \frac{1}{H_2(\Phi)} = 116. 
    \]
\end{lemma} 
\begin{proof}
    We enumerate all fields $L/\QQ_2$ with $\operatorname{Gal}(L/\QQ_2)$ isomorphic to a subgroup of $H\subseteq Q_8$. For each subgroup $H\subseteq Q_8$, there are $|\operatorname{Aut}(H)|$ many fields such that $\operatorname{Gal}(L/\QQ_2) \simeq H$. We find a complete list of the necessary field extensions over $\QQ_2$ on \cite{lmfdb}, 
    \begin{table}[H]
        \centering
        \begin{tabular}{c|c|c}
            $H$ & $|\operatorname{Aut}(H)|$&$\#$ fields  \\
            \hline
             1&1&1\\
             $\langle-1\rangle$ & 1&7\\
             $\langle i \rangle, \langle j \rangle, \langle k \rangle$ & 2&12\\
             $Q_8$ & 24&6\\
        \end{tabular}
    \end{table}\noindent
    This gives a total of 
    \[
     \sum_{\Phi \in \operatorname{Hom}_{\operatorname{cts}}(\Gamma_{\QQ_2}, Q_8)}1 = 1+7+ 3\cdot 2\cdot 12 + 24\cdot 6 =224.
    \] 
    Weighted by $\frac{1}{H_2}$, the trivial extension, one of the seven quadratic extensions and one of the twelve quartic extensions are unramified, whereas all others are ramified. This yields 
    \[
      \sum_{\Phi \in \operatorname{Hom}_{\operatorname{cts}}(\Gamma_{\QQ_2}, Q_8)} \frac{1}{H_2(\Phi)} = 1+\left(1+\frac{6}{2}\right) +2\cdot 3\cdot\left(1+\frac{11}{2}\right) + 24\cdot \frac{6}{2} =116.
    \]
\end{proof}

\begin{proof}[Proof of \Cref{thm1}]
Let $\boldsymbol{X} = (X_i, X_j, X_k, X_{-1})\in \RR_{>0}^4$. Let
\[
M(\boldsymbol{X})= \#\{(M, \sigma): \sigma :\operatorname{Gal}(M/\QQ) \overset{\simeq}{\to} Q_8 : h_x(M, \sigma) \leq X_x \text{ for } x\in \{i,j,k, -1\} \}.
\]
Then 
\[
M(\boldsymbol{X}) = \sum_{\Phi \in \operatorname{Hom}_{\operatorname{cts}}(\Gamma_{\QQ_2}, Q_8)} M_{\Phi}(\boldsymbol{X})
\]
with $ M_{\Phi}(\boldsymbol{X})$ the quantity in \Cref{thm2_local}, defined in \eqref{eq:defMPhi}.  By \Cref{lemma:local2_masses}, there are a total of $224$ such homomorphisms $\Phi$. It follows from \Cref{thm2_local} that
\begin{align*}
    M(\boldsymbol{X}) = \frac{7}{8} \prod_{p>2} \left(1-\frac{1}{p}\right)^4\left(1+\frac{4}{p}\right)X_iX_jX_kX_{-1} + E 
\end{align*}
    for an error term 
    \[
    E \ll X_iX_jX_kX_{-1} \min\{X_i,X_j, X_k, X_{-1}\}^{-\Delta}
    \]
    for some $\Delta>0$.
\end{proof}

\begin{proof}[Proof of \Cref{corollary1}]
The result is an application of the hyperbola method as developed by Pieropan and Schindler \cite{pieropan-schindler}. Let
\[
f_{\Phi} : \NN^4 \to \RR 
\]
with 
\[
f_{\Phi}(y_i, y_j, y_k,y_{-1}) := \#\{ \varphi \in \operatorname{Epi}_{\operatorname{cts}}(\Gamma_\QQ, Q_8) : h_x(\varphi)=y_x \text{ for }x \in \{i,j,k,-1\}, \varphi_{|_{\Gamma_{\QQ_2}}}  = \Phi\}.
\]
\Cref{thm2_local} shows that $f_{\Phi}$ satisfies the Property $I$ from \cite[Sec. 4]{pieropan-schindler} with 
\[
C_f = \frac{1}{256} \prod_{p>2}\left(1-\frac{1}{p}\right)^4\left(1+\frac{4}{p}\right).
\]
Let 
\[
\mathcal{S}^{f_{\Phi}}(X) = \sum_{\substack{ (y_i, y_j, y_k, y_{-1})\in \NN^4 \\ y_iy_jy_ky_{-1} \leq X}} f_{\Phi}(y_i, y_j, y_k, y_{-1}).
\]
Let 
\[
H^{\operatorname{odd}}(\varphi) = h_i(\varphi)h_j(\varphi)h_k( \varphi)h_{-1}( \varphi),
\]
which is product of tamely ramified primes. Then
\[
\mathcal{S}^{f_{\Phi}}(X) =  \#\{  \varphi \in \operatorname{Epi}_{\operatorname{cts}}(\Gamma_\QQ, Q_8) : H^{\operatorname{odd}}(\varphi ) \leq X, \varphi_{|_{\Gamma_{\QQ_2}}}  = \Phi \} .
\]
 We follow the notation introduced in \cite{pieropan-schindler} and apply \cite[Thm. 1.1]{pieropan-schindler} with $\mathcal K =\{1\}$, $\alpha_x = 1$, $\varpi_x =1$ and $s=4$. Then $F$ is the hyperplane of dimension 3 given by 
\[
t_i+t_j+t_k+t_{-1} = 1
\]
and we have $a=1$ is attained. A short calculation shows that the constant $c_P$ in \cite[Eq. 4.3]{pieropan-schindler} is $\frac{1}{6}$. So by \cite[Thm. 1.1]{pieropan-schindler}, we have 
\[
\mathcal{S}^{f_\Phi}(X) = \frac{1}{1536}  \prod_{p>2} \left(1-\frac{1}{p}\right)^4 \left(1+\frac{4}{p}\right) X (\log X)^3 + O(X(\log X)^2( \log\log X)^4).
\]
Let $H(\varphi)$ denote the product of ramified primes, then
\[
H(\varphi) = H_2(\varphi_{|\Gamma_{\QQ_2}})\cdot  H^{\operatorname{odd}}(\varphi)
\]
where $H_2 : \operatorname{Hom}_{\operatorname{cts}}(\Gamma_{\QQ_2}, Q_8) \to \RR_{>0}$ is given by 
\[
    \Phi \mapsto \begin{cases}
            1 &\text{ if } \Phi \text{ unramified}, \\
            2 &\text{ if } \Phi \text{ ramified}.
        \end{cases}
\]
 Hence 
\begin{align*}
 \#\{  \varphi \in \operatorname{Epi}_{\operatorname{cts}}(\Gamma_\QQ, Q_8) :  H(\varphi ) \leq X  \}  &= \sum_{\Phi \in \operatorname{Hom}_{\operatorname{cts}}(\Gamma_{\QQ_2}, Q_8)} \mathcal{S}^{f_\Phi}(X/H_2(\Phi)) \\
 &= c X (\log X)^3 
  +  O(X(\log X)^2 (\log\log X)^4)    
\end{align*}
with 
\[
c = \left(\sum_{\Phi \in \operatorname{Hom}_{\operatorname{cts} }(\Gamma_{\QQ_2}, Q_8)} \frac{1}{H_2(\Phi)} \right) \cdot   \frac{1}{1536}  \prod_{p>2} \left(1-\frac{1}{p}\right)^4 \left(1+\frac{4}{p}\right)
\]
By \Cref{lemma:local2_masses}
\[
 \sum_{\Phi \in \operatorname{Hom}_{\operatorname{cts} }(\Gamma_{\QQ_2}, Q_8)} \frac{1}{H_2(\Phi)} = 116
\]
from which the result follows. 
\end{proof}
\section{The leading constant}
In recent work, Loughran and Santens \cite{loughran_santens2025malle}, \cite{loughran2026leadingconstantmallesconjecture} give conjectures and interpretations for the leading constants that appear in the asymptotics related to Malle's conjecture. We show briefly how our constants from \Cref{thm1} and \Cref{corollary1} fit into this framework. Recall from \Cref{sec:notation} the definition of $BQ_8(\QQ)$ as the groupoid with objects continuous homomorphisms from $\Gamma_\QQ \to Q_8$, and morphisms given by the conjugacy action of $Q_8$. Moreover, recall $BQ_8[\QQ]$ the set of isomorphism classes of objects in $BQ_8(\QQ)$.   \\
Let $\mathcal{C}$ be the set of cyclic subgroups of $Q_8$ up to conjugation, i.e. $\mathcal{C}$ has five classes that are represented by $1, -1, i, j, k$. We further write $\mathcal{C}^* $ for the collection of non-trivial classes. We associate a family of weight functions $w_c : \mathcal{C}^* \to \RR_{\geq 0}$ for $c\in \{-1, i, j, k\}$ by setting 
\[
w_c(c') =\begin{cases}
    1 & \;\; \text{if } c'=c,\\
    0 &\; \; \text{if } c'\neq c. 
\end{cases}
\] To a weight function, we associate a collection of local height functions: For each finite place $p\neq 2$ of $\QQ$ we set  $H_{c, p} : BQ_8(\QQ_p) \to \RR_{>0}$ for all $c\in \mathcal{C}^*$, by setting 
\[
H_{c, p}(\varphi) = p^{w_c(\rho(\varphi))},
\]
where $\rho: BQ_8(\QQ_p)\to \mathcal{C}$ is the ramification type introduced in \cite[Sec. 4.2]{loughran2026leadingconstantmallesconjecture}, which in the case of extensions over $\QQ$ is given by taking the inertia subgroup up to conjugation as in \Cref{def:multi}.  For the place $p=2$ we associate the constant height functions $H_{2,c} (\varphi) =1$. This collection of local heights amounts to four global height functions $H_c: BQ_8(\QQ) \to \RR_{>0}$ by defining 
\[
H_c (\varphi)= \prod_{p<\infty} H_{c, p}(\varphi).
\]
Here, we view $\varphi$ under its diagonal image $BQ_8(\QQ) \to \prod_{v<\infty }BQ_8(\QQ_v)$. \\
The prediction of the leading constant for the multi-height setting is given in \cite[Conj. 9.1]{loughran2026leadingconstantmallesconjecture}. In particular, for $G=Q_8, k=\QQ$ and the height functions defined above, the conjecture reads 
\begin{align} \label{eq:conj1_ls}
\frac{1}{|Z(Q_8)|}\# \{\varphi \in BQ_8[\QQ] \setminus \Omega: h_c(\varphi) \leq X_c \text{ for all }c\in \mathcal{C}^* \} \sim \alpha^* \beta \ \tau X_iX_jX_kX_{-1}
\end{align}
for parameters $X_i, X_j, X_k, X_{-1} \to \infty$ and constants $\alpha^*, \beta, \tau$ that we unpack in the following. Here, $\Omega \subset BQ_8[\QQ]$ is a thin set which includes all classes of non-surjective homomorphisms. 
Set
\[
a = (\min_{c\in \mathcal{C}^*} w_c )^{-1} ,\qquad b= \#\{c\in \mathcal{C}^* : w_c = a^{-1}\}.
\]
In our case  
\[
a=1 , \qquad b=4.
\]
Then 
\[
\alpha^* := \frac{a^{b-1}}{\# \widehat{Q_8(\QQ)}} = \frac{1}{4}
\]
as there are four one-dimensional rational-valued characters of $Q_8$. The constant $\beta$ is the size of the partially unramified Brauer group of $BQ_8$. Following the notation introduced in \cite{loughran_santens2025malle} we see that the algebraic part of the Brauer group is trivial, see \cite[Lemma 6.32]{loughran_santens2025malle}. Further the transcendental part of the Brauer group can be identified with a certain subgroup of the dual of the Schur multiplier of $Q_8$, see \cite[Lemma 6.36]{loughran_santens2025malle}. As the Schur multiplier of $Q_8$ is trivial, the transcendental part of the Brauer group is trivial. This gives $\beta = 1$.
Finally  
\[
\tau = \tau_{\infty} \prod_{p}(1-1/p)^4 \tau_p
\]
is the global Tamagawa number, with local Tamagawa numbers 
\[
\tau_\infty = \frac{\#Q_8[2]}{\#Q_8} = \frac{2}{8} , \; \tau_2 = \frac{224}{8} = 28, \tau_{p} = (1+4/p) \text{ \; for $p>2$ }.
\]
Here we use for $p\neq 2, \infty$ the mass formula \cite[Cor 8.11]{loughran_santens2025malle}. For the infinite place we use  \cite[Eq 8.3]{loughran_santens2025malle}. Finally for $p=2$ we use \Cref{lemma:local2_masses} together with \cite[Lemma 3.18]{loughran2026leadingconstantmallesconjecture}. Hence \eqref{eq:conj1_ls} reads
\begin{align*}
    &\# \{\varphi \in BQ_8[\QQ] \setminus \Omega: h_c(\varphi) \leq X_c \text{ for all }c\in \mathcal{C}^* \} \\
    &\qquad\sim \underbrace{2\;  \cdot  \;\frac{1}{4}\;\cdot \;\frac{2}{8}\; \cdot \;\frac{28}{16}}
_{=  |Z(Q_8)|\ \cdot \  \alpha^*\ \cdot \ \tau_\infty \ \cdot \ (1-1/2)^4\tau_2} \prod_{p>2} \left(1-\frac{1}{p}\right)^4\left(1+\frac{4}{p}\right)  X_iX_jX_kX_{-1} \\
    &\qquad \sim \frac{7}{32} \prod_{p>2}\left(1-\frac{1}{p}\right)^4\left(1+\frac{4}{p}\right)  X_iX_jX_kX_{-1} .
\end{align*}
Recall from \Cref{sec:notation} that each (surjective) isomorphism class $\varphi \in BQ_8[\QQ]$, corresponds to exactly 4 homomorphism, equivalently labeled $Q_8$-extensions $(M, \sigma)$, that we count in \Cref{thm1}. Hence \Cref{thm1} confirms \cite[Conj. 9.1]{loughran2026leadingconstantmallesconjecture} for the multi-height counting of $Q_8$-extensions over $\QQ$, with $\Omega$ the thin set containing non-surjective homomorphisms. 

The same approach shows that \Cref{corollary1} verifies \cite[Conj. 7.1]{loughran2026leadingconstantmallesconjecture} for counting by the height function $H$ denoting the product of ramified primes. Indeed, in this case one defines the weight function $w_c = 1$ for all $c\in \mathcal{C}^*$. We have $b=4$ and obtain that $\alpha^*, \beta$ are the same as in the multi-height setting. The local Tamagawa calculation also stay the same for all tame places, as well as the place at infinity. For $p=2$, we use \Cref{lemma:local2_masses} and \cite[Lemma 3.18]{loughran2026leadingconstantmallesconjecture} to obtain 
\[
\tau_2 = \frac{116}{8}.
\]
Hence \cite[Conj. 7.1]{loughran2026leadingconstantmallesconjecture} reads 
\begin{align*}
    &\# \{\varphi \in BQ_8[\QQ] \setminus \Omega: H(\varphi) \leq X \} \\
    & \qquad \qquad \sim \underbrace{2\;  \cdot  \frac{1}{6}\cdot \;\frac{1}{4}\;\cdot \;\frac{2}{8}\; \cdot \;\frac{29}{32}}
_{=  |Z(Q_8)|\ \cdot \frac{1}{(b-1)!} \  \cdot \   \alpha^*\ \cdot \ \tau_\infty \ \cdot \ (1-1/2)^4\tau_2} \prod_{p>2} \left(1-\frac{1}{p}\right)^4\left(1+\frac{4}{p}\right)  X (\log X)^3 \\
    &\qquad \qquad  \sim \frac{29}{1536} \prod_{p>2}\left(1-\frac{1}{p}\right)^4\left(1+\frac{4}{p}\right)  X (\log X)^3 .
\end{align*}
Choosing $\Omega$ as the thin set containing non-surjective homomorphisms,  \Cref{corollary1} confirms this conjecture, as every surjective isomorphism class $\varphi \in BQ_8[\QQ]$ corresponds to $4$ epimorphisms $\varphi \in \operatorname{Epi}_{\operatorname{cts}}(\Gamma_{\QQ}, Q_8)$, see \Cref{sec:notation}.

\printbibliography[
title={References}
]

@article {munsch_character_sums,
    AUTHOR = {Munsch, Marc},
     TITLE = {Character sums over squarefree and squarefull numbers},
   JOURNAL = {Arch. Math. (Basel)},
  FJOURNAL = {Archiv der Mathematik},
    VOLUME = {102},
      YEAR = {2014},
    NUMBER = {6},
     PAGES = {555--563},
}

@article {friedlander-iwaniec,
    AUTHOR = {Friedlander, John and Iwaniec, Henryk},
     TITLE = {Ternary quadratic forms with rational zeros},
   JOURNAL = {J. Th\'eor. Nombres Bordeaux},
  FJOURNAL = {Journal de Th\'eorie des Nombres de Bordeaux},
    VOLUME = {22},
      YEAR = {2010},
    NUMBER = {1},
     PAGES = {97--113},
}

@misc{loughran_santens2025malle,
      title={Malle's conjecture and Brauer groups of stacks}, 
      author={Daniel Loughran and Tim Santens},
      year={2025},
      eprint={2412.04196},
      archivePrefix={arXiv},
      primaryClass={math.NT},
}

@misc{lmfdb,
  shorthand    = {LMFDB},
  author       = {The {LMFDB Collaboration}},
  title        = {The {L}-functions and modular forms database},
  howpublished = {\url{https://www.lmfdb.org}},
  year         = {2026},
  note         = {[Online; accessed 4 August 2026]},
}

@article {Koymans-Pagano-nilpotent,
    AUTHOR = {Koymans, Peter and Pagano, Carlo},
     TITLE = {On {M}alle's conjecture for nilpotent groups},
   JOURNAL = {Trans. Amer. Math. Soc. Ser. B},
  FJOURNAL = {Transactions of the American Mathematical Society. Series B},
    VOLUME = {10},
      YEAR = {2023},
     PAGES = {310--354},
}

@misc{koymans-pagano-archive,
      title={On Malle's conjecture for nilpotent groups, I}, 
      author={Peter Koymans and Carlo Pagano},
      year={2021},
      eprint={2103.17223},
      archivePrefix={arXiv},
      primaryClass={math.NT},
}

@misc{loughran2026leadingconstantmallesconjecture,
      title={The leading constant in Malle's conjecture}, 
      author={Daniel Loughran and Tim Santens},
      year={2026},
      eprint={2606.04983},
      archivePrefix={arXiv},
      primaryClass={math.NT},
}

@misc{hansen-zanoli,
      title={Counting $D_4$-field extensions by multi-invariants}, 
      author={Willem Hansen and Anna Zanoli},
      year={2025},
      eprint={2507.12342},
      archivePrefix={arXiv},
      primaryClass={math.NT},
}

@article {fouvry,
    AUTHOR = {Fouvry, \'Etienne and Luca, Florian and Pappalardi, Francesco
              and Shparlinski, Igor E.},
     TITLE = {Counting dihedral and quaternionic extensions},
   JOURNAL = {Trans. Amer. Math. Soc.},
  FJOURNAL = {Transactions of the American Mathematical Society},
    VOLUME = {363},
      YEAR = {2011},
    NUMBER = {6},
     PAGES = {3233--3253},
}

@article {pieropan-schindler,
    AUTHOR = {Pieropan, Marta and Schindler, Damaris},
     TITLE = {Hyperbola method on toric varieties},
   JOURNAL = {J. \'Ec. polytech. Math.},
  FJOURNAL = {Journal de l'\'Ecole polytechnique. Math\'ematiques},
    VOLUME = {11},
      YEAR = {2024},
     PAGES = {107--157},
}

@article {peyre95,
    AUTHOR = {Peyre, Emmanuel},
     TITLE = {Hauteurs et mesures de {T}amagawa sur les vari\'et\'es de
              {F}ano},
   JOURNAL = {Duke Math. J.},
  FJOURNAL = {Duke Mathematical Journal},
    VOLUME = {79},
      YEAR = {1995},
    NUMBER = {1},
     PAGES = {101--218},
}

@article {shankar2024asymptoticscubicfieldsordered,
    AUTHOR = {Shankar, Arul and Thorne, Frank},
     TITLE = {On the asymptotics of cubic fields ordered by general
              invariants},
   JOURNAL = {Comment. Math. Helv.},
  FJOURNAL = {Commentarii Mathematici Helvetici. A Journal of the Swiss
              Mathematical Society},
    VOLUME = {99},
      YEAR = {2024},
    NUMBER = {4},
     PAGES = {769--797},
}

@article {klueners2004counterexamplemallesconjecture,
    AUTHOR = {Kl\"uners, J\"urgen},
     TITLE = {A counterexample to {M}alle's conjecture on the asymptotics of
              discriminants},
   JOURNAL = {C. R. Math. Acad. Sci. Paris},
  FJOURNAL = {Comptes Rendus Math\'ematique. Acad\'emie des Sciences. Paris},
    VOLUME = {340},
      YEAR = {2005},
    NUMBER = {6},
     PAGES = {411--414},
}

@article {ellenberg_satriano_zureick-brown,
    AUTHOR = {Ellenberg, Jordan S. and Satriano, Matthew and Zureick-Brown,
              David},
     TITLE = {Heights on stacks and a generalized {B}atyrev-{M}anin-{M}alle
              conjecture},
   JOURNAL = {Forum Math. Sigma},
  FJOURNAL = {Forum of Mathematics. Sigma},
    VOLUME = {11},
      YEAR = {2023},
     PAGES = {Paper No. e14, 54},
}

@article{DardaYasuda2024,
  author    = {Ratko Darda and Takehiko Yasuda},
  title     = {The Batyrev–Manin conjecture for DM stacks},
   JOURNAL = {J. Eur. Math. Soc. (JEMS)},
  FJOURNAL = {Journal of the European Mathematical Society (JEMS)},
    VOLUME = {28},
      YEAR = {2026},
    NUMBER = {6},
     PAGES = {2351--2400},
}

@misc{gundlach2022mallesconjecturemultipleinvariants,
      title={Malle's conjecture with multiple invariants}, 
      author={Fabian Gundlach},
      year={2022},
      eprint={2211.16698},
      archivePrefix={arXiv},
}

@article {wright_abelian,
    AUTHOR = {Wright, David J.},
     TITLE = {Distribution of discriminants of abelian extensions},
   JOURNAL = {Proc. London Math. Soc. (3)},
  FJOURNAL = {Proceedings of the London Mathematical Society. Third Series},
    VOLUME = {58},
      YEAR = {1989},
    NUMBER = {1},
     PAGES = {17--50},
}

@article{malle2002,
title = {On the Distribution of Galois Groups},
journal = {Journal of Number Theory},
volume = {92},
number = {2},
pages = {315--329},
year = {2002},
author = {Gunter Malle},

}

@article {kiming,
    AUTHOR = {Kiming, Ian},
     TITLE = {Explicit classifications of some {$2$}-extensions of a field
              of characteristic different from {$2$}},
   JOURNAL = {Canad. J. Math.},
  FJOURNAL = {Canadian Journal of Mathematics. Journal Canadien de
              Math\'ematiques},
    VOLUME = {42},
      YEAR = {1990},
    NUMBER = {5},
     PAGES = {825--855},
}

@incollection {ellenberg_venkatesh,
    AUTHOR = {Ellenberg, Jordan S. and Venkatesh, Akshay},
     TITLE = {Counting extensions of function fields with bounded
              discriminant and specified {G}alois group},
 BOOKTITLE = {Geometric methods in algebra and number theory},
    SERIES = {Progr. Math.},
    VOLUME = {235},
     PAGES = {151--168},
 PUBLISHER = {Birkh\"auser Boston, Boston, MA},
      YEAR = {2005},
}

@article {malle2004,
    AUTHOR = {Malle, Gunter},
     TITLE = {On the distribution of {G}alois groups. {II}},
   JOURNAL = {Experiment. Math.},
  FJOURNAL = {Experimental Mathematics},
    VOLUME = {13},
      YEAR = {2004},
    NUMBER = {2},
     PAGES = {129--135},
   
}

@incollection {peyre-multi-heights,
    AUTHOR = {Peyre, Emmanuel},
     TITLE = {Chapter {V}: {B}eyond heights: slopes and distribution of
              rational points},
 BOOKTITLE = {Arakelov geometry and {D}iophantine applications},
    SERIES = {Lecture Notes in Math.},
    VOLUME = {2276},
     PAGES = {215--279},
 PUBLISHER = {Springer, Cham},
      YEAR = {2021}
}

@article {d4_conductor,
    AUTHOR = {Altu\u g, S. Ali and Shankar, Arul and Varma, Ila and Wilson,
              Kevin H.},
     TITLE = {The number of {$D_4$}-fields ordered by conductor},
   JOURNAL = {J. Eur. Math. Soc. (JEMS)},
  FJOURNAL = {Journal of the European Mathematical Society (JEMS)},
    VOLUME = {23},
      YEAR = {2021},
    NUMBER = {8},
     PAGES = {2733--2785},
}

@article {klueners_malle,
    AUTHOR = {Kl\"uners, J\"urgen and Malle, Gunter},
     TITLE = {Counting nilpotent {G}alois extensions},
   JOURNAL = {J. Reine Angew. Math.},
  FJOURNAL = {Journal f\"ur die Reine und Angewandte Mathematik. [Crelle's
              Journal]},
    VOLUME = {572},
      YEAR = {2004},
     PAGES = {1--26},
}

\end{document}